\documentclass[11pt]{amsart}
\usepackage{mathmacros,hyperref}
\usepackage[tmargin=1in,lmargin=1in,bmargin=1in,rmargin=1in]{geometry}
\numberwithin{equation}{section}

\begin{document}
	
\newif\ifdraft

	\title{The quintic NLS on perturbations of $\mf{R}^3$}
	\author{Casey Jao}
	\maketitle
	
	\begin{abstract}
		Consider the defocusing quintic nonlinear Schr\"{o}dinger equation on 
		$\mf{R}^3$ 
		with initial data in the energy space. This problem is 
		``energy-critical" in view 
		of a certain scale-invariance, which is a main 
		source of difficulty 
		in the 
		analysis of this equation. It is a nontrivial fact that all 
		finite-energy solutions 
		scatter to linear solutions. We show that this remains true 
		under small compact deformations of the Euclidean metric. Our main new 
		ingredient is a long-time microlocal weak dispersive estimate that 
		accounts for the refocusing of geodesics.
	\end{abstract}
	
	\section{Introduction}
	
	Let $g$ be a smooth Riemannian metric on $\mf{R}^3$. We consider the 
	large-data 
	Cauchy 
	problem for the defocusing nonlinear Schr\"{o}dinger equation
	\begin{align}
	\label{ch4:e:main_eq}
	i\partial_t u + \Delta_g u = |u|^4 u, \ u(0, x) = u_0(x) \in \dot{H}^1,
	\end{align}
	where $\Delta_g$
	is the 
	Laplace-Beltrami operator. More precise 
	assumptions on 
	$g$ shall be prescribed shortly. 
	
	This equation admits a conserved energy
	\begin{align}
	\label{ch4:e:energy}
	E(u) = \int_{\mf{R}^3} \frac{1}{2} g^{jk}\partial_j u \overline{\partial_j 
	u} + 
	\frac{1}{6} |u|^6 \, dg,
	\end{align}
	where $dg = \sqrt{|g|} dx$ is the Riemannian measure. 
	
	If $g = \delta$ is 
	the standard Euclidean metric, one recovers the well-known
	energy-critical NLS. A key feature of that 
	equation---and 
	a major source of analytical headaches---is the scaling symmetry 
	$u_\lambda(t, x) = 
	\lambda^{-1/2} u(\lambda^{-2}t, \lambda^{-1}x)$. As the energy is also 
	invariant under this rescaling, conservation of energy alone does not rule 
	out 
	the possibility that some solutions may concentrate at a point and blow 
	up in finite time.  It is a difficult theorem
	that all finite-energy
	solutions to that equation  
	scatter~\cite{bourgain_nls_radial_gwp,ckstt}.
	
	Although the exact 
	scaling 
	symmetry no longer holds for general $g$, it
	reemerges at small length scales in the sense that solutions highly 
	concentrated 
	near a 
	point $x_0$ will evolve, for short times, approximately as though $g$ were 
	the constant 
	metric $g(x_0)$. Despite this parallel, however, 
	equation~\eqref{ch4:e:main_eq} 
	is not a simple 
	perturbation of the Euclidean energy-critical equation. Indeed, disturbing 
	the 
	highest order 
	terms may destroy fundamental linear
	smoothing and decay estimates. This breakdown is linked to 
	the geometry of the 
	geodesic flow.
	

	For a general metric $g$, the linear local smoothing estimate $L^2 \to 
	L^2 
	H^{1/2}_{loc}$ is known to fail in the presence of trapping~\cite{Doi1996}. 
	Also, on a curved background, multiple geodesics emanating from a point may 
	converge at 
	another point. Linear solutions exhibit weaker decay with such 
	refocusing; in particular, by the parametrix construction of 
	Hassell-Wunsch~\cite{HassellWunsch2005},   
	the Euclidean dispersive estimate
	\begin{align}
		\label{ch4:e:intro_disp_est}
	\|e^{it\Delta} \|_{L^1(\mf{R}^d) \to L^\infty(\mf{R}^d)} 
	\lesssim |t|^{-d/2}
	\end{align}
	necessarily fails whenever the 
	metric admits conjugate points. In general one can only 
	recover a 
	frequency-localized version which holds at most for times
	inversely proportional to frequency; see~\cite{BurqGerardTzvetkov2004}. The 
	time 
	window stops the flow well before refocusing of 
	geodesics can occur. 
	
	While trapping does not occur if $g$ is sufficiently close to flat, 
	arbitrarily
	small perturbations of the flat metric may cause rays to refocus. 
	Thus~\eqref{ch4:e:intro_disp_est} typically fails on curved backgrounds. This has 
	substantial 
	implications for both the linear and nonlinear analysis.
	
	The standard abstract approach to linear Strichartz estimates combines the 
	dispersive estimate with a $TT^*$ argument~\cite{keel-tao}. This method is 
	not directly applicable where the dispersive estimate is not available.
	Nonetheless, 
	lossless Strichartz inequalities have been obtained for curved 
	backgrounds, starting with the influential work of Staffilani and 
	Tataru~\cite{StaffilaniTataru2002} and generalized substantially since
	\cite{RodnianskiTao2007,HassellTaoWunsch2006,Tataru2008,BoucletTzvetkov2008,MarzuolaMetcalfeTataru2008}.
	The basic strategy in these papers is to exploit microlocal versions of the 
	dispersive estimate through suitable parametrices and to control the errors 
	using 
	local smoothing, which holds in greater generality compared to the 
	dispersive 
	estimate.
	
	Linear dispersion also plays a key role 
	in the study of nonlinear solutions, in particular, when trying to 
	control
	highly concentrated nonlinear profiles that arise as potential obstructions 
	to global 
	existence, and for proving the decoupling of nonlinear profiles. There are by now several examples of such an analysis; see for 
	example~\cite{IonescuPausaderStaffilani2012}, 
	\cite{kvz_exterior_convex_obstacle}, or 
	\cite{me_quadratic_potential}. We briefly 
	recall the main idea from the latter reference, which discusses the 
	energy-critical
	NLS with the Schr\"{o}dinger operator for a quantum harmonic oscillator. By 
	the Mehler formula, the linear propagator obeys the dispersive estimate 
	locally in time.
	
	Suppose $u_n$ is 
	a sequence of solutions to the defocusing quintic harmonic 
	oscillator on $\mf{R}^3$ 
	with initial 
	data $u_n(0) = \lambda_n^{-1/2} \phi(\lambda_n^{-1} \cdot)$ for some 
	$\lambda_n \to 0$ and some compactly supported $\phi$. For short times 
	(more 
	precisely, when $|t| \le T \lambda_n^2$ for any $T > 0$), the harmonic 
	oscillator solution $u_n$ perceives the potential as essentially 
	constant 
	and is well-approximated by the solution $\tilde{u}_n$ to the 
	Euclidean 
	energy-critical equation with the same initial data. Using as a black box 
	the theorem that Euclidean solutions exist globally and scatter, one 
	deduces 
	via stability theory that $u_n$ is well-behaved for $t \le O(\lambda_n^2)$.
	
	For $t\ge T \lambda_n^2$, the (local in time) dispersive 
	estimate for the harmonic oscillator and the scattering of Euclidean 
	solutions 
	ensure that for large $T$ and small $\lambda_n$, the 
	nonlinearity $|u_n|^4 u_n$ is a negligible 
	perturbation of the linear harmonic oscillator. That is, for such $t$, 
	$u_n$ 
	evolves 
	essentially according to the linear flow applied to $u(T\lambda_n^2)$, 
	which is 
	perfectly well behaved. Thus, linear decay allows one to control 
	concentrated nonlinear solutions for times when the Euclidean 
	approximation no longer holds.
	
	We investigate the situation where $g$ 
	coincides with the flat metric outside the unit ball and all geodesics 
	escape to infinity. This is the simplest nontrivial generalization of the 
	Euclidean metric and is a natural counterpart to the scenario considered 
	recently by Killip, Visan, and Zhang~\cite{kvz_exterior_convex_obstacle}, 
	who 
	proved scattering for the analogue of equation~\eqref{ch4:e:main_eq} in the 
	exterior of a hard convex obstacle. We prove
	
	\begin{thm}
		\label{ch4:t:main_thm}
		Let $g$ be a smooth, nontrapping metric on $\mf{R}^3$ which coincides 
		with 
		the 
		Euclidean 
		metric outside the unit ball. For 
		any $u_0 \in \dot{H}^1$, there is a unique global solution 
		to~\eqref{ch4:e:main_eq}. Moreover, there exists $\varepsilon > 0$ such 
		that if $\|g-\delta\|_{C^3} 
		\le 
		\varepsilon$ then the 
		solutions obey global spacetime bounds
		\begin{align*}
		\|u\|_{L^{10}_{t,x}(\mf{R} \times \mf{R}^3)} \le C(E(u_0)).
		\end{align*}
	\end{thm}
	 
	
	We use the Kenig-Merle concentration compactness and rigidity 
	framework~\cite{keraani_compactness_defect,kenig-merle_focusing_nls}, an 
	evolution of the earlier induction on energy method
	of~\cite{bourgain_nls_radial_gwp,ckstt}. In particular, we follow quite 
	closely the mold 
	of~\cite{kvz_exterior_convex_obstacle}.  
	Assuming that the scattering fails, we show that there must exist a 
	global-in-time
	blowup solution $u_c$ with minimal energy among all counterexamples to the 
	theorem. In view of this minimality, $u_c$ is also shown to be 
	almost-periodic 
	in the sense that $u(t)$ is trapped in some compact subset of $\dot{H}^1$. 
	In these arguments, the asymptotic behavior of the Euclidean NLS (which we 
	use as a black box) plays a key role. 
	However, under the smallness assumption on the metric, a Morawetz 
	inequality 
	will imply that solutions to equation~\eqref{ch4:e:main_eq} can never be 
	almost-periodic. 
	The smallness condition for scattering is probably 
	artificial, but we do not see at this time how to dispense with it.
	
	The heart of the matter is how to overcome the 
	reduced linear dispersion, which is the main obstacle to the 
	linear 
	and nonlinear profile decompositions. In Section~\ref{ch4:s:extinction}, we 
	prove a weak analogue of the usual dispersive estimate which nonetheless 
	suffices for our purposes. This is a long-time variant of 
	the 
	Burq-Gerard-Tzvetkov dispersion 
	estimate~\cite{BurqGerardTzvetkov2004} where we track the 
	microlocalized Schr\"{o}dinger flow on timescales that permit refocusing.
	
	
	Several recent works have exploited analogous weak dispersive estimates to 
	study energy-critical NLS in non-Euclidean geometries. The dispersion 
	results from different mechanisms in each case.
	En route to proving global wellposedness for the quintic NLS on $\mf{T}^3$, 
	Ionescu-Pausader introduce an ``extinction lemma" 
	\cite[Lemma~4.2]{IonescuPausader2012} to control concentrated nonlinear 
	profiles at times beyond the ``Euclidean window". Afterwards, 
	Pausader-Tzvetkov-Wang~\cite{PausaderTzvetkovWang2014} obtained the 
	analogous result on 
	$\mf{S}^3$, also relying crucially on an extinction lemma. The arguments 
	there lean on the special structure of the underlying manifold, 
	using for 
	instance 
	Fourier analysis on the torus (which, when combined with number theoretic 
	arguments, yield good bounds on the Schr\"{o}dinger propagator) or the 
	concentration properties of spherical harmonics.
	
	In a different vein, Killip-Visan-Zhang~\cite{kvz_exterior_convex_obstacle} 
	also obtained an extinction lemma 
	in 
	the exterior of a convex obstacle. The geodesics in that domain are 
	broken straight lines. To study the linear evolution 
	of a 
	profile 
	concentrating near the obstacle, they construct a gaussian wavepacket 
	parametrix and carefully study how the wavepackets reflect off the 
	obstacle. 
	The 
	essential 
	geometric fact in their favor is that due to the convexity assumption, any 
	two rays diverge after 
	reflecting off the obstacle.
	
	When the hard obstacle is replaced by a lens, refracted rays can certainly 
	refocus. However, some decay still occurs for a different reason. 
	By the uncertainty principle, a solution which is 
	initially highly 
	concentrated in space must be broadly distributed in momentum (frequency). 
	Thus, 
	it will spread out along geodesics as the slower parts lag 
	behind. This is an observation of D. Tataru communicated to the author by R. Killip and M. Visan. We 
	make this heuristic precise in Section~\ref{ch4:s:extinction} by 
	building a 
	wavepacket 
	parametrix and studying the geodesic flow.

	\textbf{Outline of paper}. Section~\ref{ch4:s:preliminaries} collects 
	technical points concerning Sobolev spaces and some linear theory. From the 
	linear estimates it is a standard matter to obtain the perturbative theory, 
	and 
	we merely state the main results. 
	
	Sections \ref{ch4:s:propagator_convergence} and 
	\ref{ch4:s:extinction} are most important to this paper. In 
	Section~\ref{ch4:s:propagator_convergence}, we study linear 
	solutions  in various situations
	where the 
	variation in the metric is intuitively negligible (for 
	instance, when considering initial data supported far from the origin). The 
	most 
	interesting case is when the solution is initially concentrated near 
	the origin, where it experiences nontrivial interaction with the curvature. 
	This section relies on an extinction lemma which 
	is the subject of Section~\ref{ch4:s:extinction}.

	With those considerations out of the way, we construct the linear 
	profile decomposition in Sections~\ref{ch4:s:lpd}. We also show in 
	Section~\ref{ch4:s:euclidean_profiles} that highly concentrated nonlinear 
	profiles are well-behaved; here the extinction lemma and the existing 
	scattering result for the Euclidean quintic equation both play a critical 
	role. 
	
	In Section~\ref{ch4:s:nlpd}, we use a nonlinear profile decomposition and 
	induction on energy to reduce Theorem~\ref{ch4:t:main_thm} to considering 
	almost-periodic minimal-energy counterexamples. This will already imply 
	global 
	wellposedness. Some care is needed to control the interaction between 
	linear 
	and nonlinear profiles; see the discussion preceding 
	Lemma~\ref{ch4:l:loc_smoothing}.
	
	Finally, in Section~\ref{ch4:s:scattering} we prove scattering 
	under the smallness assumption via a Bourgain-Morawetz inequality.
	
	In the appendix, we use the ideas from
Section~\ref{ch4:s:extinction} to give a small refinement to the 
Burq-Gerard-Tzvetkov semiclassical dispersive estimate which may be of independent interest.
	

	\textbf{Acknowledgments}. The author wishes to thank Rowan Killip and 
	Monica 
	Visan for many helpful conversations. This project was partially supported 
	by a UCLA Dissertation Year Fellowship. Part of this work 
	was 
	completed at MSRI during
	the Fall 2015 program on deterministic and stochastic PDE.

	\section{Preliminaries}
	\label{ch4:s:preliminaries}

	\subsection{Sobolev spaces}
	
	The energy space $\dot{H}^1 = \dot{H}^1(g)$ is defined as the completion of 
	test functions 
	$C^\infty_0(\mf{R}^3)$ 
	with respect to the 
	quadratic form
	\begin{align*}
	\|u\|_{ \dot{H}^1 }^2 = \int_{\mf{R}^3}|du|^2_g \, dg(x) = \int_{\mf{R}^3} 
	g^{jk} \partial_j u \overline{\partial_k u} \, dg (x).
	\end{align*}
As $\| u\|_{\dot{H}^1(g)} \sim \|
		(-\Delta_\delta)^{1/2}u\|_{L^2(dx)} = \| u\|_{\dot{H}^1(\delta)}$,
	where $\dot{H}^1(\delta)$ is the Euclidean homogeneous Sobolev space,
	the spaces $\dot{H}^1(g)$ and $\dot{H}^1(\delta)$ are equal as sets 
	and have equivalent inner products. We make this distinction because we 
	shall occasionally use the fact that 
	$\Delta_g$ is self-adjoint with respect to the inner product for 
	$\dot{H}^1(g)$; when the difference is irrelevant we just
	write $\dot{H}^1(\mf{R}^3)$ or just $\dot{H}^1$. 
	
	
	For $1 < p < \infty$, define the homogeneous Sobolev spaces $\dot{H}^{1, 
		p}(\delta)$ and $\dot{H}^{1,p}(g)$ as the completion of $C^\infty_0$ 
		under 
	the 
	norms
	\begin{align}
	\label{ch4:e:homogeneous_sob_space}
	\| u\|_{\dot{H}^{1, p}(\delta)} := \| (-\Delta_{\delta})^{1/2} u 
	\|_{L^p}, 
	\quad \| u 
	\|_{\dot{H}^{1, p}(g)} := \| (-\Delta_g)^{1/2} u\|_{L^p}.
	\end{align}
	As noted in the introduction, these two definitions coincide when $p = 2$. 
	Less trivially, these norms are equivalent for all $1 
	< p < \infty$. This is a consequence of the following boundedness result 
	for 
	the Riesz 
	transform $d (-\Delta_g)^{-1/2}$ on asymptotically Euclidean manifolds.
	
	
	\begin{prop}[{\cite[Remark 5.2]{Carron2006}}]
		\label{ch4:p:riesz_transform_bdd}
		Let $(M, g)$ be a Riemannian manifold such that for some $R > 0$, $M 
		\setminus 
		B(0, R) $ is Euclidean. Then the Riesz transform $d (-\Delta_g)^{-1/2}$ 
		is 
		bounded from $L^p(M)$ to $L^p(M; T^*M)$ for all $1 < p < \infty$.
	\end{prop}
	By a well-known duality argument (see for example \cite[Section 
	2.1]{CoulhonDuong2003}), this implies the reverse inequality whose proof we 
	give for completeness:
	\begin{cor}
		\[
		\| (-\Delta_g)^{1/2} u\|_{L^p} \lesssim_p \| d u\|_{L^p}, \ \forall u  
		\in C^\infty_0, \ 1 < p < \infty.
		\]
	\end{cor}
	\begin{proof}
		By duality, it suffices to show
		\begin{align*}
		|\langle (-\Delta_g)^{1/2} u, v \rangle | \lesssim \| d u\|_{L^p} \| \| 
		v\|_{L^{p'}}.
		\end{align*}
		Then
		\begin{align*}
		\langle (-\Delta_g)^{1/2}u, v \rangle &= \langle u, (-\Delta_g)^{1/2} v 
		\rangle 
		= \langle u, (-\Delta_g) (-\Delta_g)^{-1/2}v \rangle \\
		&= \langle d u, 
		d(-\Delta_g)^{-1/2} v \rangle
		\lesssim \| du \|_{L^p} \| v\|_{L^{p'}}.
		\end{align*}
		Note that while the intermediate manipulations are justified for $v$ 
		spectrally 
		localized away from $0$ and $\infty$, we may then pass to general $v 
		\in 
		L^{p'}$ using~\eqref{ch4:e:Lp_completeness} below.
	\end{proof}
	
	Noting also that
	\begin{align*}
	\| d f \|_{L^p} = \| d (-\Delta_g)^{-1/2} (-\Delta_g)^{1/2} 
	f\|_{L^p} \lesssim \| (-\Delta_g)^{1/2} f \|_{L^p},
	\end{align*}
	we summarize the previous two estimates in the following
	\begin{cor}[Equivalence of Sobolev norms]
		\label{ch4:c:sob_equiv}
		For all $1 < p < \infty$ and $f \in C^\infty_0$,
		\begin{align*}
		\| (-\Delta_{\delta})^{1/2} u\|_{L^p} \sim_p \| df \|_{L^p} \sim_p \| 
		(-\Delta_g)^{1/2} u\|_{L^p}.
		\end{align*}
	\end{cor}
	
	This corollary lets one transfer the Euclidean Leibniz and chain rule 
	estimates to the Sobolev norms defined by $(-\Delta_g)^{1/2}$, which are 
	better adapted to the equation as $(-\Delta_g)^{1/2}$ commutes with the 
	linear propagator. 
	We shall frequently employ the following
	
	\begin{cor}
		\label{ch4:c:chain_rule}
		\begin{align*}
		\| (-\Delta_g)^{1/2} F(u) \|_{L^p} \lesssim \| F'(u)\|_{L^q} \|
		(-\Delta_g)^{1/2} u\|_{L^r}
		\end{align*}
		whenever $p^{-1} = q^{-1} + r^{-1}$. In particular, we have
		\begin{align*}
		\| (-\Delta_g)^{1/2} (|u|^4u)\|_{L^2 L^{\frac{6}{5}}} \lesssim \| 
		u\|_{L^{10}
			L^{10}}^4 \| (-\Delta_g)^{1/2} \|_{L^{10} L^{\fr{30}{13}}}.
		\end{align*}
		
	\end{cor}

	\subsection{Strichartz estimates}
	
	
	Local-in-time Strichartz estimates without loss for compact nontrapping 
	metric 
	perturbations were first established by Staffilani and 
	Tataru~\cite{StaffilaniTataru2002}. As later observed, their argument can 
	be combined with the 
	global local smoothing estimate of Rodnianski and Tao to deduce 
	global-in-time Strichartz estimates~\cite{RodnianskiTao2007}. As mentioned 
	in the introduction, these results have since been extended to long-range 
	metrics.
	\begin{prop}
		\label{ch4:p:str}[\cite{StaffilaniTataru2002, RodnianskiTao2007}]
		For any function $u : I \times \mf{R}^3 \to \mf{C}$,
		\[
		\| u\|_{L^\infty L^2 \cap L^2 L^6} \lesssim \|u(0)\|_{L^2} + 
		\|(i\partial_t + 
		\Delta_g) u
		\|_{L^1 L^2 + L^2 L^{6/5}}\]
		In particular, by Sobolev embedding and Corollary~\ref{ch4:c:sob_equiv},
		\begin{align*}
		\| u\|_{L^{10} L^{10} } \lesssim \| (-\Delta_g)^{1/2} u\|_{L^{10} 
			L^{\frac{30}{13}}} 
		\lesssim \| u(0) 
		\|_{\dot{H}^1} + \| 
		\nabla (i\partial_t + \Delta_g) u\|_{L^1 L^2 + L^2 L^{6/5}}.
		\end{align*}
	\end{prop}
	
	In the sequel we adopt the notation
	\begin{align*}
	Z(I) = L^{10}_t L^{10}_x (I \times \mf{R}^3), \quad N(I) = ( L^1_t L^2_x + 
	L^2_t 
	L^{6/5}_x)(I \times \mf{R}^3).
	\end{align*}

	\subsection{Some harmonic analysis}
	
	In this section we set up a Littlewood-Paley theory, which will underlie 
	the 
	linear profile decomposition. We use the heat semigroup and follow 
	essentially 
	standard arguments that combine a
	spectral multiplier theorem with heat kernel bounds.
	
	Gaussian heat kernel bounds for $\Delta_g$ are classical. We quote a result 
	of 
	Aronson, who in fact considered uniformly elliptic 
	operators on Euclidean space; see the book~\cite{Grigoryan2009} for a 
	comprehensive survey. 
	\begin{thm}[\cite{Aronson1967}]
		\label{ch4:t:gaussian_heat_bounds}
		There exist a constant $c > 0$ such that
		\[
		e^{t\Delta_g}(x, y) \le c_1 t^{-\frac{3}{2}} e^{-\frac{d_g(x,y)^2}{ct}},
		\]
		where $d_g(x, y)$ is the Riemannian distance between $x$ and $y$.
	\end{thm}
	
	In view of this bound, we have access to a
	very general spectral multiplier theorem. For simplicity we 
	state just the special case that we need.
	\begin{thm}[{\cite[Theorem 3.1]{ThinhDuongOuhabazSikora2002}}]
		\label{ch4:t:multiplier}
		For any $F$ satisfying the 
		homogeneous symbol estimates 
		\[
		|\lambda^k \partial^k 
		F(\lambda)| \le C_k \text{ for all } 0 \le k \le \lceil{\fr{n}{2}} 
		\rceil 
		+ 1,
		\]
		the operator $F(-\Delta_g)$ maps $L^1 \to L^{1, \infty}$ and $L^p \to 
		L^p$ for 
		all $1 < p < \infty$.
	\end{thm}

	
	For a dyadic number $N  \in 2^{\mf{Z}}$, define 
	Littlewood-Paley projections in terms of the heat kernel
	\begin{alignat*}{4}
	\tilde{P}_{\le N} &= e^{\Delta_g / N^2}, &\quad \tilde{P}_{N} &= 
	e^{\Delta_g/N^2} - e^{4 \Delta_g /N^2}.
	\end{alignat*}
	
	Later (see Lemma~\ref{ch4:l:loc_smoothing}) we also introduce 
	Littlewood-Paley projections $P_{\le N}$ and $P_N$ using compactly 
	supported spectral 
	multipliers instead of the heat kernel.
	
	We have the Bernstein estimates
	\begin{prop}
		\label{ch4:p:bernstein}
		\begin{gather}
		\| \tilde{P}_{\le N} \|_{L^p \to L^p}
		\le 2, \ 1 < p < \infty.\label{ch4:e:LP-Lpbdd}\\
		\| \tilde{P}_{\le N} \|_{L^p \to L^q} \le c N^{\fr{d}{p} -
			\fr{d}{q}}, \ 1 \le p \le q \le \infty. \label{ch4:e:LP-bernstein}\\
		f = \sum_{N} \tilde{P}_{N} f  \ \text{ in } L^p, \ 1 
		< p < \infty. 
		\label{ch4:e:Lp_completeness}.
		\end{gather}
		Also, for all $1 < p < \infty$, the following square function estimate 
		holds
		\begin{align}
		\| (-\Delta_g)^{\frac{s}{2}} f\|_{L^p} 
		\sim_p \Bigl \| ( \sum_N |N^{s} (\tilde{P}_N)^k f |^2 
		)^{1/2} f\Bigr \|_{L^p} ,  \label{ch4:e:LP-sqfn}
		\end{align}
		whenever $2k > s$.
	\end{prop}
	
	\begin{proof}
		By the pointwise bound~\eqref{ch4:t:gaussian_heat_bounds} on the heat 
		kernel,
		\begin{align}
		\label{ch4:e:heat_kernel_bound}
		\|e^{t\Delta_g}\|_{L^1 \to L^\infty} \le c t^{-3/2}.
		\end{align}
		By duality,
		\begin{align*}
		\|e^{t\Delta_g}\|_{L^1 \to L^2} = \| e^{t\Delta_g} \|_{L^2 \to 
		L^\infty} = \|
		e^{2t\Delta_g} \|_{L^1 \to L^\infty}^{1/2} \le c t^{-\fr{3}{4}}.
		\end{align*}
		Since $\int e^{t\Delta_g}(x, y) \, dg(y) = \int e^{t\Delta_g}(x, y) \, 
		dg(x) 
		\equiv 
		1$, we have 
		\begin{align*}
		\| e^{t\Delta_g} \|_{L^p \to L^p} \le 1, \ 1 \le p \le \infty.
		\end{align*}
		The claims~\eqref{ch4:e:LP-Lpbdd} and~\eqref{ch4:e:LP-bernstein} follow 
		from
		interpolating these estimates.
		
		The convergence in~\eqref{ch4:e:Lp_completeness} follows from the 
		functional 
		calculus when $p = 2$. On the other hand, 
		Theorem~\ref{ch4:t:multiplier} 
		ensures boundedness in $L^p$ for all $1 < p < \infty$. By 
		interpolation, one 
		gets convergence for all such $p$.
		
		Finally, the square function estimate~\eqref{ch4:e:LP-sqfn} follows the 
		standard
		argument using independent random signs and the multiplier 
		theorem~\ref{ch4:t:multiplier}. The lower bound on $k$ ensures that the 
		symbol 
		for $(\tilde{P}_N)^k$ (which is not quite compactly supported) vanishes 
		at the 
		origin to higher order 
		than the symbol for 
		$(-\Delta_g)^{s/2}$; see~\cite{kvz_exterior_harmonic_analysis} for 
		details.
	\end{proof}
	
	\subsection{Local wellposedness}
	
	We summarize some standard results concerning the local existence, 
	uniqueness, and stability of solutions. These are proved by the usual 
	contraction mapping and bootstrap arguments for the Euclidean NLS (see 
	\cite{claynotes} and the references therein). These arguments 
	apply equally well in dimensions $3 \le d \le 6$. When $d > 6$, however, 
	the 
	stability theorem is proved in the Euclidean setting using exotic 
	Strichartz 
	estimates~\cite{tao-visan_stability,claynotes}. These are 
	derived using the Euclidean dispersive estimate, which is unavailable to us.
	
	\begin{prop}
		\label{ch4:p:lwp}
		There exists $\varepsilon_0 > 0$ such that for any $u_0 \in \dot{H}^1$, 
		and for 
		any interval $I \ni 0$ such
		that 
		\begin{align*}
		\| (-\Delta_g)^{1/2} e^{it\Delta_g} u_0 \|_{L^{10} L^{\frac{30}{13}}(I 
		\times 
			\mf{R}^3)} \le
		\varepsilon \le \varepsilon_0,
		\end{align*}
		there is a unique solution to~\eqref{ch4:e:main_eq} on $I$ with $u(0,
		x) = u_0$, which also satisfies 
		\begin{align}
		\label{ch4:e:lwp_spacetime}
		\| (-\Delta_g)^{1/2} u\|_{L^{10} L^{\frac{30}{13}}} \le 2 \varepsilon.
		\end{align}
		In particular, solutions with sufficiently small energy
		are global and scatter.
	\end{prop}
	
	\begin{proof}
		Run contraction mapping on the space $X$ defined by the
		conditions
		\begin{align*}
		\| (-\Delta_g)^{1/2} u\|_{L^{10} L^{\fr{30}{13}}} \le 2 \varepsilon, 
		\quad  
		\| (-\Delta_g)^{1/2} u
		\|_{L^\infty L^2}  \le \|u_0\|_{\dot{H}^1} + \varepsilon
		\end{align*}
		equipped with the metric
		$\rho(u, v) = \| (-\Delta_g)^{1/2}(u-v) \|_{L^{10} L^{\fr{30}{13}}}$.  
		For
		each $u \in X$, let $\mcal{I}(u)$ be the solution to the linear equation
		\begin{align*}
		(i\partial_t + \Delta_g) \mcal{I}(u) = |u|^4 u
		\end{align*}
		We check that for $\varepsilon$ sufficiently small, the map $u \mapsto
		\mcal{I}(u)$ is a contraction on $X$. By the Duhamel formula,
		Strichartz, the Leibniz rule, and Sobolev embedding,
		\begin{align*}
		\|(-\Delta_g)^{1/2} \mcal{I}(u) \|_{L^{10} L^{\fr{30}{13}}} &\le \| 
		(-\Delta_g)^{1/2}e^{it\Delta_g} u_0\|_{L^{10}
			L^{\fr{30}{13}}} + c\|(-\Delta_g)^{1/2}(|u|^4u)\|_{L^2 
			L^{\frac{6}{5}}} \\
		&\le \varepsilon +  c\|(-\Delta_g)^{1/2} u\|_{L^{10} L^{\fr{30}{13}}}^5
		\le \varepsilon + c (2\varepsilon)^5,
		\end{align*}
		\begin{align*}
		\| (-\Delta_g)^{1/2} \mcal{I}(u) \|_{L^\infty L^2} \le \|
		(-\Delta_g)^{1/2} u_0 \|_{L^2} + c  (2\varepsilon)^5
		\end{align*}
		Thus $\mcal{I}$ maps $X$ into itself. 
		
		For $u, v \in X$, the difference $\mcal{I}(u) - \mcal{I}(v)$ solves
		the equation with right hand side
		\begin{align*}
		|u|^4 u - |v|^4 v = \bigl( |u|^4  + \overline{u} v ( |u|^2 + |v|^2) 
		\bigr) (u -
		v) + v^2 (|u|^2 + |v|^2) (\overline{u} - \overline{v}).
		\end{align*}
		Hence, applying the Leibniz rule and Sobolev embedding repeatedly,
		\begin{align*}
		&\| (-\Delta_g)^{1/2} [\mcal{I}(u) - \mcal{I}(v)] \|_{L^{10}
			L^{\fr{30}{13}}} \\
		&\lesssim \|(-\Delta_g)^{1/2} (u-v) \|_{L^{10}
			L^{\fr{30}{13}}} ( \|(-\Delta_g)^{1/2}u\|_{L^{10} 
			L^{\frac{30}{13}}}^4 + \| 
		(-\Delta_g)^{1/2}v\|_{L^{10} L^{\frac{30}{13}}}^4 )\\
		&\lesssim (2\varepsilon)^4 \| (-\Delta_g)^{1/2}(u-v)\|_{L^{10} 
		L^{\fr{30}{13}}}.
		\end{align*}
	\end{proof}
	
	\begin{prop}
		\label{ch4:p:stability}
		Let $\tilde{u}$ solve the perturbed equation
		\begin{align}
		(i\partial_t + \Delta_g)\tilde{u} = \tilde{u}^4 \tilde{u} + e,
		\end{align}
		and let $0 \in I$ be an interval such that
		\begin{align*}
		\| \tilde{u} \|_{Z(I)} \le L, \quad \|
		\nabla \tilde{u} \|_{L^\infty L^2} \le E.
		\end{align*}
		Then there exists $\varepsilon_0(E, L)$ such that if $\varepsilon \le
		\varepsilon_0$ and
		\begin{align*}
		\| \tilde{u}(0) - u_0\|_{\dot{H}^1} + \| \nabla e \|_{N(I)}
		\le \varepsilon,
		\end{align*}
		there is a unique solution $u$ to~\eqref{ch4:e:main_eq} on $I$ with 
		$u(0)
		= u_0$, with
		\begin{align*}
		\| u - \tilde{u} \|_{Z(I)} + \| \nabla (u - \tilde{u}) \|_{L^2 L^6 \cap 
			L^\infty L^2}  &\le C(E, L)\varepsilon\\
		\| \nabla u\|_{L^2 L^6 \cap L^\infty L^2(I \times \mf{R}^3)} &\le C(E, 
		L).
		\end{align*}
	\end{prop}
	

	\section{Convergence of propagators}
	\label{ch4:s:propagator_convergence}
	
	\begin{thm}
		\label{ch4:t:convergence_of_propagators}
		Let $(\lambda_n, x_n)$ be a sequence of length scales and spatial 
		centers 
		conforming to one of the following 
		scenarios:
		\begin{itemize}
			\item [(a)] $\lambda_n \to \infty$.
			\item [(b)] $|x_n| \to \infty$.
			\item [(c)] $x_n \to x_\infty$, $\lambda_n \to 0$.
		\end{itemize}
		Let $\Delta := \delta^{jk} \partial_j \partial_k$ in the first 
		two 
		cases and $\Delta := g^{jk}(x_\infty) \partial_j \partial_k$ in 
		the 
		third.
		Then for any $\phi \in \dot{H}^1$, writing $\phi_n = \lambda_n^{- 
			\frac{d-2}{2}} \phi(\tfr{\cdot - x_n}{\lambda_n})$, we have
		\[
		\lim_{n \to \infty}	\| e^{it\Delta_g} \phi_n - e^{it\Delta} 
		\phi_n\|_{L^\infty L^6 ( \mf{R} \times \mf{R}^3 )} = 0.
		\]
		In cases (a), (b), the convergence actually occurs in $L^\infty 
		\dot{H}^1$.
	\end{thm}
	
	\begin{proof}
		By approximation in $\dot{H}^1$, we may assume that $\phi$ is Schwartz.
		
		Suppose first that $\lambda_n \to \infty$. By the Strichartz inequality 
		the 
		equivalence of Sobolev norms, and the Leibniz rule,
		\begin{align*}
		\| e^{it\Delta_g} \phi_n - e^{it\Delta} \phi_n \|_{L^\infty L^6} 
		&\lesssim 
		(-\Delta_g)^{1/2} (\Delta_g - \Delta)e^{it\Delta} \phi_n 
		\|_{L^2 
		L^{6/5}} \\
		&\lesssim \| \chi \nabla e^{it\Delta} \phi_n\|_{L^{2} L^{6/5}} + 
		\| 
		\chi 
		\nabla^2 e^{it\Delta} \phi_n\|_{L^2 L^{6/5}} + \| \chi \nabla^3 
		e^{it\Delta} 
		\phi_n 
		\|_{L^2 L^{6/5}}
		\end{align*}
		where $\chi(x)$ is the characteristic function of the unit ball. By 
		H\"{o}lder 
		and the Euclidean dispersive estimate,
		\begin{align*}
		\| \chi \nabla e^{it\Delta} \phi_n \|_{L^2 L^{6/5}} = 
		\lambda_n^2 \| 
		\chi(\lambda_n 
		\cdot) e^{it\Delta} \phi\|_{L^2 L^{6/5}} \lesssim 
		\lambda_n^{-\frac{1}{2}} \| 
		e^{it\Delta} \phi \|_{L^2 L^\infty}  \lesssim \lambda_n^{-1/2} 
		\| 
		\phi\|_{L^1}.
		\end{align*}
		The terms involving two or more derivatives enjoy even better decay 
		since 
		$\lambda_n \to \infty$.
		
		Assume now that $|x_n| \to \infty$, $\lambda_n \equiv \lambda_0\in (0, 
		\infty)$. By the Duhamel formula and Sobolev embedding, 
		\begin{align*}
		\| e^{it\Delta_g} \phi_n - e^{it\Delta} \phi_n \|_{L^\infty L^2} 
		\lesssim \| 
		(\Delta_g - \Delta) e^{it\Delta} \phi_n \|_{L^1 L^2} 
		\lesssim \| \chi 
		e^{it\Delta} \nabla \phi_n \|_{L^1 L^2} + \| \chi 
		e^{it\Delta} 
		\nabla^2 \phi_n 
		\|_{L^1 
			L^2},
		\end{align*}	
		where $\chi$ is a bump function supported on the unit ball. For any 
		fixed $T > 
		0$, 
		decompose
		\begin{align*}
		\| \chi e^{it\Delta} \nabla \phi_n \|_{L^1 L^2} \le \| \chi_n 
		e^{it\Delta} 
		\phi 
		\|_{L^1L^2( \{ |t| \le T\}  )} + \| \chi_n e^{it\Delta} 
		\phi\|_{L^1 
		L^2( \{ |t| 
			> T\})},
		\end{align*}
		where $\chi_n = \chi(\cdot + x_n)$. The first term vanishes as $n \to 
		\infty$ 
		because the orbit $\{ e^{it\Delta} \phi\}_{|t| \le T}$ is 
		compact in 
		$L^2$. We 
		use H\"{o}lder's inequality and the dispersive estimate to bound the 
		second 
		term by
		\begin{align*}
		\| e^{it\Delta} \phi\|_{L^1 L^\infty(\{ |t| > T\})} \lesssim 
		T^{-\frac{1}{2}} 
		\| \phi\|_{L^1}.
		\end{align*}
		As $T$ may be chosen arbitrarily large, we conclude that $\lim_{n \to 
		\infty}\| 
		\chi e^{it\Delta} \nabla \phi_n \|_{L^1 L^2} = 0$, and similar 
		considerations 
		estimate the term $\| \chi e^{it\Delta} \nabla^2 \phi_n \|_{L^1 
		L^2}$. 
		Finally, we have
		\begin{align*}
		\| e^{it\Delta_g} \phi_n - e^{it\Delta} \phi_n \|_{L^\infty L^6} 
		\le \| 
		\cdots 
		\|_{L^\infty L^2}^{\frac{1}{3}} \| \cdots \|_{L^\infty 
		L^\infty}^{\frac{2}{3}},
		\end{align*}
		and the uniform norms may be estimated via Sobolev embedding:
		\begin{align*}
		\| e^{it\Delta_g} \phi_n\|_{L^\infty L^\infty} \lesssim \| (1 - 
		\Delta_g) 
		e^{it\Delta_g} \phi_n \|_{L^\infty L^2} \lesssim \| (1 - \Delta_g) 
		\phi_n 
		\|_{L^2} \lesssim 1.
		\end{align*}
		
		Consider now the scenario where $|x_n| \to \infty$ and $\lambda_n \to 
		0$. We 
		may assume that $\phi$ is compactly supported. Let $\chi$ be a smooth 
		function 
		such that $\chi(x) = 1 $ when $|x| \ge 11/10$ and 
		$\chi(x) = 0$ for $|x| \le 1$. First we show
		\begin{align}
		\label{ch4:e:cutoff_negligible}
		\lim_{n \to \infty} \| (1 - \chi) e^{it\Delta} \phi_n 
		\|_{L^\infty L^6} 
		= 0.
		\end{align}
		The function $\chi e^{it\Delta} \phi_n$ solves the equation
		\begin{align*}
		(i\partial_t + \Delta) (\chi e^{it\Delta} \phi_n) = [\chi, 
		\Delta] 
		e^{it\Delta} 
		\phi_n.
		\end{align*}
		Thus, by Sobolev embedding and the Duhamel formula,
		\begin{align*}
		\| (1-\chi)e^{it\Delta} \phi_n \|_{L^\infty L^6} \lesssim \| 
		\nabla 
		[\chi, 
		\Delta] e^{it\Delta} \phi_n \|_{L^1 L^2} .
		\end{align*}
		The right side has the form
		\begin{align*}
		\| \beta e^{it\Delta} \nabla \phi_n \|_{L^1 L^2} + \| \beta 
		e^{it\Delta} \nabla^2 
		\phi_n 
		\|_{L^1 L^2}
		\end{align*}
		where $\beta$ is a bump function localizing to the unit ball. We focus 
		on the 
		potentially more dangerous second term. Fix $T > 0$ large, and split
		\begin{align*}
		\|\chi e^{it\Delta} \nabla^2 \phi_n\|_{L^1 L^2} \le \| \beta 
		e^{it\Delta} 
		\nabla^2\phi_n 
		\|_{L^1 L^2( \{ |t| \le T \lambda_n\} )} + \| \beta e^{it\Delta} 
		\nabla^2 
		\phi_n 
		\|_{L^1 
			L^2 ( \{ |t| > T \lambda_n \})}.
		\end{align*}
		By H\"{o}lder in time and a change of variable, the first term may be 
		written as
		\begin{align*}
		\|\beta(x_n + \lambda_n \cdot) e^{it\Delta} \nabla^2 \phi 
		\|_{L^\infty L^2( 
		\{ |t| 
			\le T 
			\lambda_n^{-1} \})},
		\end{align*}
		which goes to zero as $n \to \infty$ by approximate finite speed of 
		propagation or, more precisely, by the Fraunhofer formula
		\begin{align*}
		\lim_{t \to \infty} \| e^{it\Delta} f - (2 i t)^{-\frac{3}{2}} 
		\hat{f}(\tfrac{x}{2t}) e^{\frac{|x|^2}{4t}} \|_{L^2} = 0.
		\end{align*}
		By the dispersive estimate,
		\begin{align*}
		\| \beta e^{it\Delta} \nabla^2 \phi_n \|_{L^1 L^2 ( \{ |t| > T 
		\lambda_n 
		\})} 
		\lesssim \| e^{it\Delta} \nabla^2 \phi_n \|_{L^1 L^\infty ( \{ 
		|t| > 
		T 
		\lambda_n 
			\})} \lesssim \lambda_n^{\frac{1}{2}} (T \lambda_n)^{-\frac{1}{2}} 
			\| 
		\phi\|_{L^1} \lesssim T^{-\frac{1}{2}}.
		\end{align*}
		Hence, choosing $T$ arbitrarily large, 
		\[
		\lim_{n \to \infty} \| \beta e^{it\Delta} \nabla^2 \phi_n\|_{L^1 
		L^2} 
		= 0,
		\]
		establishing~\eqref{ch4:e:cutoff_negligible}.
		
		Since $\Delta = \Delta_g$ on the support of the cutoff $\chi$, we also 
		have
		\begin{align*}
		(i\partial_t + \Delta_g) (\chi e^{it\Delta} \phi_n) = [ \chi, 
		\Delta] 
		e^{it\Delta} \phi_n,
		\end{align*}
		so by the Duhamel formula, Sobolev embedding, and the equivalence of 
		$\dot{H}^1$ Sobolev norms,
		\begin{align*}
		\| e^{it\Delta_g} \phi_n - \chi e^{it\Delta} \phi_n \|_{L^\infty 
		L^6} 
		&= 
		\Bigl \| \int_0^t e^{i(t-s)\Delta_g} [\Delta, \chi] e^{is 
		\Delta} 
		\phi_n \, 
		ds \|_{L^\infty L^6} \lesssim \| (-\Delta_g)^{1/2} [\chi, \Delta] 
		e^{it\Delta} \phi_n\|_{L^1 L^2} \\
		&\lesssim \| \nabla [\chi, \Delta] 
		e^{it\Delta} \phi_n\|_{L^1 L^2}
		\end{align*}
		which was just estimated.
		
		Finally, consider the last case where the profile $\phi_n$ is 
		concentrating 
		at a point. For $T > 0$, split
		\begin{align}
		\label{ch4:e:propagator_conv_concentrating_profile}
		\| e^{it\Delta_g} \phi_n - e^{it\Delta} \phi_n \|_{L^\infty L^6} 
		\le 
		\| \cdots \|_{L^\infty L^6 ( \{ |t| \le T \lambda_n^2 \})} + \| \cdots 
		\|_{L^\infty L^6 ( \{ |t| > T \lambda_n^2\})}.
		\end{align}
		For the short time contribution, let $\chi$ be a bump function centered 
		at 
		the origin, fix $0 < \theta <1$, and define
		\begin{align*}
		\chi_n = \chi\Bigl(\frac{\cdot - x_n}{\lambda_n^\theta} \Bigr), \ v_n = 
		e^{it\Delta} \phi_n.
		\end{align*}
		Then
		\begin{align*}
		(i\partial_t + \Delta)(\chi_n v_n) = [ \Delta, \chi_n] 
		v_n = 
		2 
		\langle \nabla_\infty \chi_n, \nabla v_n \rangle_\infty  + 
		(\Delta 
		\chi_n)v_n,
		\end{align*}
		where the inner product on the right is respect to the metric 
		$g(x_\infty)$, 
		hence
		\begin{align*}
		\| (1 - \chi_n)v_n \|_{L^\infty L^6( \{ |t| \le T \lambda_n^2 \})} 
		&\lesssim \| 
		(1 - \chi_n) v_n\|_{\dot{H}^1} + \| \nabla [\Delta, \chi_n] 
		v_n\|_{L^1 
			L^2 (\{|t| \le T \lambda_n^2\})}\\
		&\lesssim o(1) + T \lambda_n \| \phi\|_{H^2}.
		\end{align*}
		Further, writing $(i\partial_t + \Delta) = (i\partial_t + 
		\Delta_g) 
		+ 
		(\Delta - \Delta_g)$, we obtain by the Duhamel formula and 
		Sobolev 
		embedding
		\begin{align*}
		\| e^{it\Delta_g} \phi_n - \chi_n e^{it\Delta} \phi_n 
		\|_{L^\infty L^6( 
		\{ 
			|t| 
			\le T \lambda_n^2\} )} &\lesssim \| (1 - \chi_n) 
			\phi_n\|_{\dot{H}^1} + 
		\| 
		\nabla [\Delta, \chi_n] v_n \|_{L^1 L^2 ( \{ |t| \le T 
			\lambda_n^2\})} \\
		&+ 
		\| \nabla (\Delta_g - \Delta) (\chi_n  v_n) \|_{L^1 L^2 ( \{ |t| 
		\le 
			T 
			\lambda_n^2 \})}.
		\end{align*}
		The first two terms were estimated before. Writing out $\Delta_g - 
		\Delta$ explicitly and using the Leibniz rule, we see that the 
		worst 
		contributions to 
		the last term are quantities of the form
		\begin{align*}
		\| (g - g(x_\infty)) \chi_n \nabla^3 v_n\|_{L^1 L^2( \{ |t| \le 
		T\lambda_n^2 
			\})} 
		&\lesssim T\lambda_n^2 \lambda_n^{-2} (|x_n - x_\infty| + 
		\lambda_n^\theta) 
		\| e^{it\Delta} \nabla^3 \phi\|_{L^\infty L^2} ,
		\end{align*}
		which is acceptable.
		
		The long time contribution 
		to~\eqref{ch4:e:propagator_conv_concentrating_profile} is bounded by
		\[
		\| e^{it\Delta_g} \phi_n\|_{L^\infty L^6 ( \{|t| > T\lambda_n^2 \})} + 
		\| 
		e^{it\Delta} \phi_n \|_{L^\infty L^6 ( \{ |t| > T \lambda_n^2 
		\})},
		\]
		which are dealt with respectively by the extinction lemma in the next 
		section and the usual dispersive estimate
		\begin{align*}
		\| e^{it\Delta} \phi_n \|_{L^\infty L^6 ( \{|t| > T\lambda_n^2\} 
		)} 
		\lesssim T^{-1} \| \phi\|_{L^{6/5}}.
		\end{align*}
	\end{proof}
	
	The proof of the last case yields the following corollary, which asserts 
	that 
	on short time intervals, the convergence in Case (c) of the theorem occurs 
	in 
	the energy norm as well.
	\begin{cor}
		\label{ch4:c:concentrating_profile_strong_conv}
		Let $(\lambda_n, x_n)$ be a sequence such that $x_n \to x_\infty$ and 
		$\lambda_n\to 0$. Then for any $T > 0$
		\begin{align*}
		\lim_{n \to \infty} \| e^{it\Delta_g} \phi_n - e^{it\Delta} 
		\phi_n\|_{L^\infty \dot{H}^1 ( [-T\lambda_n^2, T \lambda_n^2] \times 
			\mf{R}^3 )} = 0.
		\end{align*}
	\end{cor}
	
	\section{An extinction lemma}
	\label{ch4:s:extinction}
	
	The purpose of this section is to prove a long-time weak dispersion 
	estimate 
	for 
	linear profiles 
	concentrating within a bounded distance of the origin, which arise in the 
	last 
	case of Theorem~\ref{ch4:t:convergence_of_propagators}.  For profiles with 
	width $h$, we want to establish decay for times $t \ge Th^2$ 
	as 
	$h \to 0$ and $T \to \infty$. We consider the times $t \le O(h)$ 
	and $t \gg h$ separately. Semiclassical techniques are used for the first 
	regime, while for longer times we invoke the global 
	geometry 
	to see that the 
	solution is essentially 
	Euclidean. Our ingredients consist of 
	the frequency-localized dispersion estimate of 
	Burq-Gerard-Tzvetkov~\cite{BurqGerardTzvetkov2004}, a 
	wavepacket parametrix, and a non-concentration estimate for the geodesic 
	flow.


	
	\begin{prop}
		\label{ch4:p:extinction}
		Let $d \ge 3$, and suppose $x_h \to x_0 \in \mf{R}^d$ as $h \to 0$. For 
		any $\phi \in 
		\dot{H}^1$, 
		denoting $\phi_h = h^{-\frac{d-2}{2}} \phi(h^{-1}(\cdot - x_h))$, we 
		have
		\begin{align*}
		\lim_{T \to \infty} \limsup_{h \to 0} \| e^{it\Delta_g}
		\phi_h\|_{L^\infty L^{\frac{2d}{d-2}} ( ( [Th^2, \infty) \times 
		\mf{R}^d )} = 
		0.
		\end{align*}
	\end{prop}
	
	\begin{proof}
		We begin with several reductions. After a translation we may assume 
		that $x_0 = 
		0$. Also, letting $\rho = |g|^{\frac{1}{4}}$ be the 
		square 
		root of the Riemannian 
		density, we have  $e^{it\Delta_g} = \rho^{-1} e^{-itA} 
		\rho$, where 
		\begin{align*}
		A = \rho (-\Delta_g) \rho^{-1} = -\partial_{j} g^{jk} \partial_k + V
		\end{align*}
		is self-adjoint on $L^2( dx)$ and $V$ is a compactly supported 
		potential. 
		Thus
		\begin{align*}
		\| e^{it\Delta_g} \phi_h \|_{L^\infty L^{\frac{2d}{d-2}}}  &\lesssim \| 
		e^{-itA} 
		\rho \phi_h \|_{L^{\frac{2d}{d-2}}} \lesssim \rho(x_h)\| e^{-itA} 
		\phi_h \|_{L^\infty L^{\frac{2d}{d-2}}} + \| e^{-itA}( \rho - 
		\rho(x_h)) \phi_h \|_{L^\infty L^{\frac{2d}{d-2}}}\\
		&\lesssim \| e^{-itA} \phi_h \|_{L^\infty L^{\frac{2d}{d-2}}} + o(1) 
		\textrm{ 
			as } h \to 0,
		\end{align*}
		and it suffices to show
		\begin{align}
		\label{ch4:e:conjugated_extinction}
		\lim_{T \to \infty} \limsup_{h \to 0} \| e^{-itA} \phi_h\|_{L^\infty 
			L^{\frac{2d}{d-2}} ( 
			[Th^2, \infty) \times \mf{R}^d)} = 0.
		\end{align}
		Compared to the Laplacian, the conjugated operator $A$ leads to better 
		error terms 
		when we later consider the dynamics driven by the principal 
		symbol $a(x, \xi) = g^{jk} \xi_j \xi_k$; the Weyl 
		quantization
		$a^w(X, D)$
		 differs from $A$ by a zero order potential whereas $a^w(X, 
		D) + \Delta_g$ is first order.

		Further, we shall assume that $\phi$ is Schwartz and that its Fourier 
		transform $\hat{\phi}$ is supported in a frequency annulus
		\begin{align}
		\label{ch4:e:freq_localization}
		\opn{supp} \hat{\phi} \subset \{ \varepsilon < |\xi| < \varepsilon^{-1} 
		\}
		\end{align} 
		for some $\varepsilon > 0$; the rescaled initial data $\phi_h$ are 
		therefore 
		frequency-localized to the region $\{ 
		h^{-1} 
		\varepsilon < |\xi| < h^{-1} \varepsilon^{-1} \}$. 
		
		By the semiclassical dispersion estimate of 
		Burq-Gerard-Tzvetkov~\cite[Lemma 
		A3]{BurqGerardTzvetkov2004} (see also \cite[Proposition 
		4.7]{KochTataru2005}), there 
		exists $c > 0$ such that
		\begin{align*}
		\| e^{-itA} \phi_h \|_{L^\infty L^{\frac{2d}{d-2}} ([Th^2, ch] \times 
			\mf{R}^d)} \lesssim 
		|Th^2|^{-1} \| \phi_h\|_{L^{\frac{2d}{d+2}}} = T^{-1} \| 
		\phi\|_{L^\frac{2d}{d+2}}.
		\end{align*}
		Hence, it remains to prove the long-time extinction
		\begin{align}
		\label{ch4:e:long_t_extinction}
		\lim_{h \to 0} \| e^{-itA} \phi_h\|_{L^\infty L^{\frac{2d}{d-2}} ([ch, 
		\infty) 
			\times 
			\mf{R}^d)} = 0.
		\end{align}
		
		\subsection*{Wavepacket decomposition}
		
For 
		each $h > 
		0$ and $(x_0, \xi_0)$, define
		\begin{align*}
		\psi_{(x_0, \xi_0)}^h(y) = 2^{-\frac{d}{2}} \pi^{-\frac{3d}{4}} 
		h^{-\frac{3d}{4}} 
		e^{\frac{i\xi_0(y-x_0)}{h}} e^{-\frac{(y-x_0)^2}{2h}},
		\end{align*}
		which is a Gaussian wavepacket localized in phase space to the box 
		\begin{align*}
		\{ (x, \xi) : |x-x_0| \le h^{1/2}, \ |\xi - h^{-1} \xi_0| \le h^{-1/2} 
		\}.
		\end{align*}
		The FBI transform at scale $h$ (see for 
		example~\cite{StaffilaniTataru2002} and the references therein) is an 
		isometry $T_h: L^2(\mf{R}^d) 
		\to L^2(\mf{R}^d \times \mf{R}^d)$ defined by 
		\begin{align*}
		T_h f (x, \xi) = \langle \psi_{(x, \xi)}^h, f \rangle = c_d 
		h^{-\frac{3d}{4}} 
		\int e^{ \frac{i \xi(x-y)}{h}} e^{-\frac{(x-y)^2}{2h}} f(y) \, dy = c_d 
		h^{-\frac{5d}{4}} \int e^{\frac{ix\eta}{h}} e^{-\frac{(\xi - 
		\eta)^2}{2h} }
		\hat{f} (\tfrac{\eta}{h}) \, d\eta.
		\end{align*}
		From the adjoint formula $T_h^* F(y) = \int \psi_{(x, \xi)}^h (y) F(x, 
		\xi) \, 
		dx d\xi$, one obtains for each $f \in L^2(\mf{R}^d)$ a 
		representation
		\begin{align*}
		f = T_h^* T_h f = \int \langle \psi_{(x, \xi)}^h, f \rangle \psi_{(x, 
		\xi)}^h 
		\, dx d\xi.
		\end{align*}
		as a continuous superposition of width $h^{1/2}$ wavepackets.
		Such a decomposition is useful for studying semiclassical 
		Schr\"{o}dinger 
		dynamics as the Schr\"{o}dinger evolution of each wavepacket 
		$\psi_{(x_0, 
			\xi_0)}^h$ will remain coherent and behave essentially as a 
			classical particle 
		on time scales of order $h$.
		
		Returning to our problem, write
		\begin{align*}
		\phi_h = \int \psi_{(x, \xi)}^h T_h \phi_h(x, \xi) \, dx d\xi.
		\end{align*}
		We may restrict attention to just the wavepackets from the region
		\begin{align}
		\label{ch4:e:fbi_support}
		B = \{ (x, \xi) : |x-x_h| \le h^\theta, \ \tfrac{\varepsilon}{10} \le 
		\xi \le 
		\tfrac{10}{\varepsilon} \}
		\end{align}
		for any $\theta < \tfrac{1}{2}$. Indeed, if $|x-x_h| > h^\theta$ then 
		\begin{align*}
		|T_h \phi_h (x, \xi)| &\lesssim h^{-\frac{3d}{4}} h^{-\frac{d-2}{2}} 
		\int 
		e^{-\frac{(x-x_h-y)^2}{2h}} |\phi(\tfrac{y}{h} )| \, dy\\
		&\lesssim h^{1 - \frac{5d}{4}}  \int_{|y| \le |x-x_h|/4} +  h^{1 - 
			\frac{5d}{4}} \int_{|y| > 
			|x-x_h| / 4} \\
		&\lesssim_N h^{1 - \frac{5d}{4} + \theta d} e^{-\frac{(x-x_h)^2}{ch}} + 
		h^{1 - 
			\frac{5d}{4}} h^{N} |x-x_h|^{-N} \\
		&\lesssim_{M, N} h^M |x-x_h|^{-N}
		\end{align*}
		for any $M, N \ge 0$. Similarly,
		\begin{align*}
		|T_h \phi_h(x, \xi)| \lesssim h^{1 - \frac{3d}{4}} \int e^{-\frac{(\eta 
		- 
				\xi)^2}{2h}} | \hat{\phi}(\eta) | \, d\eta \lesssim 
				\left\{\begin{array}{cc} 
		h^{1-\frac{3d}{4}} e^{-\frac{\varepsilon^2}{ch}}, & |\xi| < 
		\varepsilon/10\\
		h^{1-\frac{3d}{4}} e^{-\frac{\xi^2}{ch}}, & |\xi| > 10 / \varepsilon 
		\end{array}\right.
		\end{align*}

		In view of these bounds, we decompose
		\begin{align}
		\label{ch4:e:fbi_supp_decomp}
		\phi_h = T_h^* \mr{1}_B T_h \phi_h + T_h^* (1 - \mr{1}_B) T_h \phi_h = 
		f_h^1 + 
		f_h^2,
		\end{align}
		where by the triangle inequality we obtain, for any $k \ge 0$,
		\begin{align*}
		\| \partial^k f_h^2\|_{L^2} \lesssim \int_{B^c} (h^{-\frac{d+k}{2}} + 
		h^{-\frac{d}{2}} |h^{-1} \xi|^k) |T_h \phi_h (x, \xi)| \, dx d\xi = 
		O(h^\infty).
		\end{align*}
		
		By Sobolev embedding, it therefore suffices to show
		\begin{align}
		\label{ch4:e:extinction_final_reduction}
		\lim_{h \to 0} \| e^{-itA} f_h^1 \|_{L^\infty L^{\frac{2d}{d-2}}( (ch, 
		\infty) 
			\times 
			\mf{R}^3)} = 0.
		\end{align}
		
		To prove this, we fix a large $T > 0$ and consider separately the time 
		intervals $[ch, Th]$ and $[Th, \infty)$. On semiclassical time scales, 
		the 
		quantum evolution of wavepackets is modeled by the geodesic flow. More 
		precisely, if $\psi_{(x, \xi)}^h$ is a typical wavepacket, then for 
		$|t| \le 
		Th$ its Schr\"{o}dinger evolution $e^{-itA} \psi_{(x, \xi)}^h$ will 
		have width 
		$C_T h^{1/2}$ and
		travel along the geodesic starting at $x$ with initial momentum $h^{-1} 
		\xi$ 
		(that is, with velocity $h^{-1} g^{ab} \xi_b$). 
		
		If $T$ is sufficiently large, then by the nontrapping assumption 
		on the metric, all the wavepackets $e^{-iTh A} \psi_{(x, \xi)}^h$ with 
		$(x, 
		\xi) \in B$ will 
		have exited the curved region (this is why it is convenient to assume
		that $\phi$ is frequency-localized away from $0$), and for 
		$t \ge Th$ the 
		solution $e^{-itA} \psi_{(x, \xi)}^h$ will radiate to infinity while 
		dispersing 
		essentially as a Euclidean free particle. The decay for $e^{-itA} 
		f^1_h$ 
		will then be a consequence of the dispersive properties of the 
		Euclidean 
		propagator 
		$e^{it\Delta_{\mf{R}^3}}$.
		
		It will be notationally convenient in the sequel to rescale time 
		semiclassically, that is, replace 
		$t$ by $th$, so that each wavepacket $\psi_{(x, \xi)}^h$ travels at 
		speed 
		$O(1)$ 
		under the 
		propagator $e^{-ith A}$. The desired estimate then becomes
		\begin{align*}
		\lim_{h \to 0} \| e^{-ith A} f_h^1 \|_{L^\infty L^{\frac{2d}{d-2}} ( 
		(c, 
			\infty) \times 
			\mf{R}^3)} = 0.
		\end{align*}
		
		\subsection*{Frequency-localization}
		
		We show next that 
		the operator $A$ may be replaced, up to acceptable errors, by a 
		frequency-localized version. This will let us engage the results
		of Koch 
		and 
		Tataru~\cite{KochTataru2005} for the 
		evolution of wavepackets at fixed frequency.
		
		Choose frequency cutoffs $\chi_j \in C^\infty_0( \mf{R}^d \setminus 
		\{0\})$ 
		such that 
		\begin{align*}
		\{ \xi : \varepsilon \le |\xi| \le \varepsilon^{-1} \} \prec \chi_1 
		\prec 
		\chi_2 \prec \chi_3;
		\end{align*}
		that is, $\chi_1(\xi) = 1$ on the annulus $\varepsilon \le |\xi| \le 
		\varepsilon^{-1}$ and $\chi_j = 1$ near the support of $\chi_{j-1}$.
		Set $A(h) = h^2 A$, let $a = g^{ij} \xi_i \xi_j$ be the principal 
		symbol of 
		$A$, and define the localized operator
		\[
		A'(h) = (\chi_3 a)^w (X, hD) = (2\pi h)^{-d} \int_{\mf{R}^d} 
		e^{\frac{i(x-y) 
				\xi}{h}} 
		a\Bigl( \frac{x+y}{2}, \xi \Bigr)\chi_3(\xi) \, d\xi.
		\]
		We check that the propagator 
		$e^{-\frac{itA'(h)}{h}}$, which preserves $L^2$, is also
		bounded on $\dot{H}^1$ when restricted to frequency $h^{-1}$.
		
		\begin{lma}
			\label{ch4:l:freq-localized_sobolev_cty}
			\begin{align*}
			\| e^{-\frac{itA'(h)}{h}} \chi_1(hD) \|_{\dot{H}^1 \to \dot{H}^1} 
			\le c_k 
			(1 + |t|).
			\end{align*}
		\end{lma}
		
		\begin{proof}
			Let $u_h = e^{-\frac{itA'(h)}{h}} \chi_1(hD)$ be the solution to 
			the  
			evolution equation
			\begin{align*}
			[hD_t + A'(h)] u_h = 0 ,  u_h(0) = \chi_1(hD) \phi.
			\end{align*}
			Differentiating this equation, we obtain
			\begin{align*}
			[hD_t + A'(h)] (hD) u_h = [ hD, A'(h) ] u_h.
			\end{align*}
			By pseudodifferential calculus, $ \| [hD, A'(h)] \|_{L^2 \to L^2} 
			\le c h$, 
			so
			\begin{align*}
			\| hD u_h(t) \|_{L^2} &\le \| hD u_h  (0) \|_{L^2} + h^{-1} 
			\int_0^t \| 
			[hD, A'(h)] u_h(s) \|_{L^2} \ ds\\
			&\le \| hD u_h(0) \|_{L^2} + c |t| \| \chi_1(hD) \phi\|_{L^2}\\
			&\le c (1 + |t|) \| hD\chi_1(hD) \phi\|_{L^2}.
			\end{align*}
		\end{proof}
		
		\begin{lma}
			For each $T > 0$ and for all $|t| \le T$,
			\begin{align*}
			\| (e^{-\frac{itA(h)}{h}} - e^{-\frac{itA'(h)}{h}} ) \chi_1(hD) 
			\|_{\dot{H}^1 \to \dot{H}^1} \le c_T h |t|
			\end{align*}
		\end{lma}
		
		\begin{proof}
			Write
			\begin{align*}
			e^{-\frac{itA(h)}{h}} - e^{-\frac{itA'(h)}{h}} = 
			(e^{-\frac{itA(h)}{h}} - 
			e^{-\frac{it\tilde{A}(h)}{h}}) + ( e^{-\frac{it\tilde{A}(h)}{h}} - 
			e^{-\frac{itA'(h)}{h}}),
			\end{align*}
			where
			\begin{align*}
			\tilde{A}(h) = a^w(X, hD) = -h^2 \partial_j g^{jk} \partial_k - 
			\frac{h^2}{4} (\partial_j \partial_k g^{jk}).
			\end{align*}
			By the Duhamel formula,
			\begin{align*}
			\| e^{-\fr{itA(h)}{h}}\phi - e^{-\fr{it\tilde{A}(h)}{h}} \phi 
			\|_{L^2
				\to L^2} \le h \int_0^t \bigl| \frac{1}{4} \partial_j 
				\partial_k g^{jk} 
			+ V \bigr|
			\|e^{-\fr{isA(h)}{h}} \phi\|_{L^2} \, ds \le ch |t| \| \phi\|_{L^2}.
			\end{align*}
			Introducing the frequency-localization, we see from semiclassical 
			functional calculus (see~\cite{BurqGerardTzvetkov2004}) that 
			\[
			\|(1 - \chi_2(hD))e^{-\fr{itA(h)}{h}} \chi_1(hD)\|_{L^2 \to 
			H^\sigma}
			= O(h^\infty)
			\]
			and similarly with $A$ replaced by $\tilde{A}$. That is, the 
			linear
			evolutions $e^{-\fr{itA(h)}{h}}$
			and $e^{-\fr{it\tilde{A}(h)}{h}}$
			preserve frequency-support.

			Thus
			\begin{equation}
			\label{ch4:e:comp_propagators_1}
			\begin{split}
			\|D( e^{-\fr{itA(h)}{h}} - e^{-\fr{it\tilde{A}(h)}{h}}) \chi_1(hD) 
			\phi
			\|_{L^2} &\le h^{-1} \| (e^{-\fr{itA(h)}{h}} -
			e^{-\fr{\tilde{A}(h)}{h}}) \chi_1(hD) \phi\|_{L^2} + O(h^\infty)\\
			&\lesssim |t| h \| \chi_1(hD)h^{-1} \phi\|_{L^2} \lesssim |t| h \| 
			D 
			\phi\|_{L^2}.
			\end{split}
			\end{equation}
			
			Now we show that
			\begin{align*}
			\| (e^{-\frac{it\tilde{A}(h)}{h}} - e^{-\frac{itA'(h)}{h}} ) 
			\chi_1(hD) 
			\|_{\dot{H}^1 \to \dot{H}^1} \le c h |t|.
			\end{align*}
			For each $\phi \in \dot{H}^1$, the function $u_h = e^{-\frac{it 
					\tilde{A}(h)}{h}} \chi_1( hD) \phi$ solves the equation 
					$[hD_t + A'(h) ] 
			u_h = r_h$, where
			\begin{align*}
			r_h = [(\chi_3 - 1) a]^w(X, hD) \chi_2 (hD) u_h + [(\chi_3 - 1) 
			a]^w (X, 
			hD)(1 - \chi_2(hD)) u_h.
			\end{align*}
			
			As the symbols $(\chi_3 - 1)a$ and $\chi_2$ have disjoint supports, 
			the first 
			term on the left is $O(h^\infty)$ in any Sobolev norm. The 
			frequency 
			localization of $u_h$ implies that the second term is similarly 
			negligible. By 
			the Duhamel formula and 
			Lemma~\ref{ch4:l:freq-localized_sobolev_cty}, for any 
			$T > 
			0$ and $|t| \le T$ we have
			\begin{equation}
			\label{ch4:e:comp_propagators_2}
			\| (e^{-\frac{itA'(h)}{h}} - e^{-\frac{it \tilde{A}(h)}{h}}) 
			\chi_1(hD) 
			\phi\|_{\dot{H}^1} \le c_T \int_0^t \| r_h(s) \|_{\dot{H}^1} \, ds 
			\le 
			|t| O(h^\infty) \| \phi\|_{\dot{H}^1}.
			\end{equation}	
			\eqref{ch4:e:comp_propagators_1} and 
			\eqref{ch4:e:comp_propagators_2} complete the proof.
		\end{proof}
		
		Returning to the decomposition~\eqref{ch4:e:fbi_supp_decomp} and 
		recalling that 
		$\phi_h = \chi_1(hD) \phi_h$, we have
		\begin{align*}
		\| (1 - \chi_1(hD)) f_h^1\|_{\dot{H}^1} = O(h^\infty),
		\end{align*}
		By the previous lemma and Sobolev embedding, 
		\eqref{ch4:e:extinction_final_reduction} will follow from the claims
		\begin{gather}
		\lim_{h \to 0} \| e^{-\frac{itA'(h)}{h}} f_h^1 \|_{L^\infty_t 
			L^{\frac{2d}{d-2}} ( (c, T) \times \mf{R}^d)} = 0 
			\label{ch4:e:extinction_sc}\\
		\lim_{h \to 0} \| e^{-\frac{i(t-T) A(h)}{h}} e^{-\frac{iT A'(h)}{h}} 
		f_h^1 
		\|_{L^\infty L^{\frac{2d}{d-2}} ((T, \infty) \times \mf{R}^d)} = 0. 
		\label{ch4:e:extinction_euclidean}
		\end{gather}
		
		\subsection*{Evolution of a wavepacket}
		
	We describe first the short-time behavior of the
	Schr\"{o}dinger flow on a wavepacket. For each $(x, \xi) \in B \subset T^* 
	\mf{R}^d$, let 
	$t \mapsto (x^t, 
	\xi^t)$ 
	denote the 
	bicharacteristic starting at $(x, \xi)$. 
		\begin{prop}
			\label{ch4:p:wavepacket_sc}
			Let $\psi_{(x_0, \xi_0)}^h$ be a wavepacket. 
			Then
			\begin{align*}
			e^{-\frac{itA'(h)}{h}} \psi_{(x_0, \xi_0)}^h (x) = 
			h^{-\frac{3d}{4}} 
			v \Bigl(x_0, \xi_0, t, \frac{x-x_0^t}{h^{1/2}}\Bigr)e^{\frac{i}{h} 
				[\xi_0^t (x-x_0^t) + \gamma(t, x_0, \xi_0)]},
			\end{align*}
			where $\gamma(t, x_0, \xi_0) = \int_0^t (\xi a_{\xi} - a)(x_0^s, 
			\xi_0^s) \, ds$, and $v(x_0, \xi_0, t, \cdot)$ is Schwartz 
			uniformly in $(x_0, \xi_0)$ and locally 
			uniformly in $t$. 
		\end{prop}
		
		\begin{proof}
			This was proved in~\cite{KochTataru2005} when $h = 1$. We reduce to 
			that 
			case by a change of variable.
			
			For fixed $(x_0, \xi_0)$, let $u$ be the solution to 
			\begin{align*}
			[hD_t + A'(h)] u = 0, \ u(0) = \psi_{(x_0, \xi_0)}^h,
			\end{align*}
			and define the profile $v$ by 
			\begin{align*}
			u(t, x) = h^{-\frac{3d}{4}} v\Bigl( t, \frac{x-x_0^t}{h^{1/2}} 
			\Bigr) 
			e^{\frac{i}{h}[ \xi_0^t (x-x_0^t) + \gamma(t, x_0, \xi_0)]}.
			\end{align*}
			Then $v$ solves the equation $[D_t + (a_{(x_0, \xi_0)}^h)^w(t, X, 
			D) v = 0, 
			= v(0) = \psi_{(0, 0)}^1$, where
			\begin{align*}
			a_{(x_0, \xi_0)}^h (t, x, \xi) = h^{-1} [ a(t, h^{1/2} x + x_0^t, 
			h^{1/2} 
			\xi + \xi_0^t) - h^{1/2} \xi a_{\xi} (x_0^t, \xi_0^t) - h^{1/2} x 
			a_x 
			(x_0^t, \xi_0^t) - a(x_0^t, \xi_0^t)].
			\end{align*}
			As $a_{(x_0, \xi_0)}^h$ vanishes to second order at $(0, 0)$ and 
			satisfies 
			$|\partial^\alpha_x \partial^\beta_\xi a_{(x_0, \xi_0)}^h| \le 
			c_{\alpha 
				\beta}$ for all $|\alpha| + |\beta| \ge 2$, the 
			claim follows from Lemma~\ref{ch4:l:koch-tataru_key_lemma} below.
		\end{proof}
		
		The following lemma was the key step in the proof of \cite[Proposition 
		4.3]{KochTataru2005}
		\begin{lma}
			\label{ch4:l:koch-tataru_key_lemma}
			Let $a(t, \cdot, \cdot)$ be a time-dependent symbol which vanishes 
			to 
			second order at $(0, 0)$ and satisfies $|\partial^\alpha_x 
			\partial^\beta_\xi a| \le c_{\alpha \beta}$ whenever $|\alpha| + 
			|\beta| 
			\ge 
			2$, and $S(t, s)$ be the solution operator for the evolution 
			equation
			\begin{align*}
			[D_t + a^w(t, X, D)] u = 0.
			\end{align*}
			Then $S(t, s) : \mcal{S}(\mf{R}^d) \to \mcal{S}(\mf{R}^d)$ locally 
			uniformly in $t-s$.
		\end{lma}
		
		Next we consider the long-time dynamics. 
		By the nontrapping hypothesis (G2), there exists $T = 
		O(\varepsilon^{-1})$ such 
		that for each $(x, \xi) \in B$, $|x^t| \ge 10$ for all $t \ge T$. 
		\begin{prop}
			\label{ch4:p:wavepacket_long_t}
			For each $(x_0, \xi_0) \in B$ and $t \ge T$, we have a decomposition
			\begin{align*}
			e^{-\frac{i(t-T)A(h)}{h}} e^{-\frac{iT A'(h)}{h}} \psi_{(x_0, 
			\xi_0)}^h = 
			v_1 + v_2,
			\end{align*}
			where 
			\begin{align*}
			|\partial^k v_1(t, x)| \le C_{k, N} h^{-\frac{3d}{4} - k} 
			|t|^{-d/2} \Bigl( 1 + \frac{|x-x_0^t|}{h^{1/2}|t|} \Bigr)^{-N}
			\end{align*}
			and $\|v_2\|_{H^k} = O(h^\infty)$ for all $k$.
		\end{prop}
		
		\begin{proof}
			It suffices to verify the following two assertions:
			\begin{gather}
			|\partial^k_x e^{i(t-T)h\Delta} e^{-\frac{iT A'(h)}{h}} \psi_{(x_0, 
				\xi_0)}^h (x)| \le C_{k, N} h^{-\frac{3d}{4} - k}  |t|^{-d/2} 
			\Bigl( 1 + \frac{|x-x_0^t |}{h^{1/2} |t|} \Bigr)^{-N} 
			\label{ch4:e:wavepacket_long_t_ptwise}\\
			\| (e^{-i(t-T)hA} - e^{i(t-T)h\Delta}) e^{-\frac{iT A'(h)}{h}} 
			\psi_{(x_0, 
				\xi_0)}^h \|_{H^k} \le C_{k, N} h^N. 
			\label{ch4:e:wavepacket_long_t_remainder}
			\end{gather}
			
			For $\phi$ Schwartz, by a stationary phase argument we have
			\begin{align*}
			|e^{it\Delta} \phi (x)| \le C_N \langle t \rangle^{-d/2} \Bigl 
			\langle 
			\frac{x}{2t} \Bigr
			\rangle^{-N}.
			\end{align*}
			Indeed, let $\chi$ be a smoothed characteristic function of the 
			unit ball, 
			and partition
			\begin{align*}
			e^{it\Delta} \phi(x) = \int e^{ i(x\xi - t|\xi|^2)} \chi\Bigl( \xi 
			- 
			\frac{x}{2t} \Bigr) \hat{\phi}(\xi) \, d\xi + \int e^{i(x \xi - 
			t|\xi|^2)} 
			[1 - \chi \Bigl( \xi - \frac{x}{2t} \Bigr)] \hat{\phi}(\xi) \, d\xi.
			\end{align*}
			Integrating by parts in $\xi$, the second term is bounded, for any 
			$N \ge 
			0$, by $|t|^{-N} \langle \tfrac{x}{2t} \rangle^{-N}$. By the 
			stationary 
			phase expansion, the first term equals
			\begin{gather*}
			(2 \pi i t)^{-d/2} e^{\frac{i|x|^2}{4t}} \hat{\phi}\Bigl( 
			\frac{x}{2t} 
			\Bigr) + t^{-\frac{d}{2} - 1} R(t, x),\\
			|R(t, x)| \le c \| (1 + |D_\xi|^2)^d  \chi\Bigl( \xi - \frac{x}{2t} 
			\Bigr) 
			\hat{\phi}(\xi) \|_{L^2_{\xi}} \le C_N \Bigl \langle \frac{x}{2t} 
			\Bigr \rangle^{-N}.
			\end{gather*}
			
			By the previous proposition and standard identities for the 
			Euclidean 
			propagator,
			\begin{align*}
			e^{i(t-T)h\Delta} e^{-\frac{iTA'(h)}{h}} \psi_{(x_0, \xi_0)}^h = 
			h^{-\frac{3d}{4}} e^{\frac{i}{h} [ \xi_0^t (x-x_0^t) + \gamma(t, 
			x_0, 
				\xi_0)]}  e^{i (t-T)h\Delta} \Psi_{(x_0, \xi_0)} \Bigl( 
			\frac{x-x_0^t}{h^{1/2}} \Bigr),
			\end{align*}
			where $\Psi_{(x_0, \xi_0)}$ is Schwartz uniformly in $(x_0, 
			\xi_0)$, and we 
			have used the fact that $(x_0^t, \xi_0^t) = (x_0^T + 2(t-T)\xi_0^T, 
			\xi_0^T)$ for all $t \ge T$. This 
			settles~\eqref{ch4:e:wavepacket_long_t_ptwise} for $k = 0$. When $k 
			> 0$, 
			we 
			note that differentiating the above equation brings down at worst a 
			factor 
			of $|h^{-1} \xi_0| \le c\varepsilon^{-1} h^{-1}$. 
			
			To prove~\eqref{ch4:e:wavepacket_long_t_remainder}, we note first 
			that 
			$e^{-itA}$ is uniformly bounded on each Sobolev space $H^k$. 
			Indeed, for a 
			sufficiently large $C > 0$ the 
			operators $(1-\Delta)^k (C + A)^{-k}$ and $(C + A)^k 
			(1-\Delta)^{-k}$ are 
			pseudodifferential operators of order $0$, which implies that
			\begin{align*}
			\| u\|_{H^k} = \| (1-\Delta)^{k/2} u\|_{L^2} \sim \| (C+A)^{k/2} 
			u\|_{L^2}.
			\end{align*}
			Using the Duhamel formula, triangle inequality, and the above 
			pointwise 
			estimates, we can therefore bound the left side 
			of~\eqref{ch4:e:wavepacket_long_t_remainder} by
			\begin{align*}
			\int_T^\infty \sum_{m=0}^2 \| \chi D^m e^{i(t-T)h\Delta} 
			e^{-\frac{iTA'(h)}{h}} 	\psi_{(x_0, \xi_0)}^h \|_{H^k}  \le C_N h^N 
			\int_T^\infty |t|^{-d/2} \, dt \le C_N h^N.
			\end{align*}
		\end{proof}
		
		\subsection*{The semiclassical regime}
		
		By Proposition~\ref{ch4:p:wavepacket_sc}, on bounded time intervals 
		each 
		wavepacket 
		may be regarded essentially as a particle moving under the geodesic 
		flow. 
		We 
		will 
		obtain the short-time decay~\eqref{ch4:e:extinction_sc} by showing that 
		not too 
		many wavepackets pile up near any point at any time. Heuristically, by 
		the 
		uncertainty principle the wavepackets have a broad 
		distribution of initial momenta, and slower wavepackets will 
		lag behind faster 
		ones along each geodesic. 
		
		To make this rigorous, we need to study the bicharacteristics for 
		the 
		symbol $a = g^{jk}(x) \xi_j \xi_k$. Let $(x, \xi) \mapsto (x^t(x, 
		\xi), 
		\xi^t(x, \xi))$ be the flow on $T^* \mf{R}^d$ induced by the ODE
		\begin{align*}
		\left\{\begin{array}{l} \dot{x}^t = a_{\xi} = 2 g(x^t) \xi^t, \\ 
		\dot{\xi}^t = 
		-a_x = - (\xi^t)^* 
		(\partial_x g)(x^t) \xi^t \end{array}\right. \quad (x^0, \xi^0) = (x, 
		\xi);
		\end{align*}
		The curve $t \mapsto x^t(x, \xi)$ is the geodesic starting at $x$ with 
		tangent 
		vector $g^{ab} \xi_b$, and for fixed $y$ the mapping $\eta \mapsto 
		x^1(y, 
		\eta)$ is essentially the exponential map with basepoint $y$. A 
		standard fact 
		from 
		geometry is the identity
		\begin{align}
		\label{ch4:e:geodesic_homogeneity}
		x^t(x, \xi) = x^1(x, t \xi),
		\end{align}
		which follows from the fact that $s \mapsto (x^{ts}(x, 
		\xi), 
		t\xi^{ts}(x, \xi))$ is the bicharacteristic with initial data $(x, 
		t\xi)$. 
		
		\begin{lma}
			\label{ch4:l:exp_preimage_estimate}
			Let $g$ be a nontrapping metric on $\mf{R}^d$. Then, for all $x, z 
			\in 
			\mf{R}^d$ with $|x| \le 1$ and all $0 
			\le r \le 1$, 
			\begin{align*}
			m( \{ \xi \in \mf{R}^d : |x^1(x, \xi) - z| \le r\}) \le c_{x, z}
			r,
			\end{align*}
			where $m$ denotes Lebesgue measure on $\mf{R}^d_\xi$, and the 
			constant 
			$c_{x, z}$ is locally uniformly bounded in $x$ and $z$. If also $g$ 
			is 
			Euclidean outside a compact set, then
			\begin{align*}
			m( \{ \xi \in \mf{R}^d : |x^1(x, \xi) - z| \le r\}) \le c_x (1 + 
			|z|)^{d-1} r.
			\end{align*}
		\end{lma}
		
		%
		%
		
		The basic idea is that the preimage of a small ball under the 
		exponential map will 
		always be thin in the radial
		direction, though not necessarily in the other directions. This is a 
		consequence 
		of the fact that the exponential map always has nontrivial 
		radial derivative.
		Note that 
		for 
		$z$ near $x$, the above bound can be 
		improved to $O(r^d)$ as the map $\xi \mapsto x^1(x, \xi)$
		is a diffeomorphism for $\xi$ near $0$.

		
		\begin{proof}
			Fix $x$ and $z$. For each $\xi$ we have
			\begin{align*}
			x^1(x, \xi+\zeta) = x^1(x, \xi) + (\partial_\xi x^1)\zeta + 
			r(\zeta), \ 
			|r(\zeta)| = 
			O(|\zeta|^2).
			\end{align*}
			Differentiating equation~\eqref{ch4:e:geodesic_homogeneity} in 
			$t$, 
			we have
			\begin{align*}
			\partial_\xi x^1(x, \xi) \xi = \dot{x}^1(x, \xi) = g(x^1) \xi^1(x, 
			\xi)
			\end{align*}
			which implies that
			\begin{align*}
			|\xi^1| |\partial_\xi x^1(x, \xi) \xi| \ge \xi^1 \cdot 
			(\partial_{\xi} 
			x^1 ) \xi = 2g(x^1)^{jk} (\xi^1)^j (\xi^1)^k \gtrsim |\xi^1|^2.
			\end{align*}
			Using also the fact that $g^{jk}(x)\xi_j \xi_k$ is conserved along 
			the 
			flow, it follows that
			\begin{align*}
			|\partial_\xi x^1 (x, \xi) \xi| \ge c |\xi|.
			\end{align*}
			Thus, if $\zeta_0$ is 
			such that 
			that $|r(\zeta)| \le \tfrac{c}{2} |\zeta|$ for $|\zeta| \le 
			\zeta_0$, then
			\begin{align*}
			\frac{c}{2} |\zeta| \le |x^1(x, \xi + \zeta) - x^1(x, \xi)| \le 2c 
			|\zeta|
			\end{align*}
			for all $\zeta$ parallel to $\xi$ with length at most $\zeta_0$. 
			
			Let $S_{x, z}$ denote the set on the left side in the lemma; the 
			nontrapping 
			hypothesis implies that $S_{x,z}$ is compact. By the 
			preceding considerations, the intersection of each ray $t \mapsto 
			\tfrac{t\xi}{|\xi|}$ with 
			$S_{x, z}$ has measure $O(r)$. The first inequality now follows by 
			integrating in polar 	coordinates. 
			
			Under the additional hypothesis that $g$ is flat outside a compact 
			set, 
			observe that for each $x$ there exists $R > 0$ such that 
			\begin{align}
			\label{ch4:e:unit_sphere_image}
			\sup_{|\xi|_g = 1} |x^t(x,  \xi)| - R < 2t < \inf_{ |\xi|_g = 1} 
			|x^t(x, 
			\xi)| + R,
			\end{align}
			where $|\xi|_g^2 = g^{jk}(x) \xi_j \xi_k$. Indeed, for $T$ 
			sufficiently 
			large and $t \ge T$,
			\begin{align*}
			x^t(x, \xi) = x^T(x, \xi) + 2(t-T)\xi^T(x, \xi),
			\end{align*}
			and $|\xi^T| = |\xi^T|_g = |\xi^0|_g = |\xi|_g = 1$. Set $r = 
			\sup_{|\xi|_g = 1} |x^T(x, \xi)|$ to get
			\begin{align*}
			2|t-T| - r \le |x^t(x, \xi)| \le 2|t-T| + r.
			\end{align*}
			Therefore, $S_{x, z}$ is contained in an annulus $\{|z| - R_x \le 
			|\xi| \le 
			|z| + 
			R_x\}$ with thickness $O(r)$ as before, which implies that $m(S_{x, 
			z}) \le c(1 + 
			|z|)^{d-1}r$. 
		\end{proof}
		
		\begin{rmk}
			It is clear that nontrapping hypothesis may be dispensed with from the first part provided that one restricts to a compact set of $\xi$. That is, if $g$ is any metric, then for each $R > 0$ and $x, z \in \mf{R}^d$ with $|x| \le 1$ and $0 \le r < 1$, we have
			\begin{align*}
				m( \{ \xi \in \mf{R}^d : |x^1(x, \xi) - z| \le r\} \cap \{ |\xi| \le R\} ) \le c_{R, x, z} r.
				\end{align*}

		\end{rmk}
		
		We are ready to establish the short-time 
		extinction~\eqref{ch4:e:extinction_sc}. We have
		\begin{align*}
		|e^{-\frac{itA'(h)}{h}} f_h^1(x) | \le \int_{B} |e^{-\frac{itA'(h)}{h}} 
		\psi_{(x_0, \xi_0)}^h(x) | |T_{h} \phi_h (x_0, \xi_0)| \, dx_0 d\xi_0.
		\end{align*}
		By Proposition~\ref{ch4:p:wavepacket_sc}, each $e^{-\frac{itA'(h)}{h}} 
		\psi_{(x_0, 
			\xi_0)}^h$ concentrates in a radius $h^{1/2}$ ball at $x^t(x_0, 
			\xi_0)$, 
			so the 
		integral is $O(h^\infty)$ for all $|x| \gg T \varepsilon^{-1}$. 
		
		For $|x| \lesssim T \varepsilon^{-1}$, modulo $O(h^\infty)$ we may 
		restrict the 
		integral to the region
		\begin{align*}
		\{ (x_0, \xi_0) \in B : |x^t(x_0, \xi_0) - x| < h^{\alpha}\}
		\end{align*}
		for any $\alpha < \tfrac{1}{2}$. By the rapid decay of 
		$\phi_h$,
		\begin{align*}
		|T_h \phi_h(x_0, \xi_0)| \le 
		ch^{1-\frac{5d}{4}} \int e^{\frac{-(y-x_0)^2}{2h} } 
		|\phi( 
		h^{-1} y)| \, dy \sim  c_\phi h^{1 - \frac{d}{4}} 
		\end{align*}
		Combining this with 
		Proposition~\ref{ch4:l:exp_preimage_estimate} and the 
		definition~\eqref{ch4:e:fbi_support} of $B$, 
		\begin{align*}
		\int_{B} |e^{-\frac{itA'(h)}{h}} 
		\psi_{(x_0, \xi_0)}^h(x) | |T_{h} \phi_h (x_0, \xi_0)| \, dx_0 d\xi_0 
		&\lesssim  
		h^{1 - \frac{d}{4}} h^{-\frac{3d}{4}} h^{d \theta} h^{\alpha} = h^{1 + 
		\alpha - 
			d(1 - \theta)}.
		\end{align*}
		As $\theta$ and $\alpha$ may be chosen arbitrarily close to 
		$\tfr{1}{2}$, it 
		follows that 
		\begin{align*}
		|e^{-\frac{itA'(h)}{h}} f_h^1(x)| \lesssim_\varepsilon h^{\frac{3}{2} - 
			\frac{d}{2} 
			- \varepsilon}
		\end{align*}
		for any $\varepsilon > 0$. Therefore
		\begin{align*}
		\| e^{-\frac{itA'(h)}{h}} f_h^1 \|_{L^{\frac{2d}{d-2}}} \le \| 
		e^{-\frac{itA'(h)}{h}} f_h^1 \|_{L^2}^{1 - \frac{2}{d}} \| 
		e^{-\frac{itA'(h)}{h}} 
		f_h^1 \|_{L^\infty}^{\frac{2}{d}} \lesssim_\varepsilon h^{1 - 
		\frac{2}{d} + 
			\frac{3}{d} - 1 - \varepsilon} \lesssim_{\varepsilon} 
			h^{\frac{1}{d} - 
			\varepsilon}.
		\end{align*}
		
		\begin{rmks}
			\begin{itemize}
			\item [(1)] When $g$ is the Euclidean metric, the exponential map 
			is a 
			diffeomorphism, so the bound in 
			Proposition~\ref{ch4:l:exp_preimage_estimate} 
			is $O(r^d)$. Consequently,
			\begin{align*}
			\int_{B} |e^{-\frac{itA'(h)}{h}} 
			\psi_{(x_0, \xi_0)}^h(x) | |T_{h} \phi_h (x_0, \xi_0)| \, dx_0 
			d\xi_0 
			&\lesssim h^{1 - \frac{d}{4}} h^{-\frac{3d}{4}} h^{d(\theta + 
			\alpha)} 
			\lesssim_\varepsilon h^{1 - \varepsilon}
			\end{align*}
			for any $\varepsilon > 0$, and we find that 
			\begin{align*}
			\|e^{-\frac{itA'(h)}{h}} f_h^1\|_{L^{\frac{2d}{d-2}}} 
			\lesssim_\varepsilon 
			h^{1 - \varepsilon},
			\end{align*}
			recovering modulo an arbitrarily small loss the $O(h)$ decay 
			rate predicted by the $L^{\frac{2d}{d+2}} \to L^{\frac{2d}{d-2}}$ 
			dispersive estimate for 
			the Euclidean propagator $e^{it\Delta}$. The epsilon loss 
			can be 
			avoided if, instead of truncating crudely in phase 
			space~\eqref{ch4:e:fbi_support}, we account for the 
			contribution from each dyadic annulus $\{ 
			2^{k-1}h^{1/2} \le |x| < 2^k h^{1/2} \}$, using the rapid decay of 
			each wavepacket on the $h^{1/2}$ scale. See the appendix for a more 
			precise analysis.
			
			\item [(2)] While the $L^{\frac{2d}{d-2}}$ decay was obtained by 
			interpolating between $L^2$ and $L^\infty$, this argument may be 
			adapted to yield honest 
			$L^1 \to L^\infty$ bounds that more directly parallel the 
			Burq-Gerard-Tzvetkov semiclassical dispersive 
			estimate~\cite{BurqGerardTzvetkov2004}. We relegate these 
			details to the appendix as they are not presently essential.
			
			\item [(3)] Instead of counting wavepackets, one can arrive at the 
			preceding decay rates by appealing to 
			the general parametrix of Hassell-Wunsch~\cite{HassellWunsch2005} 
			for the Schr\"{o}dinger propagator on asymptotically conic 
			manifolds. In fact, their analysis 
			shows that the bounds we 
			obtain are saturated whenever there exist 
			conjugate 
			points $z, w$ of order $d-1$ (that is, the highest 
possible). However, the wavepacket approach does let us treat both the 
semiclassical and long-time regimes in a 
			unified and fairly concrete manner, and is also more 
robust if one is only interested in semiclassics as it uses little information 
about the geometry at infinity.
			\end{itemize}
		\end{rmks}
		
		
		\subsection*{The long-time regime}
		
		Whereas the preceding discussion was essentially local due to finite 
		speed of propagation, long-time decay
		necessarily hinges on the global geometry. Recall that the time 
		parameter $T$ was chosen so that all the wavepackets are far from
		the curved region and radiate to infinity essentially under 
		the 
		Euclidean 
		Schr\"{o}dinger flow. 

		
		To prove~\eqref{ch4:e:extinction_euclidean}, use 
		Proposition~\eqref{ch4:p:wavepacket_long_t} to write
		\begin{align*}
		e^{-\frac{i(t-T)A(h)}{h}} e^{-\frac{iTA'(h)}{h}} f_h^1 = \int_B 
		v_{(x_0, 
			\xi_0)}^h(t) T_h \phi_h (x_0, \xi_0) \, dx_0 d\xi_0 + \int_B 
			r_{(x_0, \xi_0)}^h 
		(t) T_h \phi_h (x_0, \xi_0) \, dx_0 d\xi_0,
		\end{align*}
		where
		\begin{align*}
		|v_{(x_0, \xi_0)}(t, x)| \le C_N h^{-\frac{3d}{4}} |t|^{-d/2} \Bigl(1 + 
		\frac{|x-x_0^t|}{h^{1/2} |t|} \Bigr)^{-N}, \quad \|r_{(x_0, \xi_0)} 
		\|_{H^1} = 
		O(h^\infty).
		\end{align*}
		The second integral is clearly negligible in $L^\infty 
		L^{\frac{2d}{d-2}}$.
		
		To estimate the first
		integral, we proceed as in the short-time estimate, interpolating 
		between 
		$L^2$ and $L^\infty$ to exhibit decay in $L^{\frac{2d}{d-2}}$. For 
		fixed $x$, modulo 
		$O(h^\infty)$ we may restrict the integral to the region
		\begin{align*}
		B' = \{ (x_0, \xi_0) \in B: |x^t(x_0, \xi_0) - x| \le h^{\alpha} (1 + 
		|t|) \}
		\end{align*}
		for any $\alpha < \tfrac{1}{2}$. As $x^{t} = x^T + 2(t-T) \xi^T$ when 
		$t \ge 
		T$, 
		for 
		each $(x_0, \xi_0) \in B$ with $|\xi_0|_g = 1$, the ray $r \mapsto 
		(x_0, 
		r\xi_0)$ 
		intersects the above set in an interval of width $O(h^{\alpha})$. The 
		region 
		$B'$ therefore has measure at most 
		$O(h^{d\theta} 
		h^{\alpha})$, and we obtain
		\begin{align*}
		\int_{B'} |v_{(x_0,  \xi_0)}^h(t, x)| |T_h \phi_h (x_0, \xi_0)| \, dx_0 
		d\xi_0 \le c_\varepsilon h^{1 - \frac{d}{4}} h^{-\frac{3d}{4}} 
		|t|^{-d/2} 
		h^{d\theta + \alpha} = h^{1 + \alpha - d(1 - \theta)} |t|^{-d/2}.
		\end{align*}
		Hence, recalling that $\theta$ may be chosen arbitrarily close to 
		$\tfrac{1}{2}$, for any $\varepsilon > 0$ we 
		have
		\begin{align*}
		\|e^{-\frac{i(t-T)A(h)}{h}} e^{-\frac{iTA'(h)}{h}} f_h^1\|_{L^\infty 
		L^{\frac{2d}{d-2}}} 
		\lesssim T^{-1} h^{1 - \frac{2}{d}} h^{\frac{2(1+\alpha)}{d} - 
		2(1-\theta)}
		\lesssim_\varepsilon T^{-1} h^{\frac{1}{d} - \varepsilon}.
		\end{align*}
		This completes the proof of Proposition~\ref{ch4:p:extinction}.
	\end{proof}
	
	\section{Linear profile decomposition}
	\label{ch4:s:lpd}

The profile decomposition will follow from repeated application of the 
following inverse Strichartz theorem.
	\begin{prop}
		\label{ch4:p:inv_str}
		Let $\{f_n\} \subset \dot{H}^1$ be a sequence such that 
		$\|f_n\|_{\dot{H}^1} \le A$ and $\|e^{it\Delta_g} f \|_{L^\infty L^6} 
		\ge 
		\varepsilon$. Then there 
		exist a function $\phi \in \dot{H}^1$ and parameters $t_n$, 
		$x_n$, $\lambda_n$ such that after passing to a subsequence,
		\begin{align}
		\lim_{n \to \infty} G_n^{-1} e^{it_n \Delta_g} \rightharpoonup \phi 
		\text{ 
			in } \dot{H}^1(g),
		\end{align}
		where $G_n \phi = \lambda_n^{-\frac{1}{2}} \phi(\tfrac{\cdot - 
			x_n}{\lambda_n})$. Setting $\phi_n = e^{-it_n \Delta_g} G_n \phi$, 
			we have
		\begin{gather}
		\liminf_{n} \| \phi_n \|_{\dot{H}^1(g)} \gtrsim 
		\varepsilon^{\frac{9}{4}} 
		A^{-\frac{5}{4}}.
		\label{ch4:e:inv_str_profile_positive_energy}\\
		\lim_n \| f_n\|_{\dot{H}^1}^2 - \| f_n - e^{-it_n \Delta_g }G_n\phi 
		\|_{\dot{H}^1}^2 - \| e^{-it_n\Delta_g} 
		G_n\phi \|_{\dot{H}^1}^2 = 0. 
		\label{ch4:e:inv_str_kinetic_energy_decoupling}\\
		\lim_n \|f_n\|_{L^6}^6 - \| f_n - \phi_n \|_{L^6}^6 - \| \phi_n 
		\|_{L^6}^6 
		= 
		0 \label{ch4:e:inv_str_potential_energy_decoupling}
		\end{gather}
		Finally, the $t_n$ may be chosen so that either $t_n \equiv 0$ or 
		$\lambda_n^{-2} t_n \to \infty$.
	\end{prop}
	
	\begin{proof}
	
	We use the following inverse Sobolev lemma:
	\begin{lma}
		\label{ch4:l:inv_sobolev}
		If $\|f\|_{\dot{H}^1} \le A$ and $\| e^{it\Delta_g} f\|_{L^\infty L^6} 
		\ge 
		\varepsilon$, then 
		there exist $t$, $x$, $N$, such that
		\begin{align}
		|(\tilde{P}_N)^2 e^{it\Delta_g} f)(x)| \gtrsim N^{\frac{1}{2}} 
		\varepsilon^{\frac{9}{4}} 
		A^{1 - \frac{9}{4}}.
		\end{align}
	\end{lma}
	
	\begin{proof}
		A Littlewood-Paley theory argument yields the following Besov 
		refinement 
		of Sobolev embedding (see~\cite{claynotes})
		\begin{align*}
		\| e^{it\Delta_g} f\|_{L^\infty L^6} \lesssim \| f 
		\|_{\dot{H}^1}^{\frac{1}{3}} \sup_N \| (\tilde{P}_N)^2  e^{it\Delta_g} 
		f\|_{L^\infty 
		 L^6}^{\frac{2}{3}}.
		\end{align*}
		Then, using the elementary inequality
		$
		\| (\tilde{P}_N)^2 f \|_2 \lesssim N^{-1} \| \tilde{P}_N 
		f\|_{\dot{H}^1}$,
		which follows from the corresponding pointwise inequality for symbols, 
		we have
		\begin{align*}
		\varepsilon^{\frac{3}{2}} A^{-\frac{1}{2}} &\lesssim \| (\tilde{P}_N 
		)^2 e^{it\Delta_g} f\|_{L^\infty L^6} \lesssim \| (\tilde{P}_N )^2 
		e^{it\Delta_g} f \|_{L^\infty L^2}^{\frac{1}{3}} \| (\tilde{P}_N )^2 
		e^{it\Delta_g} f\|_{L^\infty L^\infty}^{\frac{2}{3}}\\
		&\lesssim N^{-\frac{1}{3}} A^{\frac{1}{3}} \| (\tilde{P}_N )^2 
		e^{it\Delta_g} f \|_{L^\infty L^\infty}^{\frac{2}{3}}.
		\end{align*}
	\end{proof}

	Select $(t_n, x_n, N_n)$ according to this lemma, and set $\lambda_n = 
	N_n^{-1}$. After 
	passing to a subsequence, we may assume $\lambda_n \to \lambda_\infty \in 
	[0, \infty]$ 
	and 
	$x_n \to x_\infty \in \mf{R}^d \cup \{ \infty\}$. We may extract a weak 
	limit
	\begin{align*}
	G_n^{-1} e^{it_n\Delta_g} f_n \rightharpoonup \phi \text{ in } 
	\dot{H}^1(g).
	\end{align*}
	As $\dot{H}^1(g)$ and $\dot{H}^1(\delta)$ have equivalent norms, their 
	duals 
	may be identified; hence the weak limit also holds in 
	$\dot{H}^1(\delta)$. 
	
	Define $\phi_n = e^{-it_n \Delta_g} G_n 
	\phi$. 
	
	We verify that this profile has positive energy. From
	Theorem~\ref{ch4:t:gaussian_heat_bounds} and the facts that $d_g(x, y) \sim 
	|x-y|$, 
	$dg= \sqrt{|g|} \, dx \sim dx$, there exist constants $c_1, c_2 > 0$ such 
	that
	\[
	N_n^{\frac{1}{2}} \varepsilon^{\frac{9}{4}} A^{-\frac{5}{4}} \le c_1 
	N_n^{3} 
	\int 
	e^{-c_2 N_n^2|x_n-y|^2} |e^{it_n\Delta_g} f_n|(y) dy.
	\] 
	Thus
	\begin{align*}
	c\varepsilon^{\frac{9}{4}} A^{-\frac{5}{4}} \le \int e^{-|y|^2} G_n^{-1} 
	|e^{it_n \Delta_g} f_n |(y) \, dy.
	\end{align*}
	As $G_n^{-1} e^{it_n \Delta_g} f_n \rightharpoonup \phi$ in $\dot{H}^1$, 
	$|G_n^{-1} e^{it_n 
		\Delta_g} f_n | \rightharpoonup |\phi|$ in $\dot{H}^1$. Indeed, by the 
	Rellich-Kondrashov theorem, the sequences $G_n^{-1} e^{it_n \Delta_g}$ and 
	$|G_n^{-1} e^{it_n \Delta_g}|$ converge to their $\dot{H}^1$ weak limits in 
	$L^2_{loc}$.
	
	Taking $n \to \infty$ in the above 
	inequality and bounding $e^{-|y|^2}$ in $\dot{H}^{-1}$ by its $L^{6/5}$ 
	norm, 
	we get
	\begin{align*}
	\varepsilon^{\frac{9}{4}} A^{-\frac{5}{4}} \le \int e^{-|y|^2} |\phi| \, dy 
	\lesssim \| |\phi| \|_{\dot{H}^1} \lesssim \| \phi 
	\|_{\dot{H^1}}.
	\end{align*}
	The claim~\eqref{ch4:e:inv_str_profile_positive_energy} follows from the 
	equivalence of 
	$\dot{H}^1(\delta)$ and $\dot{H}^1(g)$.
	
	To prove the decoupling~\eqref{ch4:e:inv_str_kinetic_energy_decoupling}, 
	write
	\begin{align*}
	\|f_n\|_{\dot{H}^1}^2 - \|f_n - e^{-it_n \Delta_g} G_n \phi 
	\|_{\dot{H}^1}^2 - 
	\| e^{-it_n \Delta_g} \phi \|_{\dot{H}^1} &= 2 \opn{Re} \langle e^{it_n 
		\Delta_g} f_n - G_n \phi, G_n \phi \rangle_{ \dot{H}^1}  \\
	&= 2\opn{Re} \langle 
	G_n^{-1} e^{it_n \Delta_g} f_n - \phi, \phi \rangle_{\dot{H}^1(g_n)},
	\end{align*}
	where $g_n(x) = g(x_n + \lambda_n x)$.
	
	To see that the right side goes to $0$, we consider two cases. 
	If $\lambda_\infty <\infty $, then by Arzel\`{a}-Ascoli, after passing to a 
	subsequence the 
	metrics $g_n$ converge boundedly and locally uniformly to some metric 
	$g_\infty$. If on 
	the 
	other hand $\lambda_\infty = \infty$, then $g_n$ converges weakly to the 
	Euclidean 
	metric 
	as $g_n(x) = \delta$ outside the shrinking balls $|x-x_n| \le 
	\lambda_n^{-1}$. 
	
	To 
	streamline 
	the presentation, in the sequel we let $g_\infty$ denote the locally 
	uniform 
	limit in the first case and $g_\infty = \delta$ in the second case.

	When $\lambda_\infty < \infty$, then
	\begin{align*}
	\langle G_n^{-1} e^{it_n \Delta_g} f_n - \phi, \phi\rangle_{\dot{H}^1(g_n)} 
	= 
	\langle G_n^{-1} e^{it_n \Delta_g} f_n - \phi, 
	\phi\rangle_{\dot{H}^1(g_\infty)} + o(1) \to 0
	\end{align*}
	by dominated convergence. 
	
	If $\lambda_\infty = \infty$, writing
	\begin{align*}
	\langle u, v \rangle_{\dot{H}^1(g_n)} = \int \nabla u \cdot 
	\overline{\nabla v} 
	\, dx + \int_{ |x-x_n| \le \lambda_n^{-1}} \langle du, dv \rangle_{g_n} \, 
	dg_n - 
	\int_{ 
		|x-x_n| \le \lambda_n^{-1}} \nabla u \cdot \overline{\nabla v} \, dx,
	\end{align*}
	we have
	\begin{align*}
	&\langle G_n^{-1} e^{it_n \Delta_g} f_n - \phi, \phi \rangle_{ \dot{H}^1 ( 
	g_n 
		) 
	} = \langle G_n^{-1} e^{it_n \Delta_g} f_n - \phi, \phi \rangle_{ \dot{H}^1 
	(\delta)} + o(1),
\end{align*}
which vanishes since $G_n^{-1} e^{it_n \Delta_g} f_n - \phi 
\rightharpoonup 0$ in $\dot{H}^1(\delta)$.

We show next that if $t_n \lambda_n^{-2}$ is bounded, then after modifying the 
profile 
$\phi$ slightly we may arrange for $t_n \equiv 0$. 

Suppose that $t_n \lambda_n^{-2} \to t_\infty$. By 
Theorem~\ref{ch4:t:convergence_of_propagators} and its corollary, we have
\begin{align*}
\phi_n = e^{-it_n\Delta_g} G_n \phi = G_n e^{-it_\infty \Delta}\phi + 
r_n, 
\quad \|r_n\|_{\dot{H}^1} = o(1).
\end{align*}
Define the modified profile $\tilde{\phi} = e^{-it_\infty \Delta}$, 
$\tilde{\phi}_n = G_n \tilde{\phi}$. Clearly 
\eqref{ch4:e:inv_str_profile_positive_energy} holds with $\tilde{\phi}_n$ in 
place 
of $\phi_n$. We claim that
\begin{align}
\label{ch4:e:inv_str_modified_profile}
G_n^{-1} f_n \rightharpoonup \tilde{\phi} \text{ in } \dot{H}^1(g).
\end{align}
Suppose $\lambda_n$ is bounded above. Passing to a subsequence, the metrics 
$g_n$ 
converge locally uniformly to some metric $g_\infty$. Then as
\begin{align*}
\langle G_n^{-1} f_n - e^{-it_\infty \Delta} \phi, \psi \rangle_{ 
	\dot{H}^1(g_n)} &= \langle f_n - G_n e^{-it_\infty \Delta} \phi, G_n 
\psi \rangle)_{\dot{H}^1(g)} + o(1)\\
&= \langle e^{it_n \Delta_g} f_n - G_n \phi, e^{it_n \Delta_g} G_n \psi 
\rangle_{ \dot{H}^1(g)} + o(1)\\
&= \langle G_n^{-1} e^{it_n \Delta_g} f_n - \phi, G_n^{-1} e^{it_n \Delta_g} 
G_n \psi \rangle_{\dot{H}^1 (g_n)} + o(1)\\
&= \langle G_n^{-1} e^{it_n \Delta_g} f_n - \phi, e^{it_\infty \Delta} 
\psi \rangle_{\dot{H}^1(g_n)} + o(1)\\
&= \langle G_n^{-1} e^{it_n \Delta_g} f_n - \phi, e^{it_\infty \Delta} 
\psi \rangle_{\dot{H}^1(g_\infty)} + o(1)\\
&= o(1),
\end{align*}
we have for all $\psi \in \dot{H}^1$
\begin{align*}
\langle G_n^{-1} f_n - e^{-it_\infty \Delta} \phi, \psi 
\rangle_{\dot{H}^1(g_\infty)} = o(1)
\end{align*}
which implies weak convergence in $\dot{H}^1(g)$ since the norms defined by 
$g_\infty$ and $g$ are equivalent. 

If instead $\lambda_n \to \infty$, then as before
\begin{align*}
\langle G_n^{-1} f_n - e^{-it_\infty \Delta} \phi, \psi 
\rangle_{\dot{H}^1(\delta)} = \langle G_n^{-1} f_n - e^{-it_\infty 
	\Delta} \phi, \psi \rangle_{\dot{H}^1(g_n)} + o(1) \to 0.
\end{align*}

Having verified the weak limit~\eqref{ch4:e:inv_str_modified_profile}, the same 
argument as before establishes the decoupling of kinetic 
energies~\eqref{ch4:e:inv_str_kinetic_energy_decoupling}.

To establish the asymptotic additivity of nonlinear 
energy~\eqref{ch4:e:inv_str_potential_energy_decoupling}, we use the 
refined Fatou lemma of Brezis and Lieb:
\begin{lma}[\cite{brezis_lieb}]
	\label{ch4:l:brezis_lieb}
	Suppose $f_n \in L^p(\mu)$ converge a.e. to some $f \in L^p(\mu)$ and 
	$\sup_{n} \|f_n\|_{L^p} < \infty$. Then
	\begin{align*}
	\int_{\mf{R}^d} \Bigl| |f_n|^p - |f_n - f|^p - |f|^p \Bigr| \, d\mu \to 0.
	\end{align*}
\end{lma}

Assume $t_n \equiv 0$. Then $\phi_n = G_n \phi$ and 
$G_n^{-1} f_n$ converges weakly in $\dot{H}^1$ to $\phi$. By Rellich-Kondrashov 
and a diagonalization argument, after passing to a subsequence we have 
$G_n^{-1} f_n \to \phi$ pointwise a.e. By a change of variable, the left side 
of~\eqref{ch4:e:inv_str_potential_energy_decoupling} is bounded by
\begin{align*}
\int \Bigl| |G_n^{-1} f_n|^6 - | G_n^{-1} f_n| - \phi|^6 - |\phi|^6 \Bigr| \, 
dg_n \le \int \, dg_\infty + \int \, d|g_n - g_\infty|,
\end{align*}
where we write $d|g_n - g_\infty|= |\sqrt{|g_n|} - 
\sqrt{|g_\infty|} | \, dx$.
The first term vanishes by the Brezis-Lieb lemma, while for the second 
integral we note that $\int |\phi|^6 \, d|g_n - g_\infty| \to 0$ and argue as 
in the proof of that lemma.

Suppose $t_n \lambda_n^{-2} \to \infty$ (the case $t_n \lambda_n^{-2} \to 
-\infty$ is similar). 
Different arguments are required 
depending on the behavior of the parameters, but in each case we conclude that
\begin{align*}
\lim_{n \to \infty} \| \phi_n\|_{L^6} = 0,
\end{align*}
which clearly implies~\eqref{ch4:e:inv_str_potential_energy_decoupling}.

If $\lambda_\infty = \infty$ or $x_\infty = \infty$, then by 
Theorem~\ref{ch4:t:convergence_of_propagators} we have
\begin{align*}
\phi_n = e^{-it_n \Delta} G_n \phi + r_n, \ \|r_n\|_{L^6} = o(1),
\end{align*}
and the decay in $L^6$ follows from the dispersive estimate for the Euclidean 
propagator.

If $0 < \lambda_\infty < \infty$ and $x_\infty \in \mf{R}^3$, then $G_n \phi 
\to 
\phi'$, and we appeal to 
Lemma~\ref{ch4:l:H1_linear_scattering} below to find $\tilde{\phi} \in 
\dot{H}^1$ 
such 
that 
\begin{align*}
\lim_{t \to \infty} \| e^{it\Delta_g} \phi' - e^{it\Delta} \tilde{\phi} 
\|_{\dot{H}^1} \to 0.
\end{align*}
We bound by the triangle inequality
\begin{align*}
\| e^{it_n \Delta_g} G_n \phi \|_{L^6} \le \| e^{it\Delta_g}(G_n\phi - \phi') 
\|_{L^6} +  \| e^{it\Delta_g} \phi' - e^{it\Delta} \tilde{\phi} \|_{L^6} + \| 
e^{it_n \Delta} \tilde{\phi} \|_{L^6}
\end{align*}
and use the Euclidean dispersive estimate and Sobolev embedding.

For the remaining case where $\lambda_\infty = 0$ and $x_\infty \in \mf{R}^3$, 
we 
invoke the extinction lemma.

\begin{lma}[Linear asymptotic completeness]
	\label{ch4:l:H1_linear_scattering}
	The limits $\lim_{t \to \pm \infty} e^{-it\Delta_\delta} e^{it\Delta_g}$ 
	exist 
	strongly in 
	$\dot{H}^1$.
\end{lma}
\begin{proof}
	Suppose first that $\phi \in C^\infty_0$. By the Duhamel formula,
	\begin{align*}
	e^{-it\Delta} e^{it\Delta_g} \phi = \phi + i \int_0^t e^{-is\Delta} 
	(\Delta_g - \Delta) e^{is\Delta_g} \phi \, ds,
	\end{align*}
	and we need to show that
	\begin{align*}
	\lim_{t \to \infty} \int_0^t e^{-is\Delta} (\Delta_g - \Delta) 
	e^{is\Delta_g} \phi \, ds
	\end{align*}
	exists in $\dot{H}^1$. We use (the dual of) the endpoint Strichartz 
	estimate 
	$e^{it\Delta_g}: L^2 \to L^2 L^6$. For $t_1 < t_2$, we have
	\begin{align*}
	&\Bigl \| \int_{t_1}^{t_2} e^{-is\Delta} (\Delta_g - \Delta) e^{is 
	\Delta_g} 
	\, ds \Bigr \|_{\dot{H}^1} \lesssim \| \nabla (\Delta_g - \Delta) 
	e^{it\Delta_g} \phi  \|_{L^2 L^{6/5} ([t_1, t_2])}\\
	&\lesssim \| \chi \nabla e^{it\Delta_g} \phi\|_{L^2 L^{6/5}} + \| \chi 
	\nabla^2 
	e^{it\Delta_g} \phi\|_{L^2 L^{6/5}} + \| \chi \nabla^3 e^{it\Delta_g} 
	\phi\|_{L^2 L^{6/5}}
	\end{align*}
	for some bump function $\chi$. Using H\"{o}lder, the equivalence 
	of Sobolev spaces, and the Strichartz inequality, each term is bounded by
	\begin{align*}
	\|\chi\|_{L^{3/2}} \| (1 - \Delta)^{3/2} e^{it \Delta_g} \phi\|_{L^2 
		L^6([t_1, t_2])}  
	&\lesssim \| (1 - \Delta_g)^{3/2} e^{it\Delta_g} \phi\|_{L^2 L^6([t_1, 
		t_2])} \lesssim 
	\| \phi\|_{H^3}.
	\end{align*}
	As $t_1, t_2 \to \infty$, the left side goes to 0. 
	Thus
	\begin{align*}
	\lim_{t \to \infty} e^{-it\Delta_\delta} e^{it\Delta_g} \phi
	\end{align*}
	exists in $\dot{H}^1$ for any $\phi \in C^\infty_0$.
	
	For general $\phi \in \dot{H}^1$, select for each $\varepsilon > 0$ some 
	$\phi_\varepsilon \in C^\infty_0$ 
	with 
	$\| \phi - \phi_\varepsilon \|_{\dot{H}^1} < \varepsilon$.
	Write $W(t) = e^{-it\Delta_\delta} e^{it\Delta_g}$, 
	\begin{align*}
	W(t) \phi = W(t) \phi_\varepsilon + W(t) (\phi - \phi_\varepsilon).
	\end{align*}
	As $W(t)$ are bounded on $\dot{H}^1$ uniformly in $t$, we have for all $t_1 
	< 
	t_2$
	\begin{align*}
	\| W(t_2) \phi - W(t_1)\phi\|_{\dot{H}^1} \le \| W(t_2) \phi_\varepsilon - 
	W(t_1) \phi_\varepsilon \|_{\dot{H}^1} + c \varepsilon;
	\end{align*}
	so $W(t) \phi$ also converges in $\dot{H}^1$.
\end{proof}
This completes the proof of Proposition~\ref{ch4:p:inv_str}.
\end{proof}

We now prepare to introduce the linear profile decomposition.
\begin{define}
	Two frames $(\lambda_n^1, t_n^1, x_n^1)$ and $(\lambda_n^2, t_n^2, x_n^2)$ 	
	are 
	\emph{orthogonal} if 
	\begin{align*}
	\fr{ \lambda_n^1 }{ \lambda_n^2} + \frac{ \lambda_n^2}{\lambda_n^1} + 
	\frac{ 
		|t_n^1 - t_n^2 |}{ \lambda_n^1 \lambda_n^2 } +
	\frac{ |x_n^1 - x_n^2 | }{ \sqrt{ \lambda_n^1 \lambda_n^2} } = \infty.
	\end{align*}
	They are \emph{equivalent} if
	\begin{align*}
	\frac{\lambda_n^1}{\lambda_n^2} \to \lambda_\infty \in (0, \infty), \ 
	\frac{t_n^1 - t_n^2}{ \lambda_n^1 \lambda_n^2 } \in \mf{R}, \
	\frac{x_n^1 - x_n^2}{ \sqrt{\lambda_n^1 \lambda_n^2}} \to x_\infty \in 
	\mf{R}^3.
	\end{align*}
\end{define}

\begin{lma}
	\label{ch4:l:equiv_orthogonal_frames}
	If frames $(\lambda_n^1, t_n^1, x_n^1)$ and $(\lambda_n^2, \lambda_n^2, 
	\lambda_n^2)$ are orthogonal, then
	\[
	(e^{-it_n^2 \Delta_g} G_n^2)^{-1} e^{-it_n^1 \Delta_g} G_n^1
	\]
	 converges in 
	weak $\dot{H}^1$ to zero. If they are equivalent, then $(e^{-it_n^2 
		\Delta_g} G_n^2)^{-1} e^{-it_n^1 \Delta_g} G_n^1$ converges strongly to 
	some injective $U_\infty: \dot{H}^1 \to \dot{H}^1$. 
\end{lma}

\begin{proof}
	Assume the frames are orthogonal, and put $t_n = t_n^2 - t_n^1$. Suppose 
	first that $|(\lambda_n^1)^{-2} t_n| \to \infty$. By passing to a 
	subsequence, we may assume $\lambda_n^1 \to 
	\lambda^1_\infty \in [0, \infty]$ and $x_n^1 \to x_\infty^1 \in \mf{R}^3 
	\cup \{ \infty\}$. Then
	\begin{align*}
	\| (G_n^2)^{-1} e^{i(t_n^2 - t_n^1))\Delta_g} G_n^1  \phi \|_{L^6} \to 0 
	\textrm{ for each } \phi \in \dot{H}^1.
	\end{align*}
	Indeed, if $\lambda_\infty^1 \in (0, 
	\infty)$ and $x_\infty^1 \in \mf{R}^3$ this follows from by 
	Lemma~\ref{ch4:l:H1_linear_scattering} and the Euclidean dispersive 
	estimate.  
	For  all other configurations of 
	$\lambda_\infty^1$ and $x_\infty^1$, we appeal to
	Theorem~\ref{ch4:t:convergence_of_propagators} to see that
	\begin{align*}
	\| e^{i(t_n^2 - t_n^1)\Delta_g} G_n^1 \phi - e^{i(t_n^2 - t_n^1)} G_n^1 
	e^{it\Delta}\phi\|_{L^6} \to 0.
	\end{align*}
	where $\Delta$ is, up to a linear change of variable, the Euclidean 
	Laplacian. 
	The decay in $L^6$ therefore follows from the Euclidean dispersive estimate.
	
	As $(G_n^2)^{-1} e^{i(t_n^2 - t_n^1)\Delta_g} G_n^1 \phi$ forms a bounded 
	sequence in $\dot{H}^1$, to determine its weak limit it suffices to test 
	against compactly supported functions. For $\psi \in C^\infty_0$, we have
	\begin{align*}
	|\langle (G_n^2)^{-1} e^{i(t_n^2 - t_n^1)\Delta_g} G_n^1 \phi, \psi 
	\rangle_{L^2}| \le \|  (G_n^2)^{-1} e^{i(t_n^2 - t_n^1)\Delta_g} G_n^1 
	\phi\|_{L^{6}} \| \psi \|_{L^{6/5}} \to 0.
	\end{align*}
	
	Assume now that $(\lambda_n^1)^{-2} (t_n^2 - t_n^1) \to t_\infty \in 
	\mf{R}$. This implies that
	\begin{align}
	\label{ch4:e:equiv_orthogonal_frames_e1}
	\frac{\lambda_n^1}{\lambda_n^2} + \frac{ \lambda_n^2 }{ \lambda_n^1 } + 
	\frac{ |x_n^1 - x_n^2| } { \sqrt{ \lambda_n^1 \lambda_n^2} } \to \infty.
	\end{align}
	As before, we may assume that $\lambda_n^1 \to \lambda_\infty^1 \in [0, 
	\infty]$ and $x_n^1 \to x_\infty \in \mf{R}^3 \cup \{ \infty \}$.
	
	If $\lambda_\infty^1 \in (0, \infty)$ and $x_\infty^1 \in \mf{R}^3$, then 
	it 
	must be the case that
	\begin{align*}
	\lim_{n \to \infty} \frac{\lambda_n^1}{\lambda_n^2} \in \{0, \infty\},
	\end{align*}
	Since the functions $f_n := e^{i(t_n^2 - t_n^1)\Delta_g} G_n^1 \phi$ form a 
	precompact subset of $\dot{H}^1$, the sequences $\nabla f_n$ and $\xi 
	\hat{f}_n(\xi)$ are tight in $L^2$. It follows that
	\begin{align*}
	\langle (G_n^2)^{-1} e^{i(t_n^2 - t_n^1)\Delta_g} G_n^1 \phi , \psi 
	\rangle_{\dot{H}^1(\delta)} = \langle 
	e^{i(t_n^2 - t_n^1)\Delta_g} G_n^1 \phi, G_n^2 \psi 
	\rangle_{\dot{H}^1(\delta)} 
	\to 0.
	\end{align*}
	From the equivalence of $\dot{H}^1(\delta)$ and $\dot{H}^1(g)$ we conclude 
	weak convergence to zero in $\dot{H}^1(g)$.
	
	For all other configurations of the limiting parameters $\lambda_\infty^1$ 
	and 
	$x_\infty^1$, we appeal to Theorem~\ref{ch4:t:convergence_of_propagators} 
	and 
	Corollary~\ref{ch4:c:concentrating_profile_strong_conv} to see 
	that
	\begin{align*}
	\| (G_n^2)^{-1} e^{i(t_n^2 - t_n^2)\Delta_g} G_n^1 \phi - (G_n^2)^{-1} e^{ 
		i(t_n^2 - t_n^1)\Delta} G_n^1 \phi \|_{\dot{H}^1} \to 0,
	\end{align*}
	where $\Delta$ is the Euclidean Laplacian modulo a linear change of 
	variable.
	Thus
	\begin{align*}
	\langle (G_n^2)^{-1} e^{i(t_n^2 - t_n^1)\Delta_g} G_n^1 \phi, \psi 
	\rangle_{\dot{H}^1(\delta)} = \langle (G_n^2)^{-1}  G_n^1 e^{it_\infty 
		\Delta}\phi, \psi \rangle_{\dot{H}^1(\delta)} + o(1),
	\end{align*}
	and under the assumption~\eqref{ch4:e:equiv_orthogonal_frames_e1}, the 
	operator $(G_n^2)^{-1} G_n^1$ converges in weak $\dot{H}^1$ to 
	zero. 
	
	Now suppose the frames are equivalent. This implies that 
	$(\lambda_n^1)^{-2} (t_n^2 - t_n^1) \to t_\infty \in \mf{R}$. If 
	$\lambda_\infty^1 \in (0, \infty)$ and $x_\infty^1 \in \mf{R}^3$, then 
	$t_n \to (\lambda_\infty^1)^{-2}t_\infty \in \mf{R}$, $\lambda_n^2 \to 
	\lambda_\infty^2 \in (0, \infty)$, $x_n^2 \to x_\infty^2 \in \mf{R}^3$, and 
	$(G_n^2)^{-1} e^{i(t_n^2 - t_n^1)\Delta_g} G_n^1$ converges strongly to 
	$(G_\infty^2)^{-1} e^{it_\infty \Delta_g} G_\infty^1 \phi$ where 
	$G_\infty^j$ is the scaling and translation operator corresponding to 
	$(\lambda_\infty^j, x_\infty^j)$. For all other values of 
	$\lambda_\infty^1$ and $x_\infty^1$, we appeal to 
	Theorem~\ref{ch4:t:convergence_of_propagators} to see that
	\begin{align*}
	(G_n^2)^{-1} e^{i(t_n^2 - t_n^1)\Delta_g} G_n^1 \to G_\infty e^{it_\infty 
		\Delta}
	\end{align*}
	where
	$G_\infty$ is the scaling and translation operator associated to the 
	parameters $(\lambda_\infty, \sqrt{\lambda_\infty} x_\infty)$, and
	\begin{align*}
	\lambda_\infty = \lim_{n \to \infty} \frac{ \lambda_n^1 }{ \lambda_n^2 }, 
	\quad x_\infty = \lim_{n \to \infty} \frac{x_n^1 - 
		x_n^2}{\sqrt{\lambda_n^1 \lambda_n^2}}.
	\end{align*}
	In both cases the limiting operator is clearly invertible.
\end{proof}

\begin{prop}[Linear profile decomposition]
	\label{ch4:p:lpd}
	Let $f_n$ be a bounded sequence in $\dot{H}^1$. After
	passing to a subsequence, there exist $J^* \in \{1, 2, \dots\} \cup
	\{\infty\}$, profiles $\phi^j$, and parameters $(\lambda_n^j, t_n^j, 
	x_n^j)$ such
	that for each finite $J$ we have a decomposition
	\begin{align*}
	f_n = \sum_{j=1}^J e^{-it_n^j \Delta_g} G_n^j \phi^j + r_n^J,
	\end{align*}
	where $G_n^j \phi(x) = (\lambda_n^j)^{-\frac{1}{2}} \phi(\tfrac{\cdot - 
		x_n}{\lambda_n^j})$, which satisfies the following properties:
	\begin{gather}
	(G_n^J)^{-1} r^J_n \rightharpoonup 0 \text{ in } \dot{H}^1 
	\label{ch4:e:lpd_wlimit}.\\
	\lim_{J \to J^*} \limsup_{n \to \infty} \| e^{it\Delta_g} r_n^J\|_{L^\infty
		L^6} = 0. \label{ch4:e:lpd_vanishing_remainder}\\
	E(f_n) = \sum_{j=1}^J E(\phi^j_n) + E(r^J_n) + o(1) \text{ as } n \to
	\infty. \label{ch4:e:lpd_energy_decoupling}\\
	\frac{\lambda_n^j}{\lambda_n^k} + \frac{\lambda_n^k}{\lambda_n^j} + \frac{ 
		|x_n^j - x_n^k|}{ \sqrt{\lambda_n^j \lambda_n^k}} + \frac{ |t_n^j - 
		t_n^k | 
	}{ \lambda_n^j \lambda_n^k} \to \infty \textrm{ for all } j \ne k. 
	\label{ch4:e:lpd_params}
	\end{gather}
	Moreover, the times $t^j_n$
	may be chosen for each $j$
	so that either $t^j_n \equiv 0$
	or $\lim_{n \to \infty} (\lambda_n^j)^{-2} t^j_n \to \pm \infty$.
\end{prop}

\begin{proof}
	We iteratively apply Proposition~\ref{ch4:p:inv_str} to construct the 
	profiles. 
	Let $r^0_n = f_n$. Passing to a subsequence, we may assume the existence of 
	the limits
	\begin{align*}
	A_J = \lim_{n \to \infty} \| r^J_n\|_{\dot{H}^1}, \quad \varepsilon_J = 
	\lim_{n \to \infty} \| e^{it\Delta_g} r^J_n\|_{L^\infty L^6}.
	\end{align*}
	If $\varepsilon_J = 0$ then stop and set $J* = J$. Otherwise, apply 
	Proposition~\ref{ch4:p:inv_str} to the sequence $r^J_n$ to obtain a set of 
	parameters $(t^{J+1}_n, x^{J+1}_n, \lambda^{J+1}_n$ and a profile
	\begin{align}
	\label{ch4:e:lpd_profile_construction}
	\phi^{J+1} = \opn{w-lim} (G_n^{J+1})^{-1} e^{it^{J+1}_n \Delta} r^J_n, 
	\quad \phi^{J+1}_n = G_n^{J+1} \phi^{J+1}.
	\end{align}
	Set  $r^{J+1}_n = r^J_n - G_n^{J+1} \phi^{J+1}$, and continue the procedure 
	replacing $J$ by $J+1$. 
	
	If $\varepsilon^J$ never equals zero, then set $J^* = \infty$. In this 
	case, the kinetic energy 
	decoupling~\eqref{ch4:e:inv_str_kinetic_energy_decoupling}, the lower bound 
	\eqref{ch4:e:inv_str_profile_positive_energy} imply 
	\begin{align*}
	A_{J+1}^2 \le A_J^2 \Bigl[ 1 - c \bigl(\frac{\varepsilon_J}{A_J} 
	\bigr)^{\frac{9}{2}} \bigr]
	\end{align*}
	which in view of the Sobolev embedding $\varepsilon_J \le c A_J$ compels 
	$\varepsilon_J \to 0$ as $J \to \infty$. 
	
	It remains to verify the 
	decoupling of parameters.

	Suppose~\eqref{ch4:e:lpd_params} failed. Choose $j < k$ with $k$ minimal 
	such 
	that the 
	frames $(\lambda_n^j, t_n^j, x_n^j)$ and $(\lambda_n^k, t_n^k, x_n^k)$ are 
	not orthogonal. 
	After passing to a subsequence, we may arrange for the frames 
	$(\lambda_n^j, 
	t_n^j, x_n^j)$, $(\lambda_n^\ell, t_n^\ell, x_n^\ell)$ to be equivalent 
	when 
	$\ell = k$ and orthogonal for $j < \ell < k$.  By construction,
	\begin{align*}
	r_n^{j-1} = e^{-it_n^j\Delta_g} G_n^j \phi^j + e^{-it_n^k \Delta_g} G_n^k 
	\phi^k + \sum_{j < \ell < k} e^{-it_n^\ell \Delta_g} G_n^\ell \phi^\ell,
	\end{align*}
	hence
	\begin{align*}
	(G_n^j)^{-1} e^{it_n^j\Delta_g} r_n^{j-1} = \phi^j + (e^{-it_n^j\Delta_g} 
	G_n^j)^{-1} e^{-it_n^k \Delta_g} G_n^k \phi^k + \sum_{j < \ell < k} 
	(e^{-it_n^j\Delta_g} G_n^j)^{-1} e^{-it_n^\ell \Delta_g} G_n^\ell \phi^\ell.
	\end{align*}

	By Lemma~\ref{ch4:l:equiv_orthogonal_frames}, $U_\infty = 
	\lim_{n \to \infty} (e^{-it_n^j\Delta_g} 
	G_n^j)^{-1} e^{-it_n^k \Delta_g} G_n^k$ is an invertible operator on 
	$\dot{H}^1$, and we obtain
	\begin{align*}
	\phi^j = \phi^j + U_\infty \phi^k.
	\end{align*}
	Thus $\phi^k = 0$, contrary to the nontriviality of the profile guaranteed 
	by~\eqref{ch4:e:inv_str_profile_positive_energy}.
\end{proof}



\section{Euclidean nonlinear profiles}
\label{ch4:s:euclidean_profiles}

\begin{prop}
	\label{ch4:p:emb_profiles}
	Let $(\lambda_n, t_n, x_n)$ be a frame such that $\lambda_n \to 
	\lambda_\infty \in [0, 
	\infty]$, $x_n \to x_\infty \in \mf{R}^3 \cup \{\infty\}$, and either $t_n 
	\equiv 0$ or 
	$\lambda_n^{-2} t_n \to \pm \infty$. Assume that the limiting parameters 
	conform to one of the 
	following scenarios:
	\begin{itemize}
		\item [(i)] $\lambda_\infty = \infty$.
		\item [(ii)] $x_\infty = \infty$.
		\item [(iii)] $x_\infty \in \mf{R}^3$, $\lambda_\infty = 0$.
	\end{itemize}
	Then, for $n$ sufficiently large, there exists a unique global solution 
	$u_n$ 
	to the equation~\eqref{ch4:e:main_eq} with $u_n(0) = e^{-it_n \Delta_g} G_n 
	\phi$ and which 
	also has finite global Strichartz norm
	\begin{align*}
	\| \nabla u_n \|_{ L^{10} L^{\frac{30}{13}} (\mf{R} \times \mf{R}^3) } \le 
	C( E(u_n(0) ) ).
	\end{align*}
	Moreover, for any $\varepsilon > 0$ there exists $\psi^\varepsilon \in 
	C^\infty_0(\mf{R} \times \mf{R}^3)$ such that
	\begin{align*}
	\limsup_{n \to \infty} \| \nabla \bigl[ u_n - G_n \psi^{\varepsilon}( 
	\lambda_n^{-2} 
	(t-t_n) ) \bigr ]\|_{ L^{10} L^{\frac{30}{13}} (\mf{R} \times \mf{R}^3) } < 
	\varepsilon.
	\end{align*}
	In particular, by Sobolev embedding the spacetime bound and approximation 
	statement hold in $Z = L^{10} L^{10}$ as well.
\end{prop}

\begin{proof}
	In each 
	regime, for $n$ large one expects the solution to the variable-coefficient 
	equation~\eqref{ch4:e:main_eq} to resemble a solution to a constant 
	coefficient 
	NLS
	\begin{align*}
	i\partial_t  u = -\Delta u + |u|^4 u
	\end{align*}
	where $\Delta$ is the Laplacian for a limiting geometry. In the 
	first two cases, the metric is the standard one on $\mf{R}^3$, 
	while in the last case the geometry is given by the constant metric 
	$g(x_\infty)$. At any rate, all finite-energy solutions to the limiting 
	equations are known to scatter~\cite{ckstt}. We use these solutions to 
	build good approximate solutions to~\eqref{ch4:e:main_eq}. As the former 
	obey 
	good
	spacetime bounds, we deduce by stability theory that the same is true 
	of the actual solutions to~\eqref{ch4:e:main_eq}. 
	
	Let $g_\infty = g(x_\infty)$ in the last case and $g_\infty = \delta$ in 
	all other cases, and denote by $\Delta$ the associated Laplacian.

	If $t_n \equiv 0$, let $v$ be the global scattering solution to the 
	constant coefficient defocusing NLS
	\begin{align}
	\label{ch4:e:const_coeff_nls}
	(i\partial_t + \Delta) v = |v|^4v
	\end{align}
	with $v(0) = \phi$. If $\lambda_n^{-2} t_n \to \pm \infty$, let $v$ instead 
	be the unique solution to the above equation such that
	\begin{align*}
	\lim_{t \to \mp \infty} \| v(t) - e^{it \Delta} \phi \|_{\dot{H}^1} 
	= 0.
	\end{align*}
	In all cases, the Euclidean solution enjoys the global in time 
	spacetime 
	bounds
	\begin{align}
	\label{ch4:e:emb_profile_euclidean_scattering}
	\| \nabla v\|_{L^2 L^6 \cap L^\infty L^2 } \le 
	C(E(\phi)) < \infty.
	\end{align}
	See~\cite[Lemma 3.11]{matador}.
	
	Fix a small parameter $0 < \theta \ll 1$, and let $\chi$ be a smooth bump 
	function 
	equal to $1$ on the unit ball. Define spatial and Fourier space cutoffs
	$\chi_n$ and $P_n$ as follows.
	
	If $\lambda_n \to 0$ and $x_n \to x_\infty \in \mf{R}^3$, let $d_n = |x_n - 
	x_\infty|$ 
	and define
	\begin{align*}
	\chi_n = \chi \Bigl( \frac{ (d_n + \lambda_n)^{1/3} (x-x_n)}{\lambda_n} 
	\Bigr), 
	\ P_n = \chi( \lambda_n (\lambda_n + d_n)^{1/6} D).
	\end{align*}
	
	If $\lambda_n \to 0$ and $|x_n| \to \infty$, let
	\begin{align*}
	\chi_n = \chi\Bigl( \frac{x-x_n}{\lambda_n^{2/3}} \Bigr), \ P_n = \chi( 
	\lambda_n^{4/3} D).
	\end{align*}
	
	If $\lambda_n \to \lambda_\infty \in (0, \infty)$ and $d_n = |x_n| \to 
	\infty$, 
	set 
	\begin{align*}
	\chi_n = \chi\Bigl(\frac{x-x_n}{d_n^{1/2}} \Bigr), \ P_n = \chi(d_n^{1/2}D).
	\end{align*}
	
	If $\lambda_n \to \infty$, set
	\begin{align*}
	\chi_n = \chi \Bigl( \frac{x-x_n}{\lambda_n^{4/3}} \Bigr), \ P_n = \chi( 
	\lambda_n^{5/6}D).
	\end{align*}
	There is of course some latitude in the choice of exponents. Define the 
	rescaled Euclidean solutions
	\begin{align*}
	v_n(t) = \lambda_n^{-1/2} v( \lambda_n^{-2}t ,  \lambda_n^{-1}(\cdot-x_n)) 
	= G_n v(\lambda_n^{-2} t).
	\end{align*}
	
	For $T > 0$ to be chosen later, set
	\begin{align*}
	\tilde{u}_n = \left\{\begin{array}{ll} \chi_n P_n v_n, & |t| \le T 
	\lambda_n^2\\
	e^{i(t-T \lambda_n^2) \Delta_g} \tilde{u}_n (T \lambda_n^2),  & t \ge T 
	\lambda_n^2\\
	e^{i(t+T\lambda_n^2)\Delta_g} \tilde{u}_n (-T \lambda_n^2), & t \le 
	-T\lambda_n^2\end{array}\right.
	\end{align*}
	
	In the next two lemmas we prepare to invoke 
	Proposition~\ref{ch4:p:stability}.
	\begin{lma}
		\label{ch4:l:embedding_profiles_eqn_error}
		\begin{align*}
		\lim_{T \to \infty} \limsup_{n \to \infty} \| \nabla [ (i\partial_t + 
		\Delta_g) \tilde{u}_n - 
		F(\tilde{u}_n) ] \|_{N} \to 0.
		\end{align*}
	\end{lma}
	\begin{lma}
		\label{ch4:l:embedding_profiles_data}
		\begin{align*}
		\lim_{T \to \infty} \limsup_{n \to \infty} \| \tilde{u}_n(-t_n) - 
		e^{-it_n \Delta_g} G_n \phi \|_{\dot{H}^1} = 0
		\end{align*}
	\end{lma}
	
	\begin{proof}[Proof of Lemma~\ref{ch4:l:embedding_profiles_eqn_error}]
		We estimate 
		separately the 
		contributions on $\{ |t| \le T \lambda_n^2\}$ and $\{ |t| > T 
		\lambda_n^2\}$.
		
		\textbf{The Euclidean window}. 
		When $|t| \le T \lambda_n^2$, write
		\begin{align*}
		\label{ch4:e:emb_profiles_short_time_error}
		&(i\partial_t + \Delta_g) \tilde{u}_n - F(\tilde{u}_n) \\
		&= \chi_n P_n 
		(i\partial_t + \Delta) v_n + (\Delta_g - \Delta) \chi_n 
		P_n v_n + [i\partial_t + \Delta, \chi_n P_n] v_n - F(\chi_n P_n 
		v_n)\\
		&= (\Delta_g - \Delta) \chi_n P_n v_n + [ \Delta, \chi_n 
		P_n] v_n + \chi_n P_n F (v_n) - F(\chi_n P_n v_n)\\
		&= (a) + (b) + (c).
		\end{align*}
		
		Consider first the scenario where $\lambda_\infty = 0$ and $x_\infty 
		\in \mf{R}^3$. 	
		We have
		\begin{align*}
		&\| \nabla (\Delta_g- \Delta) \chi_n P_n v_n \|_{L^1 
			L^2} \\
		&\lesssim \| (g^{jk} - g^{jk}(x_\infty))  \nabla \partial_j 
		\partial_k 
		\chi_n P_n v_n \|_{L^1 L^2} + \| (\partial g^{jk}) \partial_j 
		\partial_k 
		\chi_n P_n v_n \|_{L^1 L^2} \\
		&+ \| \nabla  g^{jk} \Gamma_{jk}^m 
		\partial_m \chi_n P_n v_n \|_{L^1 L^2}
		\end{align*}	
		By H\"{o}lder in time and the definition of the cutoffs, the first 
		term is bounded by
		\begin{align*}
		\Bigl( \frac{\lambda_n}{ (\lambda_n + d_n)^{1/3}} + d_n \Bigr) 
		(T\lambda_n^2) \lambda_n^{-2}  (\lambda_n + 
		d_n)^{-1/3} \| v_n\|_{ L^\infty \dot{H}^1} \le T (\lambda_n + 
		d_n)^{1/3} \| v\|_{L^\infty \dot{H}^1} \to 0.
		\end{align*}
		Similarly, the second and third terms are at most
		\begin{align*}
		(T \lambda_n^2)  \lambda_n^{-1} (\lambda_n + d_n)^{-1/6} 
		\|v_n\|_{L^\infty \dot{H}^1} \le T \lambda_n (\lambda_n + d_n)^{-1/6} 
		\|v\|_{L^\infty \dot{H}^1} \to 0.
		\end{align*}
		Hence $(a)$ is acceptable.
		
		Next, we have by H\"{o}lder and Sobolev embedding
		\begin{align*}
		\|\nabla [ \Delta, \chi_n P_n ] v_n \|_{L^1 L^2} &\le \| 
		\nabla (\Delta \chi_n) P_n v_n \|_{L^1 L^2} + \| \nabla \langle 
		\nabla \chi_n, 
		\nabla P_n v_n \rangle \|_{L^1 L^2}\\
		&\lesssim T (\lambda_n + d_n)^{\frac{2}{3}} \| \nabla v\|_{L^\infty 
			L^2} + T (d_n + \lambda_n)^{\frac{1}{6}} \| \nabla v\|_{L^\infty 
			L^2}.
		\end{align*}
		Thus $(b)$ is also acceptable.
		
		To bound the nonlinear commutator $(c)$, write
		\begin{align*}
		\chi_n P_n F(v_n) - F(\chi_n P_n v_n) = (\chi_n P_n - 1) F(v_n) + 
		F(v_n) - F(\chi_n P_n v_n).
		\end{align*}
		Estimate
		\begin{align*}
		&\| \nabla (1 - \chi_n P_n) F(v_n) \|_{L^2 L^{6/5}} \\
		&\le \| \nabla 
		(1-\chi_n) F(v_n) \|_{L^2 L^{6/5}} + \|( \nabla \chi_n )(1 - P_n) 
		F(v_n) \|_{L^2 L^{6/5}}\\
		&+ \| (1-\chi_n)  (1 - P_n) \nabla F(v_n) \|_{L^2 
			L^{6/5}}.
		\end{align*}
		By a change of variable, the last term is at most
		\begin{align*}
		\| (1-P_n)\nabla F(v_n)\|_{L^2 L^{6/5} ([-T\lambda_n^2, T 
			\lambda_n^2])} = \| (1- \tilde{P}_n) \nabla F(v) \|_{L^2 L^{6/5} ( 
			[-T, 
			T])},
		\end{align*}
		where $\tilde{P}_n = \chi( (\lambda_n + d_n)^{\frac{1}{6}} D)$, which 
		goes to zero by the estimate
		\begin{align*}
		\| \nabla F(v)\|_{L^2 L^{6/5}} \lesssim \| v\|_{L^{10} L^{10}}^4 \| 
		\nabla v\|_{L^{10} L^{\frac{30}{13}}} < C(E(\phi))
		\end{align*}
		and dominated convergence.
		
		By H\"{o}lder and Sobolev embedding, the first two terms are bounded by
		\begin{align*}
		(\lambda_n + d_n)^{1/3} \lambda_n^{-1} (T\lambda_n^2)^{\frac{1}{2}} \| 
		v_n\|_{L^\infty L^6}^5 \lesssim T^{\frac{1}{2}} (\lambda_n + d_n)^{1/3} 
		\| \nabla 
		v\|_{L^\infty L^2} \to 0
		\end{align*}
		Also, as
		\begin{align*}
		F(v_n) - F(\chi_n P_n v_n) &= (1 - \chi_n P_n)v_n \int_0^1 F_z( 
		(1-\theta)\chi_n P_n v_n + \theta v_n) \, d\theta \\
		&+ 
		\overline{(1-\chi_nP_n)v_n} \int_0^1 F_{\overline{z}} ( 
		(1-\theta)\chi_n P_n v_n + \theta v_n) \, d\theta,
		\end{align*}
		we obtain by the Leibniz rule, H\"{o}lder, Sobolev embedding, and the 
		$L^{p}$ continuity
		of the Littlewood-Paley
		projections
		\begin{align*}
		&\| \nabla [F(v_n) - F(\chi_n P_n v_n)]\|_{L^2 L^{\frac{6}{5}}} 
		\\
		&\lesssim  \| \nabla (1 - \chi_n P_n) v_n\|_{L^{10} L^{\frac{30}{13}}} 
		\| v\|_{L^{10} L^{10}}^4 + \| (1 - \chi_n P_n) v_n\|_{L^{10} L^{10}} \| 
		v_n\|_{L^{10}}^3 \| \nabla \chi_n P_n v_n\|_{L^{10} L^{\frac{30}{13}}}\\
		&\lesssim \|\nabla v\|_{L^{10} L^{\frac{30}{13}}}^4 \| \nabla (1 - 
		\tilde{\chi}_n \tilde{P}_n v)\|_{L^{10} L^{\frac{30}{13}}},
		\end{align*}
		where $\tilde{P}_n = \chi( (\lambda_n + d_n)^{\frac{1}{6}} D)$ and 
		$\chi_n = \chi( (\lambda_n + d_n) x)$. By dominated convergence, this 
		also vanishes as $n \to \infty$.
		
		Now we consider the case where $\lambda_n \to \infty$, and estimate the 
		errors $(a)$, $(b)$, and $(c)$ as before.
		
		Since $\Delta_g - \Delta = (g^{jk} - \delta^{jk}) \partial_j 
		\partial_k - g^{jk} \Gamma^m_{jk} \partial_m$, we have
		\begin{align*}
		\| \nabla (\Delta_g - \Delta) \chi_n P_n v_n \|_{L^2 L^{6/5}} 
		&\le 
		\sum_{j=1}^3 \| \tilde{\chi} \nabla^j \chi_n P_nv_n \|_{L^2 L^{6/5}}
		\end{align*}
		where $\tilde{\chi}$ is a spatial cutoff. As the $v_n$ 
		are being rescaled to low frequencies, the terms with 
		the 
		fewest derivatives applied to $v_n$ are least favorable. Estimate
		\begin{align*}
		\| \tilde{\chi} \nabla \chi_n P_n v_n \|_{L^2 L^{6/5}} &\le \| 
		\tilde{\chi} 
		(\nabla 
		\chi_n) P_n 
		v_n 
		\|_{L^2 L^{6/5}} + \| \tilde{\chi} \nabla P_n v_n\|_{L^2 L^{6/5}}\\
		&\lesssim \lambda_n^{-\frac{4}{3}} (T \lambda_n^2)^{\frac{2}{5}} \| 
		\chi\|_{L^{\frac{15}{11}}} \| P_nv_n \|_{L^{10} L^{10}} + \| 
		\tilde{\chi} 
		\|_{L^{\frac{6}{5}}} \| \nabla P_n v_n\|_{L^2 L^\infty}\\
		&\lesssim T^{\frac{2}{5}} \lambda_n^{-\frac{8}{15}} \| v\|_{L^{10} 
		L^{10}} 
		+ 
		\lambda_n^{-\fr{5}{12}} \| \nabla v\|_{L^2 L^6},
		\end{align*}
		which is acceptable by~\eqref{ch4:e:emb_profile_euclidean_scattering}. 
		Also,
		\begin{align*}
		&\| \nabla [ \Delta, \chi_n P_n] v_n \|_{L^1 L^2} \le \| \nabla 
		(\Delta \chi_n) P_n v_n\|_{L^1 L^2} + \| \nabla \langle \nabla 
		\chi_n, \nabla P_n v_n \rangle \|_{L^1L^2}\\
		&\lesssim \| (\nabla^3 \chi_n) P_n v_n\|_{L^1 L^2} + \| (\nabla^2 
		\chi_n) 
		\nabla 
		P_n v_n\|_{L^1 L^2} + \| (\nabla \chi_n) \nabla^2 P_n v_n\|_{L^1 L^2}\\
		&\lesssim 
		(\lambda_n^{-\frac{4}{3}} )^2 (T \lambda_n^2) \| \nabla P_n v_n 
		\|_{L^\infty L^2} + \chi_n^{-\frac{4}{3}} (T\lambda_n^2) 
		\lambda_n^{-\frac{5}{6}} \| 
		\nabla P_n v_n\|_{L^\infty L^2}\\
		&\lesssim T (\lambda_n^{-\frac{2}{3}} + \lambda_n^{-\frac{1}{6}} ) \| 
		\nabla v\|_{L^\infty L^2}.
		\end{align*}	
		
		Finally, the same argument as above yields
		\begin{align*}
		\|\nabla [ \chi_n P_n F(v_n) - F(\chi_n P_n v_n)] \|_{L^2 
		L^{\frac{6}{5}} } 
		\to 0.
		\end{align*}
		
		The remaining cases $\lambda_\infty < \infty$, $|x_n| \to \infty$ are 
		dealt 
		with similarly.

		\textbf{The long-time contribution}. 
		When $t \ge T \lambda_n^2$, 
		\begin{align*}
		\| \nabla [ (i\partial_t + \Delta_g) \tilde{u}_n - 
		F(\tilde{u}_n)]\|_{L^2 L^{6/5}}
		&\lesssim \| \tilde{u}_n\|_{L^{10} L^{10} ( (T\lambda_n^2, \infty) 
			)}^4 
		\| (-\Delta_g)^{1/2} \tilde{u}_n \|_{L^{10} L^{\frac{30}{13}}}.
		\end{align*}
		The last norm on the right is bounded by Strichartz and energy 
		conservation. To estimate the $L^{10}$ norm, let $v_+ \in 
		\dot{H}^1$ be 
		the forward scattering state for the Euclidean solution $v$, 
		defined by
		\begin{align*}
		\lim_{t \to \infty} \| v(t) - e^{it\Delta} v_+ 
		\|_{\dot{H}^1} = 
		0,
		\end{align*}
		and write $v_{+n} = G_n v_+$. Then
		\begin{align*}
		&\tilde{u}_n (t) = e^{i(t-T\lambda_n^2)\Delta_g} \chi_n P_n 
		v_n(T\lambda_n^2) \\
		&= e^{it\Delta_g} (v_{+n}) + e^{i(t-T\lambda_n^2)\Delta_g} [ 
		e^{iT\lambda_n^2 \Delta} (v_{+n}) - e^{iT \lambda_n^2 
			\Delta_g} 
		(v_{+n}) ]\\
		&+ e^{i(t-T \lambda_n^2)\Delta_g} (\chi_n P_n - 1) v_n 
		(T\lambda_n^2) + 
		e^{i(t-T\lambda_n^2)\Delta_g}[ v_n(T \lambda_n^2) - e^{iT 
			\lambda_n^2 
			\Delta} ( v_{+n})],
		\end{align*}
		and we see that if $T$ is sufficiently large, each term becomes 
		acceptably small for $n$ large. Indeed, by interpolating 
		Theorem~\ref{ch4:t:convergence_of_propagators} or 
		Proposition~\ref{ch4:p:extinction} with a Strichartz estimate,
		\begin{align*}
		\lim_{T \to \infty}\lim_{n \to \infty} \| e^{it\Delta_g} v_{+n} 
		\|_{L^{10} L^{10}( (T\lambda_n^2, \infty) } = 0.
		\end{align*}
		The remaining terms are also acceptable due to Strichartz, 
		Theorem~\ref{ch4:t:convergence_of_propagators}, dominated convergence, 
		and 
		the definition of the scattering state $v_{+n}$. 	
	\end{proof}
	
	\begin{proof}[Proof of Lemma~\ref{ch4:l:embedding_profiles_data}]
		If $t_n \equiv 0$ then there is nothing to prove. So suppose 
		$\lambda_n^{-2} t_n \to \infty$. Recall that by definition,
		\begin{align*}
		\lim_{t \to -\infty} \| v(t) - e^{it\Delta} \phi\|_{\dot{H}^1} = 
		0.
		\end{align*}
		Referring to the definition of $\tilde{u}_n$, for $n$ large enough
		\begin{align*}
		\tilde{u}_n(-t_n) &= e^{-it_n \Delta_g} e^{iT\lambda_n^2 \Delta_g} 
		\chi_n P_n v_n(-T\lambda_n^2) = e^{-it_n \Delta_g} e^{iT \lambda_n^2 
			\Delta_g} G_n v(-T) + r_n\\
		&= e^{-it_n\Delta_g} e^{iT\lambda_n^2 \Delta_g} e^{-iT 
			\lambda_n^2\Delta_g}G_n \phi + r_n\\
		&= e^{-it_n\Delta_g} G_n \phi + r_n,
		\end{align*}
		where, by Theorem~\ref{ch4:t:convergence_of_propagators} and 
		Corollary~\ref{ch4:c:concentrating_profile_strong_conv}, in each line
		\begin{align*}
		\lim_{T \to \infty} \limsup_{n \to \infty} \| r_n\|_{ \dot{H}^1} = 0.
		\end{align*}
	\end{proof}
	
	By the preceding lemmas, for $T$ large enough and $n$ large, the function 
	$\tilde{u}_n(t-t_n, x)$ is a good approximate solution 
	to~\eqref{ch4:e:main_eq} 
	in the sense of Proposition~\ref{ch4:p:stability}. Thus for any 
	$\varepsilon > 
	0$ and all $n$ sufficiently large, there is a unique global solution $u_n$ 
	to~\eqref{ch4:e:main_eq} with
	\begin{align*}
	\|  u_n \|_{Z(\mf{R})} + \| \nabla u_n \|_{L^{10} L^{\frac{30}{13}} (\mf{R} 
		\times \mf{R}^3)  }\le 
	C(E(u_n(0))).
	\end{align*}

	Finally, for any $\varepsilon > 0$ there exists $\psi^\varepsilon \in 
	C^\infty(\mf{R} \times \mf{R}^3)$ such that $\| \nabla(v - 
	\psi^\varepsilon) \|_{L^{10} L^{\frac{30}{13} (\mf{R} \times \mf{R}^3)}} < 
	\varepsilon$. In view of the definition of 
	$\tilde{u}_n$ and the fact that, as proved above,
	\begin{align*}
	\lim_{T \to \infty} \limsup_{n \to \infty}\| \tilde{u}_n 
	\|_{Z([-T\lambda_n^2, T\lambda_n^2]^c)} = 0,
	\end{align*}
	another application of Proposition~\ref{ch4:p:stability} yields
	\begin{align*}
	\lim_{T \to \infty} \limsup_{n \to \infty} \| \nabla u_n \|_{L^{10} 
		L^{\frac{30}{13}} ( [-T\lambda_n^2, T\lambda_n^2]^c \times \mf{R}^3)} = 
		0.
	\end{align*}
	Therefore
	\begin{align*}
	\lim_{T\to \infty} \limsup_{n \to \infty}\| \nabla[ \tilde{u}_n - G_n 
	\psi^\varepsilon(\lambda_n^{-2}t)] \|_{ L^{10} L^{\frac{30}{13}} (\mf{R} 
		\times \mf{R}^3) } \lesssim \varepsilon,
	\end{align*}
	as required.
\end{proof}

\section{Nonlinear profile decomposition and global wellposedness}
\label{ch4:s:nlpd}

In this section, we show that failure of Theorem~\eqref{ch4:t:main_thm} would 
imply 
the existence of an ``almost-periodic" solution in the sense that it remains in 
a 
precompact subset of $\dot{H}^1$. This will already preclude finite time blowup 
and hence prove the global existence part of the theorem. In the next section, 
we rule out almost-periodic solutions under a smallness assumption on the 
metric 
and obtain global spacetime bounds in that setting.

Although we have worked mainly with the $Z = L^{10} L^{10}$ norm, in 
the sequel we 
shall also need 
the stronger 
norm 
\begin{align*}
Y = L^{10} \dot{H}^{1, \frac{30}{13}}.
\end{align*}

Let
\begin{gather*}
\Lambda(E) : \sup \{ \|u\|_{Z(\mf{R})} : E(u) \le E, \ u \text{ solves } 
\eqref{ch4:e:main_eq} \}\\
E_c = \sup \{ E : \Lambda(E) < \infty\}.
\end{gather*}
\begin{gather*}
\Lambda'(E) : \sup \{ \|u\|_{Z(I)} : |I| \le 1, \ E(u) \le E, \ u \text{ 
	solves} \eqref{ch4:e:main_eq} \}\\
E_c' = \sup \{ E : \Lambda'(E) < \infty \}.
\end{gather*} 
The small data theory implies that $E_c, E_c' > 0$. Global existence 
(resp. scattering) would follow if we show that $E'_c < \infty$ (resp. $E_c < 
\infty$).

\begin{prop}
	\label{ch4:p:palais-smale}
	Suppose $E_c < \infty$. Let $u_n$ be a sequence of solutions 
	to~\eqref{ch4:e:main_eq} with $E(u_n) \to 
	E_c$ such that for some sequence of times $t_n$, $\| u_n\|_{Z( (-\infty, 
		t_n))} \to \infty$ and $\| u_n \|_{ Z( (t_n, \infty) )} \to \infty$. 
		Then 
	some subsequence of $u(t_n)$ converges in $\dot{H}^1$.
\end{prop}

The method of proof yields an analogous statement for global existence:
\begin{prop}
	\label{ch4:p:palais-smale_unit_time}
	Suppose $E_c' < \infty$, and fix any $\delta > 0$. Let $u_n$ be a 
	sequence of solutions 
	to~\eqref{ch4:e:main_eq} with $E(u_n) \to 
	E_c$ such that for some sequence of times $t_n$, $\| u_n\|_{Z( (t_n - 
		\delta, 
		t_n))} \to \infty$ and $\| u_n \|_{ Z( (t_n, t_n + \delta) )} \to 
	\infty$. Then 
	some subsequence of $u(t_n)$ converges in $\dot{H}^1$.
\end{prop}
We prove the global-in-time proposition; as the reader may verify, a nearly 
identical argument yields the  
local-in-time version.

\begin{proof}[Proof of Prop. \ref{ch4:p:palais-smale}]
	By translating in time, we may assume without loss that $t_n \equiv 0$. 
	After passing to a subsequence, we obtain a decomposition
	\begin{align}
	\label{ch4:e:palais-smale_lpd}
	u_n(0) = \sum_{j=1}^J e^{-it_n^j\Delta_g} G_n^j \phi^j + r_n^J
	\end{align}
	into asymptotically independent profiles with the properties described in
	Proposition~\ref{ch4:p:lpd}.  In particular, 
	\begin{gather}
	\lim_{n \to \infty} \Bigl[ E(u_n(0)) - \sum_{j=1}^J E(e^{-it_n^j \Delta_g} 
	G_n^j \phi_n^j) - E(r_n^J) \Bigr] = 0, 
	\label{ch4:e:palais-smale_energy_decoupling}\\
	\lim_{J \to J^*} \limsup_{n \to \infty} \| e^{it\Delta_g} r_n^J\|_{L^\infty 
		L^6} = 0. \label{ch4:e:palais-smale_vanishing_remainder}
	\end{gather}
	
	\begin{lma} 
		\label{ch4:l:palais-smale_crit_profile}
		There exists $j$ such that $\limsup_{n \to \infty} 
		E(e^{-it_n^j\Delta_g} G_n^j \phi^j) = E_c$. 
	\end{lma}
	
	This will be proved below using a nonlinear profile decomposition. For the 
	moment, we assume the result and observe how it yields the proposition. By 
	the 
	lemma, $u_n(0)$ takes the form
	\begin{align*}
	u_n(0) = e^{-it_n\Delta_g} G_n \phi + r_n
	\end{align*}	
	where $\|r_n\|_{\dot{H}^1} \to 0$ and $G_n$ is associated to some frame 
	$(\lambda_n, x_n)$. After passing to a subsequence, we may assume that 
	$\lambda_n \in \lambda_\infty \in [0, \infty]$, $x_n \to x_\infty \in 
	\mf{R}^3 
	\cup \{\infty\}$, and $\lambda_n^{-2} t_n \to t_\infty \in \mf{R} \cup \{ 
	\pm 
	\infty\}$. 
	
	We claim that $\lambda_\infty \in (0, \infty)$, $x_\infty 
	\in \mf{R}^3$, and $t_\infty = 0$, which would clearly imply that $u_n(0)$ 
	converges in $\dot{H}^1$. If either of the first two statements failed, 
	Proposition~\ref{ch4:p:emb_profiles} would imply that  $\limsup_{n \to 
	\infty} 
	\| 
	u_n\|_{Z(\mf{R})} < \infty$, contrary to the assumptions on $u_n$. Thus 
	$\lambda_\infty \in (0, \infty)$ and $x_\infty \in \mf{R}^3$. If $t_n \to 
	\infty$, then, writing $G_\infty$ for the operator associated to the 
	parameters 
	$(\lambda_\infty, x_\infty)$, we have by the Strichartz estimate
	\begin{align*}
	\|(-\Delta_g)^{1/2} e^{it\Delta_g} u_n(0) \|_{L^{10} L^{\frac{30}{13}} ( 
		(-\infty, 0) \times \mf{R}^3)} \le \| (-\Delta_g)^{1/2} e^{it\Delta} 
		G_\infty 
	\phi \|_{ L^{10} L^{\frac{30}{13}}( (-\infty, -t_n) \times \mf{R}^3)} + 
	o(1) 
	\to 0,
	\end{align*}
	which implies by the small data theory that $\lim_{n \to \infty} \| 
	u_n\|_{Z(-\infty, 0)} = \infty$, contrary to the hypothesis that $u_n$ 
	blows up 
	forwards and backwards in time.
\end{proof}

\begin{cor}
	\label{ch4:c:soliton-like}
	If $E_c < \infty$, then there exists a global solution $u_c$ 
	to~\eqref{ch4:e:main_eq} with $E(u_c) = E_c$ and $\| u\|_{Z( (-\infty, 0])} 
	= 
	\| u\|_{Z( [0, \infty) )} = \infty$. Moreover, $u$ is 
	\emph{almost-periodic} in the sense that $\{u_c(t) : t \in \mf{R}\}$ is 
	precompact in $\dot{H}^1$.
\end{cor}

\begin{proof}
	Let $u_n$ be a sequence of solutions with $E(u_n) \to E_c$ and $\|u_n\|_{Z} 
	\to \infty$. Choose $t_n$ such that $\|u_n\|_{(-\infty, t_n]} = \| 
	u_n\|_{[t_n, \infty)}$. By the previous proposition, there exists $\phi \in 
	\dot{H}^1$ such that after passing to a subsequence, $u_n(t_n) \to \phi$ in 
	$\dot{H}^1$. Let $u_c$ be the maximal solution with $u_c(0) = 0$. 
	Proposition~\ref{ch4:p:palais-smale} and the 
	stability theory imply that $u_c$ is global and blows up forwards and 
	backwards in time. Another application of the previous proposition yields 
	precompactness of the orbit $\{u(t) : t \in \mf{R}\}$ in $\dot{H}^1$.
\end{proof}

An immediate consequence of Proposition~\ref{ch4:p:palais-smale_unit_time} and 
the 
stability theory is that 
the equation~\eqref{ch4:e:main_eq} is globally wellposed.
\begin{cor}
	\label{ch4:c:gwp}
	Under the hypotheses of Theorem~\ref{ch4:t:main_thm}, solutions 
	of~\eqref{ch4:e:main_eq} are global in time.
\end{cor}
\begin{proof}
	If $E_c'< \infty$, then there exists a sequence of solutions $u_n$ with 
	\[
	E(u_n) \to E_c', \quad \|u_n\|_{Z( ( -\fr{1}{2}, 0))}, \ \|u_n\|_{Z ( (0, 
		\frac{1}{2}) )} \to \infty.
	\]
	 By 
	Proposition~\ref{ch4:p:palais-smale_unit_time}, 
	some subsequence of $u_n(0)$ converges to some $\phi \in \dot{H}^1$. Let 
	$u_c : (T_-, T_+) \times \mf{R}^3 \to \mf{C}$ be the maximal-lifespan 
	solution with $u_c(0) = \phi$. By the Proposition~\ref{ch4:p:stability}, 
	$u_c$ 
	has infinite $Z$-norm on $(-\tfrac{1}{2}, 0)$ and $(0, \tfrac{1}{2})$, so 
	the interval of definition for $u_c$ is contained in $(-\tfrac{1}{2}, 
	\tfrac{1}{2})$. As in the previous corollary, the solution curve $u_c(t)$ 
	is precompact in $\dot{H}^1$, so along some sequence of times $t_n \to T_+$ 
	the functions $u_c(t_n)$ converge to some $\phi_+$ in $\dot{H}^1$. But then 
	we may use a local solution $u_+$ with $u_+(0) = \phi_+$ and the stability 
	theory to continue $u_c$ to a larger time interval $(-T, T_+ + \delta)$, 
	contradicting its maximality.
\end{proof}

We prove Proposition~\ref{ch4:p:palais-smale} in the remainder of this section. 
While the argument is fairly standard, involving a nonlinear 
profile 
decomposition, some remarks are warranted concerning the 
interaction  between nonlinear profiles and the linear evolution of the 
remainder in the 
decomposition. This is normally controlled using local smoothing, 
which prevent high-frequency linear solutions from lingering in a confined 
region. In Euclidean space, the local smoothing estimate takes the form
\begin{align*}
\| \nabla e^{it\Delta_{\mf{R}^3}}\phi \|_{L^2(\mf{R} \times \{ |x| \le R\})} 
\lesssim R^{1/2} \| 
\phi\|_{\dot{H}^{1/2}}.
\end{align*}
However, most existing local smoothing estimates on manifolds work at 
a fixed spatial scale, and since the metric is not scale-invariant, it is not 
obvious how the constants depend on the 
size of the physical localization. 


The following lemma
is analogous to Lemma~7.1 of Ionescu-Pausader concerning NLS on the torus 
\cite{IonescuPausader2012}, although the proof 
there is quite different due to trapping.

Let $\chi(\lambda)$ be a smooth function on the real line equal to $1$ when 
$\lambda \le 1$ and vanishing when $\lambda \ge 1.2$, and define the spectral 
multipliers $P_{\le N} = \chi(\sqrt{-\Delta_g}/N)$. By 
Theorem~\ref{ch4:t:multiplier}, these satisfy the 
Littlewood-Paley estimates of Proposition~\ref{ch4:p:bernstein} except when $p 
= 1$ or $p = \infty$ (which will not be needed).

\begin{lma}
	\label{ch4:l:loc_smoothing}
	For any $R, N, T> 0$, $B \ge 1$, and 
	$(t_0, x_0) \in \mf{R} \times \mf{R}^3$,
	\begin{align*}
	\| \nabla e^{it\Delta_g} P_{>BN} \phi \|_{L^2( |t-t_0| \le T N^{-2}, \ | 
		x-x_0| \le R N^{-1} 
		)} \le C  B^{-1/2} N^{-1} \| \phi \|_{\dot{H}^1}.
	\end{align*}
\end{lma}

\begin{proof}
	By invariance under time translation, we may take $t_0 = 0$. We adapt the 
	standard proof of local smoothing on 
	Euclidean space via a Morawetz multiplier but need to deal with error terms 
	arising from 
	the background curvature. These will be controlled by a separate local 
	smoothing 
	estimate 
	adapted to the metric.
	
	Let $a(x) = \langle x \rangle$. We compute (all derivatives are 
	partial 
	derivatives)
	\begin{alignat*}{4}
	&\partial a = \frac{ x }{ \langle x \rangle}, &\quad & \partial^2 a = 
	\langle x 
	\rangle^{-3} P_r + \langle x \rangle^{-1} P_\theta,\\	
	&\Delta a \ge \frac{3}{\langle x \rangle^{3} }, &\quad &\Delta^2 a= 
	-\frac{15}{\langle x 
		\rangle^7}\\
	&|\partial^k a | \le \frac{c_k}{\langle x \rangle^{k-1}},
	\end{alignat*}
	where $P_r$ and $P_\theta = I - P_r$ are the radial and tangential 
	projections, respectively.
	
	Now write $D = d + \Gamma$ for the Levi-Civita covariant derivative, 
	where 
	$\Gamma$ are the Christoffel symbols for the metric $g$; by our assumptions 
	on $g$, $\Gamma$ is supported in the unit ball.
	
	If $u$ is a solution to the equation
	\begin{align*}
	(i\partial_t + \Delta_g)u = \mu |u|^4 u, \ \mu \in \mf{R},
	\end{align*}
	define the Morawetz action
	\begin{align*}
	M(t) = \int_{\mf{R}^3} a(x) |u(t, x)|^2 \, dg.
	\end{align*}
	Then as in the Euclidean setting, we have 
	\begin{align*}
	\partial_t M(t) = 2 \opn{Im} \int \overline{u} D^\alpha a D_\alpha u \, dg,
	\end{align*}
	and the Morawetz identity
	\begin{align}
	\label{ch4:e:morawetz}
	\partial_t^2 M = 4 \opn{Re} \int (D^2_{\alpha \beta} a) D^\alpha 
	\overline{u} 
	D^\beta u \, dg -  \int (\Delta_g^2 a) |u|^2 \, dg + \frac{4\mu}{3} \int 
	(\Delta_g a) |u|^6 \, dg.
	\end{align}
	
	We apply this identity with $\mu = 0$ and $u= e^{it\Delta_g} P_{>BN} \phi$; 
	later on we will use this when $\mu = 1$.
	For $N > 0$ and $x_0 \in \mf{R}^3$, let $a_{N, x_0}(x) = a\bigl ( N 
	(x-x_0) \bigr )$. We compute
	\begin{align*}
	&D a_{N, x_0} = N (\partial a)\Bigl( N(x-x_0) \Bigr)\\
	&D^2_{\alpha \beta} a_{N, x_0}  = N^{2} (\partial_\alpha \partial_\beta a) 
	\Bigl( 
	N(x-x_0) 
	\Bigr) - N \Gamma_{\alpha \beta}^\mu  \partial_\mu a \Bigl( 
	N(x-x_0) \Bigr).
	\end{align*}
	Then
	\begin{align*}
	\Delta_g^2 a_{N, x_0} &= g^{\alpha \beta} (\partial_\alpha \partial_\beta - 
	\Gamma_{\alpha \beta}^\mu \partial_\mu) g^{\alpha'\beta'} ( 
	\partial_{\alpha'} \partial_{\beta'} - \Gamma_{\alpha'\beta'}^{\mu'} 
	\partial_{\mu'} ) 
	a_{N, x_0}\\
	&= N^{4} g^{\alpha \beta} g^{\alpha'\beta'} (\partial_\alpha 
	\partial_\beta \partial_{\alpha'} \partial_{\beta'}  a) \Bigl( 
	N(x-x_0) \Bigr) + (N^{3} P_3 a + N^{2} P_2 a + N^{1} P_1 a)\Bigl( 
	N(x-x_0) \Bigr),
	\end{align*}
	where $P_k$ is a differential operator of order $k$ with coefficients 
	supported in the unit ball; hence
	\begin{align*}
	\Bigl| P_k a \Bigl(N(x-x_0) \Bigr) \Bigr| \le c_k \opn{1}_{\{ |x| \le 
		1\}}(x) 
	\Bigl \langle N(x-x_0) \Bigr \rangle^{1 - k}.
	\end{align*}
	
	Inserting these bounds into the Morawetz identity and integrating in time 
	over the interval $|t| \le TN^{-2}$, we obtain
	\begin{align*}
	\iint_{|t| \le TN^{-2}} \langle N(x-x_0) \rangle^{-3} |\nabla u|^2 \, dx dt 
	&\lesssim N^{-1} 
	\iint 
	\mr{1}_{ \{ |t| 
		\le 
		TN^{-2}, |x| \le 1 \} } (t, x) (|\nabla u|^2  + |u|^2 )\, dx dt \\
	&+ N^{2} \iint_{|t| \le TN^{-2} } |u|^2 \, dx dt + N^{-1} \| u\|_{L^\infty 
		L^2}  \| u\|_{ L^\infty \dot{H}^1}.
	\end{align*}
	By the unitarity of the propagator and the spectral localization of $u$, we 
	have 
	\[
	\|u\|_{L^\infty L^2} 
	\|u\|_{L^\infty \dot{H}^1} \lesssim (BN)^{-1} \| \phi\|_{\dot{H}^1}^2.
	\] 
	Also, 
	by H\"{o}lder in time, the second term on the right may be bounded by
	\begin{align*}
	T \| u\|_{L^\infty L^2}^2 &\lesssim T (BN)^{-2} \| \phi\|_{\dot{H}^1}^2.
	\end{align*}
	Finally, the first term on the right is controlled by the following 
	scale-$1$ 
	local 
	smoothing 
	estimate of Rodnianski and Tao~\cite{RodnianskiTao2007}
	\begin{align*}
	\| \langle x \rangle^{-\frac{1}{2} - \sigma} \nabla e^{it\Delta_g} \phi\|_{ 
		L^2(\mf{R} \times \mf{R}^3)}  + \| \langle x \rangle^{-\frac{3}{2} - 
		\sigma} u 
	\|_{L^2 (\mf{R} \times \mf{R}^3)} &\lesssim_\sigma \| 
	\phi\|_{\dot{H}^{1/2}}, \ 
	\sigma > 0,
	\end{align*}
	who strengthened an earlier local-in-time version by Doi~\cite{Doi1996}. 
	Summing up, we obtain
	\begin{align*}
	\| \nabla u\|_{L^2(  \{ |t| \le TN^{-2}, |x-x_0| \le RN^{-1} \} ) }^2 
	&\lesssim 
	(B^{-1} 
	N^{-2} + T (BN)^{-2} ) \| \phi\|_{\dot{H}^1}^2.
	\end{align*}	
\end{proof}

\begin{proof}[Proof of Lemma~\ref{ch4:l:palais-smale_crit_profile}]
	Assuming that the claim fails, the asymptotic additivity of energy implies 
	the existence of some $\delta > 0$ such that $\limsup_{n \to \infty} E( 
	e^{-it_n^j \Delta_g} G_n^j \phi^j) \le E_c - \delta$ for all $j$. We shall 
	deduce that
	\begin{align}
	\label{ch4:e:palais-smale_contradiction}
	\limsup_{n \to \infty} \| u_n \|_{Z(\mf{R})} \le C(E_c, \delta) < \infty,
	\end{align}
	which contradicts the hypotheses on $u_n$.
	
	For each $j \le J$, let $u_n^j$ be the maximal-lifespan nonlinear solution 
	with $u_n^j(0) = e^{-it_n^j \Delta_g} G_n^j \phi^j$; by the 
	definition of $E_c$, for all $n$ sufficiently large we have 
	$\|u_n^j\|_{Z(\mf{R})} \le C$, hence $\|u_n^j\|_{Y (\mf{R})} \le C'$. Define
	\begin{align*}
	\tilde{u}_n^J = \sum_{j=1}^J u_n^j + e^{it\Delta_g} r_n^J.
	\end{align*}
	The bound~\eqref{ch4:e:palais-smale_contradiction} will be a
	consequence of Proposition~\ref{ch4:p:stability} and the following three 
	assertions:
	\begin{enumerate}
		\item \label{ch4:enum:bdd} $\limsup_{J \to J^*} \limsup_{n \to \infty} 
		\| 
		\tilde{u}_n^J 
		\|_{Y(\mf{R})} \le C(E_c, \delta) < \infty$.
		\item \label{ch4:enum:matching_initial_data} $\lim_{J \to J^*} 
		\limsup_{n 
			\to \infty} \| u_n(0) - 
		\tilde{u}_n^J(0) \|_{\dot{H}^1} = 0$.
		\item \label{ch4:enum:approx_soln} $\lim_{J \to J^*} \limsup_{n \to 
			\infty} 
		\| \nabla [ 
		(i\partial_t + \Delta_g) \tilde{u}_n^J - F(\tilde{u}_n^J) ] 
		\|_{N(\mf{R})} = 0$, where $F(z) = |z|^4 z$.
	\end{enumerate}
	
	\textbf{Proof of claim \eqref{ch4:enum:bdd}}. As the Strichartz estimate 
	and 
	the hypothesis of bounded energy imply that the remainder $e^{it\Delta_g} 
	r_n^J$ is bounded in $Y$, it suffices 
	to show that
	\begin{align*}
	\limsup_{J \to J^*} \limsup_{n \to \infty} \, \Bigl \| \sum_{j=1}^J u_n^j 
	\Bigr \|_{Y(\mf{R})} < \infty.
	\end{align*}
	For each $J$, we have
	\begin{align}
	\label{ch4:e:palais-smale_bdd_eqn1}
	\Bigl \| \sum_{j=1}^J u_n^j \Bigr \|_Y^2 &= \Bigl \| \bigl ( \sum_{j=1}^J 
	\nabla 
	u_n^j \bigr )^2 \Bigr\|_{L^5 L^{\frac{15}{13}}} \le \sum_{j=1}^J \| \nabla 
	u_n^j \|_{L^{10} L^{\frac{30}{13}}}^{2} + c_J \sum_{j \ne k} \| (\nabla 
	u_n^j)(\nabla u_n^k) \|_{L^5 L^{ \frac{15}{13} }}.
	\end{align}
	By Lemma~\eqref{ch4:l:nonlinear_profile_decoupling}, the cross-terms vanish 
	as 
	$n \to \infty$. By the asymptotic additivity of energy, there is some $J_0$ 
	such that $\limsup_{n \to \infty} \| \nabla u_n^j(0) \|_{L^2}$ is 
	smaller than the small-data threshold in Proposition~\ref{ch4:p:lwp} 
	for all $j \ge J_0$. In view of the small-data 
	estimate~\eqref{ch4:e:lwp_spacetime}, for any $J > J_0$ we have
	\begin{align*}
	\limsup_{n \to \infty} \, \Bigl \| \sum_{j=1}^J u_n^j \Bigr \|_Y^2 &\le 
	C_{J_0}(E_c) + \limsup_{n \to \infty} \sum_{j=J_0}^J E(u_n^j) \le C_{J_0} 
	(E_c) + E_c.
	\end{align*}
	
	For future reference, we observe this also proves that for any 
	$\varepsilon > 0$, there exists 
	$J'(\varepsilon, E_c)$ with
	\begin{align}
	\label{ch4:e:palais-smale_vanishing_tails}
	\limsup_{n \to \infty} \, \Bigl \| \sum_{J'\le j \le J} u_n^j \Bigr \|_{Y} 
	< 
	\eta.
	\end{align}
	for all $J$.
	
	\begin{lma}
		\label{ch4:l:nonlinear_profile_decoupling}
		For all $j \ne k$,
		\begin{align*}
		\lim_{n \to \infty} \|u_n^j u_n^k\|_{L^5 L^5} + \| u_n^j \nabla 
		u_n^k\|_{L^5 L^{\frac{15}{8}}} + \| \nabla u_n^j \nabla u_n^k\|_{L^5 
			L^{\frac{15}{13}}} = 0.
		\end{align*}
	\end{lma}
	\begin{proof}[Proof of Lemma~\ref{ch4:l:nonlinear_profile_decoupling}]
		
		The argument is well-known, and we will just illustrate it by 
		estimating the 
		middle term. By Proposition~\ref{ch4:p:emb_profiles}, for each 
		$\varepsilon 
		> 0$ there 
		exist $\psi^j, \psi^k \in C^\infty_0( \mf{R} \times \mf{R}^3)$ such 
		that 
		\begin{align*}
		\| \nabla [ u_n^j(t) - G_n^j \psi^j ( (\lambda_n^j)^{-2} (t-t_n^j) ) ]
		\|_{L^{10} L^{\frac{30}{13}}} +  \| \nabla [ u_n^k(t) - G_n^k \psi^k ( 
		(\lambda_n^k)^{-2} 
		(t-t_n^k) )] \|_{L^{10} L^{\fr{30}{13}}} < \varepsilon
		\end{align*}
		for all $n$ sufficiently large. Letting $v^j_n$, $v^k_n$ denote the 
		compactly supported approximations, we have by H\"{o}lder
		\begin{align*}
		\| u_n^j (\nabla u_n^k)\|_{L^5 L^{\frac{15}{8}}} &\le \| u_n^j - v_n^j 
		\|_{L^{10} L^{10}} \| \nabla u_n^k\|_{L^{10} L^{\frac{30}{13}}} + \| 
		v_n^j \|_{L^{10} L^{10}} \| \nabla (u_n^k - v_n^k) \|_{L^{10} 
			L^{\frac{30}{13}}}\\
		&+ \|v_n^j \nabla v_n^k \|_{L^5 L^{\frac{15}{8}}}.
		\end{align*}
		The last term vanishes due to the pairwise orthogonality 
		of the frames $(\lambda_n^j, t_n^j, x_n^j)$ and $(\lambda_n^k, t_n^k, 
		x_n^k)$. Thus
		\begin{align*}
		\limsup_{n \to \infty} \| u_n^j \nabla u_n^k \|_{L^5 L^{\frac{15}{8}}} 
		&\le C(E_c, \delta) 
		\varepsilon
		\end{align*}
		for any $\varepsilon > 0$.  
	\end{proof}
	
	\textbf{Claim~\eqref{ch4:enum:matching_initial_data}} is immediate.
	
	\textbf{Proof of Claim~\eqref{ch4:enum:approx_soln}}. Write
	\begin{align*}
	(i\partial_t + \Delta_g) \tilde{u}_n^J - F(\tilde{u}_n^J) &= \sum_{j=1}^J 
	F(u_n^j) - F \Bigl( \sum_{j=1}^J u_n^j \Bigr) \\
	&+ F\Bigl( \sum_{j=1}^J u_n^j 
	\Bigr) - F \Bigl( \sum_{j=1}^J u_n^j + e^{it\Delta_g} r_n^J\Bigr),
	\end{align*}
	and expand
	\begin{align*}
	&F \Bigl ( \sum_{j=1}^J u_n^j \Bigr ) - \sum_{j=1}^J F(u_n^j) = 
	\Bigl | \sum_{j=1}^J u_n^j \Bigr|^4 \Bigl ( \sum_{j=1}^J 
	u_n^j \Bigr) - \sum_{j=1}^J 
	|u_n^j|^4 u_n^j\\
	&= \sum_{j=1}^J \Bigl( \Bigl| \sum_{j=1}^J u_n^j \Bigr|^4 - |u_n^j|^4 
	\Bigr) u_n^j\\
	&= \sum_{j=1}^J \sum_{k \ne j} \Bigl( u_n^j u_n^k \int_0^1 G_z\Bigl( 
	\sum_{\ell=1}^J u_n^\ell - \theta u_n^j \Bigr) \, d\theta + 
	u_n^j\overline{u_n^k} \int_0^1 G_{\overline{z}} \Bigl( \sum_{\ell=1}^J 
	u_n^\ell 
	- \theta u_n^j \Bigr) \, d\theta \Bigr), 
	\end{align*}
	where $G(z) = |z|^4$. By the Leibniz rule, H\"{o}lder, and 
	Lemma~\ref{ch4:l:nonlinear_profile_decoupling}, 
	\begin{align*}
	\| F \Bigl ( \sum_{j=1}^J u_n^j \Bigr ) - \sum_{j=1}^J F(u_n^j)\|_{L^2 
		L^{\frac{6}{5}}}  &\le \sum_{j=1}^J \sum_{j \ne k} \|\nabla (u_n^j 
	u_n^k)\|_{L^{5} L^{\frac{15}{8}}} \Bigl\| 
	\int_0^1 G' \Bigl( \sum_{\ell=1}^Ju_n^\ell - \theta u_n^j \Bigr) \, d\theta 
	\Bigr\|_{L^{\frac{10}{3}} L^{\frac{10}{3}}}\\
	&\le c_J \sum_{j=1}^J \sum_{j \ne k}  \| 
	u_n^j \nabla u_n^k\|_{L^5 L^{\frac{15}{8}}}\Bigl( \Bigl\| \sum_{\ell=1}^J 
	u_n^\ell 
	\Bigr \|_{L^{10} L^{10}}^3 + \| u_n^j\|_{L^{10}L^{10}}^3 \Bigr)\\
	&\to 0 \text{ as } n \to \infty.
	\end{align*}
	
	Similarly, write
	\begin{align*}
	F\Bigl( \sum_{j=1}^J u_n^j + e^{it\Delta_g} r_n^J \Bigr) - F \Bigl( 
	\sum_{j=1}^J u_n^j \Bigr) &= \Bigl(\Bigl| \sum_{j=1}^J u_n^j + 
	e^{it\Delta_g}  r_n^J
	\Bigr|^4 - \Bigl| \sum_{j=1}^J u_n^j \Bigr|^4\Bigr) \sum_{j=1}^J u_n^j  \\
	&+ 
	\Bigl| \sum_{j=1}^J u_n^j + e^{it\Delta_g} r_n^J \Bigr|^4 e^{it\Delta_g} 
	r_n^J\\
	&= (I) + (II).
	\end{align*}
	
	First consider $(I)$. As before, writing $G(z) = |z|^4$, we have by the 
	Leibniz rule, H\"{o}lder, and Sobolev embedding
	\begin{align*}
	\Bigl\| \nabla (I)\|_{L^2 L^{\frac{6}{5}}} &\le \Bigl\| \nabla 
	(e^{it\Delta_g}r_n^J ) \int_0^1 G'\Bigl( \sum_{j=1}^J u_n^j + \theta 
	e^{it\Delta_g} r_n^J \Bigr) \sum_j u_n^j  \Bigr\|_{L^2 L^{\frac{6}{5}}} \\
	&+ \Bigl \| (e^{it \Delta_g} r_n^J) \nabla \int_0^1 G'\Bigl (\sum_{j=1}^J 
	u_n^j + \theta e^{it\Delta_g} r_n^J \Bigr) \sum_{j=1}^J u_n^j \Bigr \|_{ 
		L^2 L^{ 
			\frac{6}{5}}}\\
	&\lesssim \| \nabla(e^{it\Delta_g} r_n^J) \Bigl| \sum_{j=1}^J u_n^j 
	\Bigr|^4 
	\Bigr \|_{L^2 L^{\frac{6}{5}}} + \| \nabla(e^{it\Delta_g} r_n^J) 
	\|_{L^{10} L^{\frac{30}{13}}} \| e^{it\Delta_g} r_n^J \|_{L^{10} L^{10}}^3 
	\Bigl \|\sum_{j=1}^J u_n^j \Bigr \|_{L^{10} L^{10}}\\
	&+ \|e^{it\Delta_g} r_n^J \|_{L^{10} L^{10}} ( \| \nabla u_n^J\|_{L^{10} 
		L^{\frac{30}{13}}}^4 + \| \nabla e^{it\Delta_g} r_n^J \|_{L^{10} 
		L^{\frac{30}{13}}}^4).
	\end{align*}
	By~\eqref{ch4:e:palais-smale_vanishing_remainder} and interpolation, 
	$\lim_{J 
		\to J^*} \limsup_{n \to \infty} \| e^{it\Delta_g}r_n^J \|_{L^{10} 
		L^{10}} = 
	0$; therefore, all but the first term are acceptable.
	
	To deal with the first term, we recall that for any $\varepsilon$, there 
	exists by~\eqref{ch4:e:palais-smale_vanishing_tails} a threshold 
	$J'(\varepsilon)$ such that for all $n$ large,
	\begin{align*}
	\Bigl \| \sum_{j=J'}^J u_n^j \Bigr \|_{L^{10} 
		L^{10}} < \varepsilon.
	\end{align*}
	With $\varepsilon$ fixed but arbitrary, this implies that
	\begin{align*}
	&\| \nabla(e^{it\Delta_g} r_n^J ) \Bigl| \sum_{j=1}^J u_n^j \Bigr|^4 \Bigr 
	\|_{L^2 L^{\frac{6}{5}}} \\
	&\le c_{J'} \sum_{j=1}^{J'} \| (u_n^j)^4 \nabla 
	e^{it\Delta_g} r_n^J \|_{L^2 L^{\frac{6}{5}}} + \varepsilon^4 E_c^{1/2}.
	\end{align*}
	It therefore remains to show that
	\begin{align}
	\label{ch4:e:palais-smale_lin-nonlin_interaction}
	\limsup_{J \to J^*} \limsup_{n \to \infty} \| (u_n^j)^4 \nabla 
	e^{it\Delta_g} r_n^J \|_{L^2 L^{\frac{6}{5}}} \lesssim \varepsilon \text{ 
		for each } j \le J'.
	\end{align}
	
	Select $\psi^j \in C^\infty_0(\mf{R} \times \mf{R}^3)$ 
	so that
	$\| u_n^j  - v_n^j \|_{Y} < \varepsilon$, where $v_n^j = G_n^j \psi^j ( 
	(\lambda_n^j)^{-2}(t-t_n^j))$. Then we may replace $u_n^j$ by $v_n^j$ in 
	the above sum since for all $n$ sufficiently large, since
	\begin{align*}
	\| (u_n^j)^4 \nabla e^{it\Delta_g} r_n^J \|_{L^2 L^{\frac{6}{5}}} & \le 
	c\|u_n^j - v_n^j \|_{L^{10} L^{10}} ( \|u_n^j \|_{L^{10} L^{10}}^3 + \| 
	v_n^j\|_{L^{10} L^{10}}^3 ) \| \nabla e^{it\Delta_g} r_n^J \|_{L^{10} 
		L^{\frac{30}{13}}}\\
	&\le C(E_c)\varepsilon.
	\end{align*}
	Let $\chi_n^j$ denote the characteristic function of 
	$\opn{supp}(v_n^j)$. Putting $N_n^j = (\lambda_n^j)^{-1}$, we estimate 
	using H\"{o}lder, Littlewood-Paley theory, and 
	Lemma~\ref{ch4:l:loc_smoothing}
	\begin{align*}
	\| (v_n^j)^4 \nabla e^{it\Delta_g} r_n^J \|_{L^2 L^{\frac{6}{5}}} 
	&\lesssim  
	(N_n^j)^2 \|\chi_n^j \nabla e^{it\Delta_g} P_{\le BN_n^j} r_n^J \|_{L^2 
		L^{\frac{6}{5}}} + (N_n^j)^2 \|\chi_n e^{it\Delta_g} P_{> BN_n^j} r_n^J 
	\|_{L^2 L^{\frac{6}{5}}}\\
	&\lesssim  (N_n^j)^{-1} \| \nabla e^{it\Delta_g} P_{\le BN_n^j} r_n^J 
	\|_{L^\infty L^6} + 
	N_n^j \| \chi_n e^{it\Delta_g} P_{> BN_n^j} r_n^J \|_{L^2 L^2}\\
	&\lesssim B \| e^{it\Delta_g} r_n^J\|_{L^\infty L^6} + B^{-1/2}.
	\end{align*}
	As the remainder vanishes in $L^\infty L^6$ and $B$ is arbitrary, it 
	follows 
	that
	\begin{align*}
	\lim_{J \to J^*} \limsup_{n \to \infty} \| (v_n^j)^4 \nabla 
	e^{it\Delta_g}r_n^J 
	\|_{L^2 L^{\frac{6}{5}}} = 0.
	\end{align*}
	Altogether, we obtain~\eqref{ch4:e:palais-smale_lin-nonlin_interaction}, 
	hence 
	$(I)$ is acceptable.
	
	The contribution of $(II)$ is estimated similarly. By the Leibniz rule,
	\begin{align*}
	\| \nabla (II) \|_{L^2 L^{\frac{6}{5}}} &\lesssim \| e^{it\Delta_g} r_n^J 
	(u_n^J)^3 \nabla u_n^J\|_{L^2 L^{\frac{6}{5}}} + \| (u_n^J)^4 \nabla 
	e^{it\Delta_g} r_n^J \|_{L^2 L^{\frac{6}{5}}}.
	\end{align*}
	The first term is acceptable due to the undifferentiated $e^{it\Delta_g} 
	r_n^J$, while the second term is handled is above. This completes the proof 
	of 
	Claim~\ref{ch4:enum:approx_soln}, and therefore finishes the proof of 
	Lemma~\ref{ch4:l:palais-smale_crit_profile} asserting the existence of a 
	critical 
	profile. Consequently, Proposition~\ref{ch4:p:palais-smale} is proved.
\end{proof}


\section{Scattering for small metric perturbations}
\label{ch4:s:scattering}

In this final section we prove scattering for metrics $g$ with 
$\|g-\delta\|_{C^3} \le 
\varepsilon$ for some $\varepsilon$ depending on the diameter of 
$\opn{supp}(g-\delta)$. 
If the curvature is sufficiently mild, we can adapt 
the one-particle Bourgain-Morawetz inequality 
\cite{bourgain_nls_radial_gwp} for the Euclidean nonlinear Schr\"{o}dinger 
equation to preclude the existence of almost-periodic solutions, which, when 
combined with Corollary~\ref{ch4:c:soliton-like}, yields scattering. We do not attempt to optimize $\varepsilon$ here as we believe the smallness assumption is artificial, but do not know how to prove that at this time.

\begin{prop}
	\label{ch4:p:bourgain_morawetz}
	There exists $\varepsilon > 0$ such that if $\|g-\delta\|_{C^3} \le 
	\varepsilon$, 
	then for any solution $u$ to the 
	nonlinear equation~\eqref{ch4:e:main_eq} and any time interval $I$, 
	\begin{align*}
	\int_I \int_{|x| \le  |I|^{1/2}} \frac{|u|^6}{\langle x \rangle} \, dx dt 
	&\le c |I|^{1/2} E(u).
	\end{align*}
\end{prop}

\begin{proof}
	Let $a = \langle x \rangle$ as in the proof of 
	Lemma~\ref{ch4:l:loc_smoothing}, 
	and 
	write $a_R = a(x) \chi(\tfrac{\cdot}{R})$ where $\chi$ is a smooth cutoff 
	equal 
	to $1$ on the ball $|x| \le 1$ and supported in $|x| \le 2$. Then
	\begin{alignat*}{4}
	\partial a_R &= O(1), &\quad \Delta a_R &=  \Bigl(\frac{2}{\langle x 
	\rangle}  + \frac{1}{\langle x 
		\rangle^3} \Bigr)\mr{1}_{\{ |x| \le R\}} + O( 
	R^{-1} \mr{1}_{ \{|x| \sim R\} } ) \\
	\partial^2 a_R &= \partial^2 a \mr{1}_{\{|x| 
		\le 
		R\}}+ O( R^{-1} \mr{1}_{ \{ |x| \sim R \} } ), &\quad \Delta^2 
		a_R &= 
	-\frac{15}{\langle x \rangle^7}\mr{1}_{\{ |x| \le R \}} + O( R^{-3} 
	\mr{1}_{ \{ 
		|x| \sim R \} }).
	\end{alignat*}
	
	Let $D = d + \Gamma$ denote the covariant derivative, where $\Gamma$ is 
	supported in 
	the unit ball and $\| \Gamma\|_{C^2} = O(\varepsilon)$. It follows that if 
	$\varepsilon$ is sufficiently small, the above formulas continue to hold 
	with 
	the partial derivatives replaced by the covariant derivative $D$ and 
	$\Delta$ 
	by the metric Laplacian 
	$\Delta_g$. Applying the Morawetz identity~\eqref{ch4:e:morawetz} with 
	action 
	$M(t) 
	= \int a_R(x) |u(t,x)|^2 \, dg$, we obtain $|\partial_t M| \le c R \| 
	\nabla 
	u\|_{L^2}^2$ and
	\begin{align*}
	\int_{|x| \le R} \frac{ |u|^6}{ \langle x \rangle} \, dx &\le\partial_t^2 
	M + cR^{-3} \int_{|x| \sim R} |u|^2 \, dx + cR^{-1} \int_{|x| \sim R} 
	|\nabla u|^2 + |u|^6 \, dx\\
	&\le \partial_t^2 M + c R^{-1} E(u).
	\end{align*}
	Setting $R = |I|^{1/2}$ and integrating in time, we obtain
	\begin{align*}
	\int_{I} \int_{|x| \le |I|^{1/2}} \frac{ |u|^6 }{ \langle x \rangle } \, dx 
	dt 
	\le \sup_{t} 2|\partial_t M| + c|I| R^{-1} E(u) \lesssim |I|^{1/2} E(u).
	\end{align*}
\end{proof}

By Corollary~\ref{ch4:c:soliton-like}, if there is a finite energy solution 
to~\eqref{ch4:e:main_eq} that failed to scatter, then there exists a nonzero 
almost-periodic solution $u_c$, i.e. which remains in a precompact subset of 
$\dot{H}^1$.

\begin{cor}
	If $\|g-\delta\|_{C^3} \le \varepsilon$, the equation~\eqref{ch4:e:main_eq} 
	does 
	not 
	admit nonzero almost-periodic solutions. Hence, all finite-energy solutions 
	to~\eqref{ch4:e:main_eq} scatter. 
\end{cor}

\begin{proof}
	Suppose $0 \ne u_c$ is almost-periodic. Then
	there exists $\eta > 0$ and a radius $R$ such that 
	\[
	\|u_c(t) \|_{L^6 ( \{ 
		|x| 
		\le R\} )} \ge \eta \text{ for all } t.
	\]
	For if not, there would exist radii $R_n \to 
	\infty$ and times $t_n$ such that $\| u_c(t_n)\|_{L^6 ( \{ |x| \le R_n 
		\}) } \to 0$. By compactness, it follows that some subsequence of 
	$u_c(t_n)$ 
	converges in $\dot{H}^1$ to $0$. But this yields the contradiction that 
	$E(u_c) = 0$.
	
	We now apply Proposition~\ref{ch4:p:bourgain_morawetz} on time intervals 
	$I$ 
	with 
	$|I|^{1/2} > 
	R$, and deduce
	\begin{align*}
	\eta |I| R^{-1} \le \int_I \int_{|x| \le |I|^{1/2}} \frac{|u_c|^6}{\langle 
	x 
		\rangle } 
	\, dx dt \lesssim |I|^{1/2} E(u_c).
	\end{align*}
	But this yields a contradiction for $I$ sufficiently large.
\end{proof}

\section{Appendix: A semiclassical $L^1 \to L^\infty$ estimate}
\label{s:appendix}

We demonstrate how the techniques from the proof of Proposition~\ref{ch4:p:extinction} yield a refinement of the
Burq-Gerard-Tzvetkov dispersive estimate~\cite{BurqGerardTzvetkov2004}.

\begin{prop}
	Let $g$ be a smooth metric on $\mf{R}^d$ with all derivatives bounded, set 
	$a(x, \xi) = g^{jk} \xi_j 
	\xi_k$ and denote by 
	\begin{align*}
	A(h) = a^w(x, hD) = (2\pi h)^{-d} \int_{\mf{R}^d} e^{\frac{i(x-y)\xi}{h}} a 
	\left( \frac{x+y}{2}, \xi \right) \, d\xi
	\end{align*}
	its semiclassical Weyl quantization. Let $\chi \in C^\infty_0( \mf{R}^d)$ 
	be a frequency cutoff. Then, 
	\begin{itemize}
	\item [(1)] 
	There exists a constant $c > 0$, depending on $g$,
	\begin{align*}
	\| e^{-\frac{it A(h)}{h}} \chi(hD) \|_{L^1(\mf{R}^d) \to L^\infty( 
	\mf{R}^d)} \lesssim |ht|^{-d/2} \text{ for all } |t| \le c.
	\end{align*}
	
	\item [(2)] For each $T > 0$, we have 
	\begin{align*}
		\| e^{-\frac{it A(h)}{h}} \chi(hD)  \|_{L^1(\mf{R}^d) \to 
		L^\infty( \mf{R}^d)} \lesssim_T h^{-d + \frac{1}{2}} \text{ for all } c 
		\le |t| \le T.
	\end{align*}
	\end{itemize}
\end{prop}

\begin{rmk}
 If $g$ is sufficiently flat at infinity (i.e. is a scattering metric), the 
second part follows immediately from the Hassell-Wunsch 
parametrix~\cite{HassellWunsch2005}, which also shows that the exponent $-d + 
\tfr{1}{2} = -\tfr{d}{2} - \tfr{d-1}{2}$ is the best possible in general (the 
$d-1$ in the second fraction is the maximum number of linearly independent 
Jacobi fields along a geodesic that vanish at both endpoints).
\end{rmk}

\begin{proof}
	
	Only the second part is new, as the first part follows essentially from 
the WKB analysis of~\cite{BurqGerardTzvetkov2004}. We use the basic argument for Proposition~\ref{ch4:p:extinction} but do the accounting more carefully. That 
	is, we frequency-localize, decompose into wavepackets, and track their evolution along the geodesic flow.
	
	By introducing a spatial cutoff $\eta(h^{-1}X)$, where $\eta \in 
	C^\infty_0( \mf{R}^d)$, we further assume the initial data $u_h 
(0) = 
	\chi(hD) 
	\eta(h^{-1}X) \phi$ is localized to a minimum uncertainty box in phase 
	space. This will be convenient for the wavepacket analysis below.
	As the domain is $L^1$, the full bound may be recovered via a 
	partition of unity and the triangle inequality (note that the hypotheses 
	are translation-invariant). 
	
	Write $u_h = e^{-\frac{itA(h)}{h}}  \chi(hD) \eta(X) \phi$, where by 
	linearity we normalize $\|\phi\|_{L^1} = 1$; thus $u_h$ 
	solves the equation 
	\begin{align*}
	[hD_t + A(h)] u_h = 0, \ u_h(0) = \chi(hD) \eta(h^{-1} X) \phi.
	\end{align*}
	Select $\chi_1, \chi_2 \in C^\infty_0(\mf{R}^d \setminus \{0\})$ such that 
	$\chi \prec \chi_1 \prec \chi_2$, and define the localized operator $A'(h) 
	= (\chi_2 a)^w(x, hD)$. By the semiclassical functional calculus
	$\| [1 - \chi_1(hD) ] u_h\|_{H^s} \le O(h^\infty) \|u_h \|_{L^2}$ for 
	any $s > 0$, and we deduce that
	\begin{align*}
	[hD_t + A'(h)] u_h = r_h
	\end{align*}
	where $\|r_h\|_{H^s}  \le O(h^\infty) \| u_h\|_{L^2} \le O(h^\infty) \| 
	\phi\|_{L^1}$. By the Duhamel formula and Sobolev embedding, the remainder 
	may safely be ignored, therefore we shall just study the 
	free evolution under $A'(h)$.
	
	To make the notation less cumbersome, write $\phi_h = \chi(hD) \eta( h^{-1} 
	X) \phi$. From the 
	convolution representation of $\chi(hD)$ (and recalling the normalization 
	of $\phi$), we have the 
	pointwise bound
	\begin{align*}
	|\phi_h(x) | \lesssim_N h^{-d} \left\langle \frac{x}{h} \right\rangle^{-N} .
	\end{align*}
	Let 
	\begin{align*}
	T_h \phi_h (x, \xi) = c_d h^{-\frac{3d}{4}} \int e^{\frac{i \xi(x-y)}{h}} 
	e^{-\frac{(x-y)^2}{2h}} \phi_h(y) \, dy = c_d' h^{-\frac{5d}{4}} \int 
	e^{\frac{ix\eta}{h}} e^{- \frac{(\eta - \xi)^2}{2h}} \widehat{\phi_h} \left ( 
	\frac{\eta}{h} \right ) \, d\eta
	\end{align*}
	denote the FBI transform of $\phi_h$. When $|x| \sim 2^k h^{1/2}$, we have
	\begin{align*}
	|T_h \phi_h| &\lesssim h^{-\frac{3d}{4}} \int_{|y| \le \frac{1}{2}|x|} 
	e^{-\frac{(x-y)^2}{2h}} h^{-d} 
	\left\langle \frac{y}{h} \right\rangle^{-N} \, dy + h^{-\frac{3d}{4}} \int_{ |y| \ge 
	\frac{1}{2} |x|} h^{-d} \left\langle \frac{y}{h} \right \rangle^{-N} \, dy\\
	&\lesssim h^{-\frac{3d}{4}} ( e^{-c 2^{2k}} + h^{100d} 2^{-100 d k}).
	\end{align*}
	Similarly, for $|\xi| \gg 1$,
	\begin{align*}
	|T_h \phi_h| &\lesssim h^{-\frac{5d}{4}} \int e^{-\frac{c |\xi|^2}{h}} 
	\chi(\eta) |\widehat{ \eta(h^{-1} X) \phi } (\eta)| \, d\eta\\
	&\lesssim h^{-\frac{5d}{4}} e^{-\frac{c|\xi|^2}{h}}.
	\end{align*}
	
	Choose $R > 0$ so that $\mr{supp} (\chi) \subset \{ |\xi| \le R/4\}$, 
and 
	partition phase space as $T^* \mf{R}^d = B' \cup \bigcup_{k \ge 0} B_k$, 
	where
	\begin{gather*}
	B' = \{ |\xi| > R\}\\
	B_0 = \{ |x| \le h^{1/2}, \ |\xi| \le R \}, \ B_k = \{ 2^{k-1} h^{1/2} < 
	|x| \le 2^{k} h^{1/2}, \ |\xi| \le R\}.
	\end{gather*}
	Decompose
	\begin{align*}
	\phi_h = \int_{B'} T_h \phi_h (x, \xi) \psi^h_{(x, \xi)} \, dx d\xi + 
	\sum_{k \ge 0} \int_{B_k} T_h \phi_h (x, \xi) \psi^h_{ x, \xi} \, dx d\xi,
	\end{align*}
	where $\psi_{x, \xi}^h (y) = 2^{-\frac{d}{2}} \pi^{-\frac{3d}{4}} 
	h^{-\frac{3d}{4}} e^{ \frac{i\xi (y-x)}{h}} e^{ -\frac{(y-x)^2}{2h}}$. By 
	the preceding bounds and Proposition~\ref{ch4:p:wavepacket_sc}, we estimate
	\begin{align*}
	\left| \int_{ B'} T_h \phi_h  e^{-\frac{it A'(h)}{h}} \psi_{x, \xi}^h  \, 
	dx 
	d\xi \right| &\lesssim \sum_k \int_{ |x| \sim 2^k h^{1/2}, \ |\xi| \ge R}  
	h^{-\frac{3d}{4}} h^{-d} (e^{-c 2^{2k}} + 2^{-50dk}) 
e^{-c\frac{|\xi|^2}{h}} 
	\, dx d\xi\\
	&= O(h^\infty).
	\end{align*}
	
	Next, let $0 < \alpha < \tfr{1}{2}$ be a small parameter, and estimate
        \begin{align*}
         \sum_{2^k > h^{-\alpha}} \left| \int_{B_k} T_h \phi_h 
e^{-\fr{itA'(h)}{h}} \psi_{x, \xi}^h \, dx d\xi \right| &\lesssim \sum_{2^k > 
h^{-\alpha}} h^{-\fr{3d}{2}} (e^{-c 2^{2k}} + h^{100d} 2^{-100dk} )\\
&\lesssim h^{10d},
        \end{align*}
so this contribution is acceptable. It remains to bound the terms 
	\begin{align*}
	 \left| \int_{B_k} T_h \phi_h 
e^{-\fr{itA'(h)}{h}} \psi_{x, \xi}^h (y) \, dx d\xi \right| &\lesssim 
\int_{B_k} h^{-\fr{3d}{4}} \left \langle \fr{y - x^t}{h^{1/2}} 
\right\rangle^{-100d} | T_h \phi_h (x, \xi)|,
	\end{align*}
	where $2^{k} h^{1/2} \le h^{1/2 - \alpha} < 1$. We bound the right side by
	\begin{align*}
	\sum_{j} \int_{B_{k, j}} 
	h^{-\frac{3d}{4}}  \Bigl\langle \frac{y - x^t}{h^{1/2}} 
\Bigr\rangle^{-100d} 
	|T_h 
	\phi_h(x, \xi)| \, d\xi dx \le \sum_{2^j \le h^{-\alpha}} + \sum_{2^j > h^{-\alpha}}.
	\end{align*}
	where $B_{k, j} = B_k \cap \{ (x, \xi): 
	2^{j-1} < |x^t - y| \le 2^j h^{1/2} \}$ and $B_{k, 0} = B_k \cap \{ (x, 
\xi): |x^t - y| \le h^{1/2} \}$. The second sum is negligible as before. For 
the first, we invoke the exponential map estimate of 
Lemma~\eqref{ch4:l:exp_preimage_estimate} to bound the $j$th member
by
\begin{align*}
 &h^{-\fr{3d}{4}}(2^{-100jd}) h^{-\fr{3d}{4}} (e^{-c2^{2k}} + h^{100d} 
2^{-100dk}) (2^k h^{1/2})^d (2^j h^{1/2}) \\
&\le h^{-d + \fr{1}{2}} 2^{-99jd} (e^{-c2^{2k}} + h^{100d} 2^{-100dk} ).
\end{align*}
The proof is finished by summing over $j$ and $k$.
\end{proof}

	\bibliographystyle{amsalpha}
    \bibliography{../bibliography}

	\end{document}